\documentclass[11pt]{amsart}

\usepackage{amsfonts, amstext, amsmath, amsthm, amscd, amssymb}
\usepackage{enumerate}
\usepackage{epsfig, color, overpic, psfrag}

\usepackage[
pdfborderstyle={},
pdfborder={0 0 0},
pagebackref]{hyperref}

\renewcommand*{\backref}[1]{}
\renewcommand*{\backrefalt}[4]{
  \ifcase #1 %
   [No citations.]%
  \or
   [#2]%
  \else
   [#2]%
  \fi
}

\let\oldmarginpar\marginpar
\renewcommand\marginpar[1]{\oldmarginpar[\raggedleft\footnotesize #1]%
{\raggedright\footnotesize #1}}

\newcommand{\HH}{{\mathbb{H}}}
\newcommand{\RR}{{\mathbb{R}}}
\newcommand{\ZZ}{{\mathbb{Z}}}
\newcommand{\NN}{{\mathbb{N}}}

\newcommand{\TT}{{\mathbb{T}}}
\newcommand{\ff}{{\mathcal{F}}} 
\newcommand{\pp}{{\mathcal{P}}} 

\renewcommand{\setminus}{{\smallsetminus}}
\newcommand{\bdy}{\partial}
\newcommand{\area}{{\mathrm{area}}}
\newcommand{\abs}[1]{{\left\vert #1 \right\vert}} 
\def\co{\colon\thinspace}

\theoremstyle{plain}
\newtheorem{theorem}{Theorem}[section]
\newtheorem{corollary}[theorem]{Corollary}
\newtheorem{lemma}[theorem]{Lemma}

\newtheorem*{namedtheorem}{\theoremname}
\newcommand{\theoremname}{testing}
\newenvironment{named}[1]{\renewcommand{\theoremname}{#1}\begin{namedtheorem}}{\end{namedtheorem}}

\theoremstyle{definition}
\newtheorem{define}[theorem]{Definition}

\newtheorem{condition}[theorem]{Condition}
\newtheorem{remark}[theorem]{Remark}

\numberwithin{equation}{section}

\begin{document}


\title{Dehn filling and the geometry of unknotting tunnels}
\author{Daryl Cooper}
\author{David Futer}
\author{Jessica S. Purcell}

\address{Department of Mathematics, University of California, Santa Barbara, CA 93106, USA}

\email[]{cooper@math.ucsb.edu}

\address[]{Department of Mathematics, Temple University, Philadelphia,
	PA 19122, USA}

\email[]{dfuter@temple.edu}

\address[]{ Department of Mathematics, Brigham Young University,
Provo, UT 84602, USA}

\email[]{jpurcell@math.byu.edu }

\thanks{{Cooper is supported in part by NSF grants DMS--0706887 and DMS--1207068. \\
\indent Futer is  supported in part by NSF grant DMS--1007221. \\
\indent Purcell is supported in part by NSF grant 
DMS--1007437 and the Alfred P. Sloan Foundation.}}

\thanks{ \today}

\begin{abstract}
Any one--cusped hyperbolic manifold $M$ with an unknotting tunnel
$\tau$ is obtained by Dehn filling a cusp of a two--cusped hyperbolic
manifold.
In the case where $M$ is obtained by ``generic'' Dehn filling, we
prove that $\tau$ is isotopic to a geodesic, and characterize whether
$\tau$ is isotopic to an edge in the canonical decomposition of $M$.
We also give explicit estimates (with additive error only) on the
length of $\tau$ relative to a maximal cusp. These results give
generic answers to three long--standing questions posed by Adams,
Sakuma, and Weeks.

We also construct an explicit sequence of one--tunnel knots in $S^3$,
all of whose unknotting tunnels have length approaching infinity.
\end{abstract}

\maketitle


\section{Introduction}\label{sec:intro}

Let $M$ be a compact orientable $3$--manifold whose boundary consists of tori. An \emph{unknotting tunnel} for $M$ is a properly embedded arc $\tau$ from $\bdy M$ to $\bdy M$, such that $M \setminus \tau$ is a genus--2 handlebody. Not all $3$--manifolds with torus boundary admit an unknotting tunnel; those that do are said to be \emph{one--tunnel} or \emph{tunnel number one}. The above definition immediately implies that every one--tunnel manifold $M$ has one or two boundary components, and has Heegaard genus two (unless it is a solid torus).

For this paper, we investigate one-tunnel manifolds $M$ such that the interior of $M$ carries a complete hyperbolic metric. The geometric study of unknotting tunnels in this setting begins with two foundational papers published in 1995 by Adams \cite{adams:tunnels} and Sakuma and Weeks \cite{sakuma-weeks}. These papers posed three open questions about unknotting tunnels of one-cusped manifolds:

\begin{enumerate}
\item\label{q:geodesic} Is $\tau$ always isotopic to a geodesic?
\item\label{q:length} Is there a universal bound $B$, such that outside a maximal cusp neighborhood in $M$, the geodesic in the homotopy class of $\tau$ is always shorter than $B$?
\item\label{q:canonical} Is $\tau$ isotopic to an edge in the canonical polyhedral decomposition of $M$?
\end{enumerate}

One motivation behind these questions is that for complements of two--bridge knots in $S^3$, the answer to all three questions is ``yes''  \cite{adams-reid, aswy:book}. However, apart from the special family of two--bridge knots, the only progress to date has consisted of selected examples for question \eqref{q:geodesic}, and selected counterexamples to questions \eqref{q:length} and \eqref{q:canonical}.

In this paper, we give detailed answers to all three questions, under the hypothesis that $M$ is obtained by ``generic'' Dehn filling on one cusp of a two--cusped hyperbolic manifold $X$. In this generic setting, the tunnel $\tau$ is indeed isotopic to a geodesic. Generically, this geodesic is quite long, and we provide explicit estimates on the length, with additive error only. Whether or not $\tau$ is isotopic to an edge of the canonical decomposition turns out to depend on the length of an associated tunnel $\sigma \subset X$ (see Figure \ref{fig:filled-tunnel}).

In addition, we construct an explicit sequence of one-tunnel knots $K_n \subset S^3$, such that each $K_n$ has two unknotting tunnels, whose length approaches infinity as $n \to \infty$.

\subsection{Generic Dehn fillings and generic unknotting tunnels}\label{sec:generic}
Let $X$ be a compact orientable $3$--manifold whose boundary consists of  one or more tori, and whose interior is hyperbolic. Let $T$ be one of the tori of $\bdy M$. A \emph{slope} on $T$ is an isotopy class of simple closed curves. 
The \emph{Dehn filling of $X$ along a slope
  $\mu$}, denoted $X(\mu)$, is the manifold obtained by attaching a
solid torus $D^2 \times S^1$ to $T$, so that $\bdy D^2$ is glued to $\mu$. The slope $\mu$ is called the \emph{meridian} of the filling.

\begin{define}\label{def:mer-long}
An embedded, horoball neighborhood of a boundary torus $T \subset X$ is called a \emph{horocusp}, and denoted $H_T$. If we fix such a horocusp $H_T$, the horospherical torus $\bdy H_T$ inherits a Euclidean metric that allows us to measure the length of slopes. In particular, a slope $\mu$ chosen as a meridian for Dehn filling has a well--defined length $\ell(\mu)$, namely the length of a Euclidean geodesic representing $\mu$ on $\bdy H_T$. In a similar way, we define a \emph{longitude} of the Dehn filling to be a simple closed curve on $\bdy H_T$ that intersects $\mu$ once.  We will typically be interested in the shortest longitude, denoted $\lambda$.  In highly symmetric cases, there can be two shortest longitudes (up to isotopy), and we may choose either one.
 Note that once the horocusp is fixed, the lengths $\ell(\mu)$ and $\ell(\lambda)$ are always well-defined.
\end{define}

\begin{define}\label{def:generic}
We say that the Dehn filling along $\mu$ is \emph{generic} if both $\mu$ and $\lambda$ are sufficiently long. Equivalently, the filling is generic if both $\mu$ and $\lambda$ avoid finitely many prohibited slopes on the torus $T$.
\end{define}

If we fix a basis $\langle \alpha, \beta \rangle$ for $H_1(T) \cong \ZZ^2$, then all the possible choices of Dehn filling slope are parametrized by primitive pairs of integers $(p,q) \in \ZZ^2$. In this setting, choosing a generic slope amounts to avoiding finitely many points and finitely many lines in $\RR^2$.

The term \emph{generic} can be justified as follows. Let $\ff$ be the Farey graph, whose vertices are slopes on the torus $T$, and whose edges correspond to slopes that intersect once. For each prohibited value of $\lambda$, the values of $\mu$ that have $\lambda$ as a longitude lie on a circle of radius $1$ in $\ff$, centered at $\lambda$. Thus prohibiting finitely many values of $\mu$ and $\lambda$ amounts to prohibiting $\mu$ from lying in finitely many closed balls of radius $1$. Since the Farey graph $\ff$ has infinite diameter, almost all choices of $\mu$ will avoid the prohibited sets, and are indeed generic. In particular, a random walk in $\ff$ will land on a generic slope with probability approaching $1$.

 We would like to argue that the unknotting tunnels created by generic Dehn filling (as in Definition \ref{def:generic}) are also ``generic,'' in an appropriate sense. This must be done with some care, as there are multiple reasonable notions of genericity \cite{dunfield-thurston:generic, lustig-moriah}.

Suppose $X$ is a manifold with cusps $T$ and $T'$, and an unknotting tunnel $\sigma$. Then the Heegaard surface
associated to $\sigma$ (namely, the boundary $\Sigma$ of a regular
neighborhood of $T \cup T' \cup \sigma$) cuts $X$ into a genus--2
handlebody $C$ and a compression body $C'$. (See Definition
\ref{def:compression-body} for details.) One way to obtain a ``random'' two--cusped $3$--manifold of this type is to glue $C$ to $C'$ via a random walk in the generators of
the mapping class group $\textrm{Mod}(\Sigma)$; see \cite{dunfield-thurston:generic}. In this context, Maher has shown that with probability approaching $1$, a random walk in $\textrm{Mod}(\Sigma)$ gives a Heegaard splitting of high distance  \cite{maher:random-splitting}. Then, the work of Scharlemann and Tomova implies that $\Sigma$ will be the only genus--2 Heegaard surface of $X$ \cite{scharlemann-tomova}. The same conclusion holds under more measure--theoretic notions of genericity: see Lustig and Moriah  \cite{lustig-moriah}. Thus, in two reasonable senses, one can say that generic
one--tunnel manifolds have a unique unknotting tunnel and a unique
minimal--genus Heegaard surface.

This generic uniqueness is preserved after Dehn filling. Results of Moriah--Rubinstein \cite{moriah-rubinstein} and Rieck--Sedgwick \cite{rieck-sedgwick:persistence} imply that for a generic Dehn filling slope $\mu$ on $T$, every genus--$2$ Heegaard surface $\Sigma$ of $X(\mu)$ comes from a Heegaard surface of $X$. See Theorem \ref{thm:heegaard-correspondence} for details, including quantified hypotheses. Thus, by the previous paragraph, a generic Dehn filling of a generic one-tunnel manifold $X$ will have exactly one unknotting tunnel.

The unknotting tunnels created by Dehn filling have a natural visual description, summarized in Figure \ref{fig:filled-tunnel}. If $X$ is a manifold with cusps $T$ and $T'$, any
unknotting tunnel $\sigma$ of $X$ must connect $T$ with $T'$. Then there will be a new tunnel $\tau \subset M = X(\mu)$ that starts at $T'$ and runs along $\sigma$, followed by a longitude of the filling, followed by backtracking along $\sigma$. It is easy to verify that $M \setminus \tau \cong X \setminus \sigma$, hence a handlebody. We call $\tau$ the tunnel of $M = X(\mu)$ that is \emph{associated to $\sigma$}. See Definition \ref{def:associated} and Theorem \ref{thm:tunnel-correspondence} for a much more detailed description of associated tunnels.

\begin{figure}
\begin{overpic}{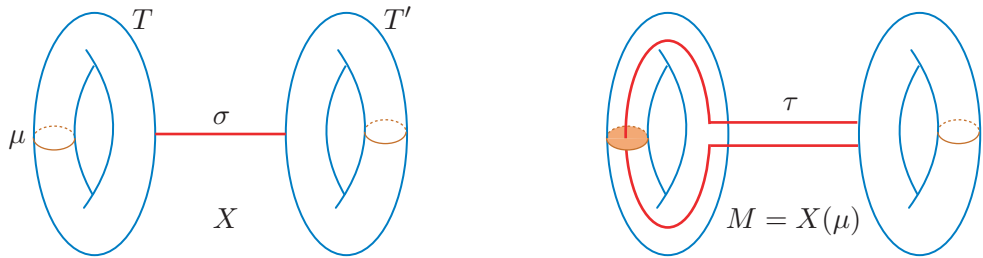}
\put(10.5,24){$T$}
\put(37,24){$T'$}
\put(19,14){$\sigma$}
\put(-2.5,12){$\mu$}
\put(19,3){$X$}
\put(73,3){$M=X(\mu)$}
\put(79,15.5){$\tau$}
\end{overpic}
\caption{A schematic picture of unknotting tunnels under Dehn filling. Left: $\sigma$ is a tunnel for a 2-cusped manifold $X$. Right: the associated tunnel $\tau$ of the 1-cusped manifold $M = X(\mu)$.}
\label{fig:filled-tunnel}
\end{figure}

The theorems stated below describe the geometry of all associated tunnels created by generic Dehn filling. In particular, we answer questions \eqref{q:geodesic}, \eqref{q:length}, and \eqref{q:canonical} for these tunnels.

\subsection{Tunnels isotopic to geodesics}
Question \eqref{q:geodesic} has the following history. 
Adams showed, using a symmetry argument, that every unknotting tunnel of a two--cusped hyperbolic manifold is isotopic to a geodesic \cite{adams:tunnels};  see  also Lemma \ref{lemma:hyperelliptic}. Shortly after, Adams and Reid extended these symmetry arguments to prove that the upper and lower tunnels of a 2--bridge knot are isotopic to geodesics  \cite{adams-reid}. However, since the late 1990s, there has been only minimal progress on question \eqref{q:geodesic}. As a negative result, Futer showed that the symmetry arguments of Adams and Reid do not apply to any knots in $S^3$ besides 2--bridge knots \cite{futer:tunnel}.

For generic Dehn fillings, we provide a positive answer to question \eqref{q:geodesic}:

\begin{theorem}\label{thm:geodesic}
Let $X$ be an orientable hyperbolic $3$--manifold that has two cusps and tunnel number one. Choose a generic filling slope $\mu$ on one cusp of $X$, and let $\tau \subset X(\mu)$ be an unknotting tunnel associated to a tunnel $\sigma \subset X$. Then $\tau$ is isotopic to a geodesic in the hyperbolic metric on $X(\mu)$.
\end{theorem}

Note that by the discussion above (more precisely, by Theorem \ref{thm:tunnel-correspondence} and Remark \ref{rem:generic-unique}),
a two--cusped manifold $X$ constructed by a random Heegaard splitting 
will have a unique unknotting tunnel $\sigma$, and its generic Dehn filling $M=X(\mu)$ will have a unique tunnel $\tau$ associated to $\sigma$. Thus, generically, $M$ has exactly one unknotting tunnel, which is isotopic to a geodesic.

\subsection{The length of unknotting tunnels}

To measure the length of an unknotting tunnel $\tau$, one first needs to choose horospherical cusp neighborhoods in the ambient manifold $M$. If $M$ has one boundary torus, there is a canonical choice of \emph{maximal cusp}, namely the closure of the largest embedded horocusp  about $\bdy M$. If $M$ has two (or more) cusps, then the maximal cusp neighborhood will depend on the order in which the cusps are expanded. Once the cusp neighborhoods are fixed, we define the \emph{length} of an unknotting tunnel $\tau$ to be the length of the geodesic in the homotopy class of $\tau$, outside the given horocusps in $M$.

If a hyperbolic one-tunnel manifold $M$ has two boundary tori, Adams \cite{adams:tunnels} showed that there always exist disjoint horocusps about these tori such that every unknotting tunnel of $M$ has  length at most $\ln(4)$. In a subsequent preprint \cite{adams:waist2}, he improved the upper bound to $\tfrac{7}{4} \ln(2)$. These universal upper bounds for two-cusped manifolds prompted a wide belief that the  unknotting tunnels of one-cusped manifolds also have universally bounded length.

In a recent paper \cite{clp:long-tunnels}, Cooper, Lackenby and Purcell showed that in fact, the answer to question \eqref{q:length} is ``no'': there exist one-cusped hyperbolic manifolds whose unknotting tunnels are arbitrarily long outside a maximal cusp.  However, the examples in
\cite{clp:long-tunnels} either were non-constructive, or could not be
complements of knots in $S^3$.  The authors asked whether there exist knots in $S^3$ with
arbitrarily long unknotting tunnels, and whether such examples can be explicitly described.

In this paper, we show that generically, unknotting tunnels are very long. In fact, we compute the length of $\tau$, up to additive error only.

\begin{theorem}\label{thm:length-estimate}
Let $X$ be an orientable hyperbolic $3$--manifold that has two cusps and an unknotting tunnel $\sigma$. Let $T$ be one boundary torus of $X$. Then, for all but finitely many choices of a Dehn filling slope $\mu$, the unknotting tunnel $\tau$ of $X(\mu)$ associated to $\sigma$ satisfies
$$2 \ln \ell(\lambda) - 6 \: < \: \ell (\tau) \: < \: 2 \ln \ell(\lambda) + 5,$$
where $\lambda$ is the shortest longitude of the Dehn filling, $ \ell(\lambda)$ is the length of $\lambda$ on a maximal cusp corresponding to $T$, and $\ell(\tau)$ is the length of the geodesic in the homotopy class of $\tau$.
\end{theorem}

We also apply this result to knots in $S^3$: see Theorem \ref{thm:knot-long-tunnel} below.

\subsection{Canonical geodesics}

In a one-cusped hyperbolic manifold $M$, let $H$ be a closed, embedded horocusp.  Then the Ford--Voronoi domain $F$ is defined to be the set of all points in $M$ that have a unique shortest path to $H$. This is an open set in $M$, canonically determined by the geometry of $M$ and, in particular, independent of the chosen size of $H$. The complement $L = M \setminus F$ is a compact $2$--complex, called the \emph{cut locus}. The combinatorial dual to $L$ is an ideal polyhedral decomposition $\pp$ of $M$; the $n$--cells of $\pp$ are in bijective correspondence with the $(3-n)$--cells of $L$. This is called the \emph{Epstein--Penner decomposition} or \emph{canonical polyhedral decomposition} of $M$. 

For one-cusped manifolds, the canonical decomposition $\pp$ is a complete invariant of the homeomorphism type of $M$. For multi-cusped manifolds, one may perform the same construction, although the combinatorics of the resulting polyhedral decomposition may depend on the relative volumes of the horocusps $H_1, 
\ldots, H_k$.

\begin{define}\label{def:canonical}
Let $M$ be a one-cusped hyperbolic manifold. We say that an arc $\tau$ from cusp to cusp (in practice, an unknotting tunnel) is \emph{canonical} if $\tau$ is isotopic to an edge of the canonical polyhedral decomposition $\pp$.
\end{define}

Sakuma and Weeks performed an extensive study of the  triangulations of 2--bridge knot complements \cite{sakuma-weeks}. Using experimental evidence from SnapPea \cite{weeks:snappea}, they conjectured that these triangulations are canonical -- a conjecture subsequently proved by Akiyoshi, Sakuma, Wada, and Yamashita  \cite{aswy:book}. In addition, Sakuma and Weeks observed that the unknotting tunnels of 2--bridge knots are always isotopic to edges of this triangulation, which led them to conjecture that  all
unknotting tunnels of hyperbolic manifolds are canonical \cite{sakuma-weeks}.

This conjecture was disproved in 2005, with a single counterexample constructed by Heath and Song  \cite{heath-song}. For the $(-2,3,7)$ pretzel knot $K$, they showed that $S^3 \setminus K$ has four unknotting tunnels but only three edges in its canonical triangulation. Although this example settled question \eqref{q:canonical} in the negative, it did not shed light on the broader question of what properties of an unknotting tunnel imply that it is, or is not, canonical.

In the context of generic Dehn filling, this broader question has the following answer.

\begin{theorem}\label{thm:canonicity}
Let $X$ be a two-cusped, orientable hyperbolic $3$--manifold in which there is a unique
shortest geodesic arc between the two cusps. Choose a generic Dehn filling slope $\mu$ on a cusp of $X$. Then, for each unknotting tunnel $\sigma \subset X$, the tunnel $\tau \subset X(\mu)$ associated to $\sigma$ will be canonical 
 if and only if $\sigma$ is the shortest geodesic between the two cusps of $X$.
\end{theorem}

If there are several shortest geodesics between the two cusps of $X$, then Theorem \ref{thm:canonicity} does not  give any information. However, the existence of a unique shortest geodesic can also be regarded as a ``generic'' property of hyperbolic manifolds.

In practice, both alternatives of Theorem \ref{thm:canonicity} are
quite common. Using this theorem, one may easily construct infinite
families of manifolds whose tunnels are canonical, as well as infinite
families that have non-canonical tunnels. See Theorem \ref{thm:2bridge} for one such construction.


\subsection{Knots with long tunnels in $S^3$}

Theorems \ref{thm:geodesic}, \ref{thm:length-estimate}, and \ref{thm:canonicity} can be applied to construct explicit families of knots in $S^3$, whose unknotting tunnels have interesting properties.

\begin{theorem}\label{thm:knot-long-tunnel}
There is a sequence $K_n$ of hyperbolic knots in $S^3$, such that each $K_n$ has exactly two unknotting tunnels. Each unknotting tunnel $\tau_n$ of $K_n$ is isotopic to a canonical geodesic, whose length is
$$2 n \ln \left( \tfrac{1 + \sqrt{5}}{2} \right) - 5 \: < \: \ell (\tau_n) \: < \: 2 n \ln \left( \tfrac{1 + \sqrt{5}}{2} \right) + 6.$$
\end{theorem}

The sequence of knots $K_n$ is explicitly described in Section \ref{sec:example}. See Figure \ref{fig:knot-diagram} for a preview.

\subsection{Organization of the paper}
This paper is organized as follows. In Section \ref{sec:setup}, we fill in the details of a number of definitions and theorems that were mentioned above. In Theorem \ref{thm:heegaard-correspondence}, we describe the effect of Dehn filling on Heegaard surfaces, adding quantified hypotheses to a theorem of Moriah--Rubinstein
\cite{moriah-rubinstein} and Rieck--Sedgwick \cite{rieck-sedgwick:persistence}. In Theorem \ref{drilling-thm}, we recall the drilling and filling theorems of Hodgson--Kerckhoff \cite{hk:univ-bounds} and
Brock--Bromberg \cite{brock-bromberg:density}, which allow precise bilipschitz estimates on the change in geometry during Dehn filling. This bilipschitz control will be used in all the geometric estimates that follow.

In Section \ref{sec:length}, we use the work of Adams on unknotting tunnels of two-cusped manifolds \cite{adams:tunnels}, combined with Theorem \ref{drilling-thm}, to prove Theorem \ref{thm:length-estimate}. More precisely, we prove two-sided estimates on the length of the geodesic $g_\tau$ in the homotopy class of a tunnel $\tau$, without yet knowing that $\tau$ is isotopic to $g_\tau$. Several quantitative estimates from Section \ref{sec:length} will be used in Section \ref{sec:geodesic} to show that the tunnel $\tau$ is isotopic to a geodesic, establishing Theorem \ref{thm:geodesic}.

In Section \ref{sec:canonical}, we prove Theorem \ref{thm:canonicity}, which relates the canonicity of a tunnel $\tau \subset X(\mu)$ to the length of its associated tunnel $\sigma \subset X$. The argument in this section relies on the recent work of Gu\'eritaud and Schleimer \cite{gueritaud-schleimer}, and also uses the length estimates of Section \ref{sec:length}. As an application, we construct an infinite family of one-cusped manifolds, each of which has one canonical and one non-canonical tunnel.

In Sections \ref{sec:knots} and \ref{sec:example}, we construct knots in $S^3$ whose unknotting tunnels are arbitrarily long. This construction has two flavors. The argument in Section \ref{sec:knots} is quick and direct, but requires making non-explicit ``generic'' choices. The argument in Section \ref{sec:example} is completely explicit, and gives the precise quantitative estimate of Theorem \ref{thm:knot-long-tunnel}. The cost of this entirely explicit construction is that the argument of Section \ref{sec:example} is longer, and requires rigorous computer assistance from the programs Regina \cite{burton:regina} and SnapPy \cite{weeks:snappea}.

\subsection{Acknowledgements}

We thank Ken Bromberg and Aaron Magid for clarifying a number of points about the drilling and filling theorems. We thank Saul Schleimer for numerous helpful conversations about Heegaard splittings and normal surface theory, and for permitting us to use Figure \ref{fig:almost-normal}. We are also grateful to David Bachman, Fran\c{c}ois Gu\'eritaud, and Yoav Moriah for explaining results that we needed in the paper.

\section{Geometric setup}\label{sec:setup}

The goal of this section is to review and synthesize several past
results. We recall the work of Moriah--Rubinstein
\cite{moriah-rubinstein} and Rieck--Segwick
\cite{rieck-sedgwick:persistence} on Heegaard splittings under Dehn
filling, the work of Hodgson--Kerckhoff \cite{hk:univ-bounds} and
Brock--Bromberg \cite{brock-bromberg:density} on the change in
geometry under Dehn filling, and the work of Adams on unknotting
tunnels of two-cusped manifolds \cite{adams:tunnels}. Then, we
synthesize these results in Theorem \ref{thm:tunnel-correspondence},
which explains the $1$--$1$ correspondence between genus--$2$ Heegaard splittings of a two-cusped manifold $X$ and the unknotting tunnels of any
generic Dehn filling $X(\mu)$.

\subsection{Heegaard splittings under Dehn filling}

Although the goal of this paper is to study unknotting tunnels of
one-cusped hyperbolic manifolds, this study will require a slightly
more general setup.

\begin{define}\label{def:compression-body}
A \emph{compression body} $C$ is a $3$--manifold with boundary,
constructed as follows. Start with a genus--$g$ surface
$\Sigma$. Thicken $\Sigma$ to $\Sigma \times [0,1]$, and attach some
number (at least one, at most $g$) of non-parallel $2$--handles to
$\Sigma \times \{0 \}$. If, after attaching $2$--handles, any
component of the boundary becomes a $2$--sphere, cap it off with a
$3$--ball.

The \emph{positive boundary} of $C$ is the boundary component $ \bdy_+
C = \Sigma \times \{1 \}$, untouched during the construction. The
\emph{negative boundary} is $ \bdy_- C = \bdy C \setminus \bdy_+
C$. When the negative boundary is empty, the compression body $C$ is a
genus--$g$ handlebody.

Given a compact orientable $3$--manifold $M$ and a genus--$g$ surface
$\Sigma \subset M$, we say that $\Sigma$ is a \emph{Heegaard splitting
  surface} of $M$ if $\Sigma$ cuts $M$ into compression bodies $C_1$
and $C_2$, such that $\Sigma = \bdy_+ C_1 = \bdy_+ C_2$.
\end{define}

\begin{define}\label{def:core-tunnel}
Let $C$ be a compression body whose positive boundary $ \bdy_+ C$ has
genus $2$. Then $\bdy_- C$ is a disjoint union of at most two tori. If
$\bdy_- C = \emptyset$, then $C$ is a handlebody. If $\bdy_- C \neq
\emptyset$, then $C$ can be constructed by adding exactly one $2$--handle
to $\Sigma \times \{0 \}$. Define the \emph{core tunnel} of $C$ to
be an arc $\sigma$ dual to this $2$--handle. It is well--known that
the core tunnel is unique up to isotopy.
\end{define}

If a genus--$2$ surface $\Sigma$ is a Heegaard surface for $X$, where
$\bdy X$ consists of tori, then there are at most two such tori on
each side of $\Sigma$. If one component of $M \setminus \Sigma$ is a
handlebody while the other has non-empty negative boundary, the core
tunnel $\sigma$ is an unknotting tunnel for $X$.

Suppose that, as above, $X$ is a $3$--manifold with toroidal boundary
and a genus--$2$ Heegaard splitting surface $\Sigma$. If we perform a
Dehn filling along one of the boundary tori of $X$, it is easy to
check that $\Sigma$ remains a Heegaard surface for the filled manifold
$X(\mu)$. (The meridian disk of a solid torus can be thought of as a
$2$--handle added to the negative boundary of a compression body
$C$. This creates a $2$--sphere boundary component, and filling it in
amounts to adding the rest of the solid torus.) In this setting, a
core tunnel $\sigma$ of a pre-filling compression body gives rise to a
core tunnel $\tau$ of the after--filling compression body, exactly as
in Figure \ref{fig:filled-tunnel}.

Moriah and Rubinstein \cite{moriah-rubinstein} and Rieck and Sedgwick
\cite{rieck-sedgwick:persistence} showed that generically, there is a
$1$--$1$ correspondence between minimum-genus Heegaard splittings of
$X$ and those of $X(\mu)$.  More recently, Futer and Purcell found a way
to quantify the hypotheses in their theorem \cite{futer-purcell:heegaard}. Here is what the result says in
genus $2$.

\begin{theorem}\label{thm:heegaard-correspondence}
Let $X$ be an orientable hyperbolic $3$--manifold with one or more
cusps and Heegaard genus $2$. Let $T$ be one boundary torus of
$X$. Choose a Dehn filling slope $\mu$ on $T$, such that $\ell(\mu) >
6\pi$ and the shortest longitude $\lambda$ for $\mu$ has length
$\ell(\lambda) > 6$. Then
\begin{enumerate}[$(a)$]
\item\label{i:hyperbolic} $X(\mu)$ is a hyperbolic manifold of Heegaard genus $2$.
\item\label{i:core-in-sigma} For every genus--$2$ Heegaard surface $\Sigma$ of $X(\mu)$, the core curve $\gamma$ of the Dehn filling solid torus is  isotopic into $\Sigma$.
\item\label{i:sigma-heegaard} Once $\gamma$ is isotoped into one of the compression bodies separated by $\Sigma$, the surface $\Sigma$ becomes a Heegaard surface of $X = X(\mu) \setminus \gamma$.
\end{enumerate}
\end{theorem}

\begin{proof}
If $\ell(\mu) > 2\pi$, the
$2\pi$--Theorem of Gromov and Thurston implies
that the filled manifold $X(\mu)$ admits a negatively curved metric  \cite{bleiler-hodgson}. (See also Futer, Kalfagianni, and Purcell \cite[Theorem 2.1]{fkp:volume} for an explicit construction, with curvature estimates.)
Since $X(\mu)$ is negatively curved, its Heegaard genus must be at
least 2.  But a Heegaard surface $\Sigma$ of $X$ is also a Heegaard
surface for all fillings, hence $X(\mu)$ must have Heegaard genus
exactly 2.
By geometrization, $X(\mu)$ must also admit a hyperbolic
metric, proving conclusion \eqref{i:hyperbolic}.

Conclusions \eqref{i:core-in-sigma} and \eqref{i:sigma-heegaard} are a restatement of \cite[Theorem 1.1]{futer-purcell:heegaard}.
\end{proof}

\subsection{Geometric estimates}\label{subsec:drill-fill}

We will repeatedly need to bound the amount of change of geometry
under Dehn filling.  To do so, we use a version of the drilling
theorem of Brock and Bromberg \cite{brock-bromberg:density}. Before
stating the theorem, we recall several definitions.

\begin{define}\label{def:thick-thin}
Given $\epsilon >0$ and a hyperbolic $3$--manifold $M$, the
\emph{$\epsilon$--thin part} $M_{< \epsilon}$ of $M$ is the set of all
points in $M$ whose injectivity radius is less than
$\epsilon/2$. Equivalently, $M_{< \epsilon}$ is the set of all points
that lie on a non-trivial closed curve of length less than $\epsilon$.

A given $\epsilon >0$ is called a \emph{Margulis number} for $M$ if
each component of $M_{< \epsilon}$ has abelian fundamental group. In
this case, the $\epsilon$--thin part $M_{< \epsilon}$ is a disjoint
union of horocusps and tubular neighborhoods
about geodesics. For a particular cusp $T$, we let $\TT_\epsilon(T)$
denote the component of $M_{< \epsilon}$ corresponding to
$T$. Similarly, if $\gamma$ is a geodesic of length less than
$\epsilon$, a tubular neighborhood $\TT_\epsilon(\gamma)$ is a
component of $M_{< \epsilon}$.
\end{define}

The Margulis lemma states that there is a positive number $\epsilon$
that serves as a Margulis number for every hyperbolic $3$--manifold
\cite[Chapter D]{benedetti-petronio}. The greatest such $\epsilon$,
denoted $\epsilon_3$, is called the ($3$--dimensional) \emph{Margulis
  constant}.

The best available estimate on the Margulis constant is $\epsilon_3
\geq 0.104$, due to Meyerhoff
\cite{meyerhoff:first-estimate}. However, under additional hypotheses
there are stronger estimates on Margulis numbers. For example, Culler
and Shalen recently showed \cite{culler-shalen:margulis-numbers} that
every cusped hyperbolic $3$--manifold has a Margulis number at least
0.292. See also Shalen \cite[Proposition
  2.3]{shalen:margulis-numbers}.

\begin{define}\label{def:normalized-length}
Let $T$ be a Euclidean torus, and let $g$ be a closed geodesic on
$T$. The \emph{normalized length} of $g$ is defined to be
\begin{equation}\label{eq:normalized-length}
L(g) = \ell(g) / \sqrt{\area(T)}.
\end{equation}
The normalized length $L(\mu)$ of a slope $\mu$ on $T$ is defined in
the same way, via the normalized length of a geodesic representative
of $\mu$.

Note that equation \eqref{eq:normalized-length} is
scaling--invariant. Hence, if $T$ is a cusp torus in $M$, the
normalized length of a slope on $T$ does not depend on the choice of
horospherical torus.
\end{define}

%
%
%
%

We can now state a version of Brock and Bromberg's drilling theorem.

\begin{theorem}[Drilling theorem]  \label{drilling-thm} 
Let $X$ be a hyperbolic $3$--manifold with one or more cusps, and let
$T$ be a cusp torus of $X$.  Choose any $J>1$ and any $\epsilon > 0$
that is a Margulis number for each hyperbolic filling along $T$. Then
there is some $K = K(J, \epsilon) \geq 4 \sqrt{2} \cdot \pi$ such that
every slope $\mu$ on $T$ with normalized length $L(\mu) \geq K$
satisfies the following:
\begin{enumerate}[$(a)$]
\item\label{item:filling-exists} $X(\mu)$ is a hyperbolic
  $3$--manifold, obtainable from $X$ by a cone deformation.
\item\label{item:core-length} The core curve $\gamma$ of the added
  solid torus is a geodesic satisfying
  $$\ell(\gamma) \leq \frac{2\pi}{L(\mu)^2 - 4(2\pi)^2} .$$
\item\label{item:j-bilip} There is a $J$--bilipschitz diffeomorphism 
  $$\phi: X \setminus \TT_\epsilon(T) \to X(\mu) \setminus \TT_\epsilon(\gamma).$$
\item\label{item:level-preserve} $\phi$ is level--preserving on any
  remaining cusps of $X$, mapping horospherical tori to horospherical tori. In particular, if $T' \neq T$ is a different cusp,
  then
$$\phi( \bdy  \TT_\epsilon(T') ) =  \bdy  \TT_\epsilon(\phi(T') ).$$
\end{enumerate}
\end{theorem} 

In the setting of finite--volume manifolds, conclusions
\eqref{item:filling-exists} and \eqref{item:core-length} are due to
Hodgson and Kerckhoff \cite{hk:univ-bounds}. Conclusions
\eqref{item:j-bilip} and \eqref{item:level-preserve} are due to Brock
and Bromberg \cite[Theorem 6.2 and Lemma
  6.17]{brock-bromberg:density}, who construct the reverse
diffeomorphism $\phi^{-1}$ under the hypothesis that the core curve
$\gamma$ is sufficiently short. (When $\mu$ is sufficiently long, this
hypothesis will be satisfied by \eqref{item:core-length}.) See Magid
\cite[Section 4]{magid:deformation} for a unified treatment of all
four statements in this version of the theorem.

Conclusions \eqref{item:j-bilip} and \eqref{item:level-preserve} can
be fruitfully combined, as follows. Let $\TT_\epsilon(X)$ denote the
union of all the $\epsilon$--thin cusp neighborhoods of $X$, i.e. the
cusp components of $X_{<\epsilon}$. Similarly, let
$\TT_\epsilon(X(\mu))$ denote the union of $\TT_\epsilon(\gamma)$ and
all the $\epsilon$--thin cusp neighborhoods in $X(\mu)$. Then, by the
Drilling theorem, we have
\begin{equation}\label{eq:bilip-thick}
\phi: X \setminus \TT_\epsilon(X) \to X(\mu) \setminus \TT_\epsilon(X(\mu)).
  \end{equation}
a $J$--bilipschitz diffeomorphism between compact manifolds.

\begin{remark}\label{rem:thick}
In the forthcoming arguments, particularly in Sections \ref{sec:length} and \ref{sec:geodesic}, we will refer to $X \setminus \TT_\epsilon(X)$ and $X(\mu) \setminus \TT_\epsilon(X(\mu))$ as the \emph{thick parts} of $X$ and $X(\mu)$, respectively. This usage is somewhat abusive, for instance since the manifold $X(\mu)$ may contain other $\epsilon$--short geodesics besides $\gamma$. However, any ``extra'' $\epsilon$--thin regions of $X$ or $X(\mu)$ will not affect the arguments in any way, rendering the abuse relatively harmless.
\end{remark}

In practice, we will always use Theorem \ref{drilling-thm} in the
setting where $X$ is a finite--volume hyperbolic manifold with two or
more cusps. Thus, since every filled manifold $M = X(\mu)$ has one or
more cusps remaining, Culler and Shalen's recent theorem
\cite{culler-shalen:margulis-numbers} implies that $\epsilon = 0.292$
is a Margulis number for both $X$ and every hyperbolic $X(\mu)$. Unless 
stated otherwise (e.g.\ in the proof of Theorem \ref{thm:canonicity}), we will always work with the value $\epsilon = 0.29$.

One immediate consequence of Theorem \ref{drilling-thm} is the
following fact, which we will use repeatedly.

\begin{lemma}\label{lemma:bilip-lengths}
Let $\alpha$ be a homotopically essential closed curve in $X$, or an
essential arc whose endpoints are on $\bdy \TT_\epsilon(X)$. Let
$g_\alpha$ be a shortest geodesic in the free homotopy class of
$\alpha$ in $X \setminus \TT_\epsilon(X)$, where the endpoints of
$\alpha$ are allowed to slide along $\bdy \TT_\epsilon(X)$ if $\alpha$ is
an arc.

Choose $J >1$ and a slope $\mu$ that satisfies Theorem
\ref{drilling-thm}.  Let $\bar{\alpha} = \phi(\alpha)$ be a curve or
arc in $X(\mu)$, where $\phi$ is the $J$-bilipschitz diffeomorphism guaranteed by Theorem \ref{drilling-thm}. Let $\bar{g}_\alpha$ be a shortest geodesic in
the free homotopy class of $\bar{\alpha}$ in $X(\mu) \setminus \TT_\epsilon X(\mu)$. Then
\begin{equation}\label{eq:j-related}
\frac{1}{J} \cdot \ell (g_\alpha) \: \leq \: \ell(\bar{g}_\alpha) \:
\leq \: J \cdot \ell (g_\alpha).
\end{equation} 
\end{lemma}

When the inequality \eqref{eq:j-related} holds, we will say that the lengths of $g_\alpha$ and
$\bar{g}_\alpha$ are \emph{$J$--related}.

The reason for the non-unique terminology ``a shortest geodesic'' is
that $\alpha$ can be, for instance, a peripheral curve in $\bdy
\TT_\epsilon$. In this case, $g_\alpha$ is a Euclidean geodesic.

Note that there is no reason to expect that the $J$--bilipschitz
diffeomorphism $\phi$ maps the geodesic $g_\alpha$ to the geodesic
$\bar{g}_\alpha$. Nevertheless, the estimate on the geodesic lengths
still holds.

\begin{proof}[Proof of Lemma \ref{lemma:bilip-lengths}]
To prove the upper bound on $\ell(\bar{g}_\alpha)$, suppose that
$\alpha = g_\alpha$ is already geodesic. Then, by Theorem
\ref{drilling-thm}, the arc $\bar{\alpha} = \phi(g_\alpha)$ has length
at most $J \cdot \ell(g_\alpha)$. Since the geodesic $\bar{g}_\alpha$
can be no longer than $\bar{\alpha}$, the same upper bound applies:
$$\ell(\bar{g}_\alpha) \: \leq \: J \cdot \ell (g_\alpha).$$

By the same argument, starting with the geodesic $\bar{g}_\alpha$ and
applying the $J$--bilipschitz diffeomorphism $\phi^{-1}$, we obtain
$$  \ell (g_\alpha) \: \leq \: J \cdot \ell(\bar{g}_\alpha),$$
which is exactly what is needed to complete the proof.
\end{proof}

\subsection{Geometric estimates and core tunnels}\label{subsec:core-estimates}

In the remaining sections of the paper, we will apply Theorems
\ref{thm:heegaard-correspondence} and \ref{drilling-thm} to
unknotting tunnels in cusped
hyperbolic $3$--manifolds, and more generally, to the core tunnels as in Definition \ref{def:core-tunnel}. In order to do this, we need information
about the core tunnels before filling.

\begin{lemma}\label{lemma:hyperelliptic}
Suppose  $X$ is a finite--volume hyperbolic $3$--manifold, with a genus--$2$ Heegaard surface $\Sigma$. Suppose that $\sigma \subset X$ is the
core tunnel for a compression body of $X \setminus \Sigma$, whose endpoints are on distinct cusp tori $T$ and
$T'$.  Then
\begin{enumerate}[$(a)$]
\item $X$ admits a hyper-elliptic involution $\psi$, which preserves
  $\Sigma$ up to isotopy.
\item The hyperbolic isometry isotopic to $\psi$ fixes a hyperbolic
  geodesic isotopic to $\sigma$.
\end{enumerate}
\end{lemma}

\begin{proof}
When $\sigma$ is an unknotting tunnel, this statement is due to Adams
\cite[Lemma 4.6]{adams:tunnels}, and his proof carries through
verbatim to core tunnels that connect distinct cusps. We recall the
argument briefly. Each compression body $C_i$ in the complement of
$\Sigma$ admits a hyper-elliptic involution, and the restriction of
these involutions to $\Sigma$ is unique up to isotopy
\cite{bleiler-moriah}. Thus the involutions of $C_1$ and $C_2$ can be
glued together to obtain an involution $\psi$ on $X$ preserving
$\Sigma$ setwise.

The hyper-elliptic involution $\psi$, restricted to $\Sigma$,
preserves the isotopy class of every simple closed curve; separating
curves on $\Sigma$ are preserved with orientation. The compression
disk of $C_1$ dual to $\sigma$ separates $T$ from $T'$, hence its
boundary is preserved with orientation (up to isotopy). As a result,
$\psi$ can be chosen to fix $\sigma$ pointwise.

By Mostow--Prasad rigidity, $\psi$ is homotopic to a hyperbolic
isometry. This order--$2$ isometry of $X$ lifts to an elliptic
isometry of $\HH^3$ that preserves the endpoints of a lift of
$\sigma$, hence preserves the geodesic $\widetilde{g}_\sigma$ connecting these
endpoints. Now, by the work of Waldhausen \cite{waldhausen:suff-large} and Tollefson \cite{tollefson:involution}, two homotopic involutions of $X$ are connected by a continuous path of involutions. Thus the 
fixed--point set of $\psi$ is isotopic to the fixed--point set of the isometry, hence $\sigma$ is isotopic to the geodesic $g_\sigma$ in its homotopy class.
\end{proof}

In the remainder of the paper, we will assume that every core tunnel
$\sigma$ connecting distinct cusps of $X$ is already a geodesic. We
will be studying the behavior of this geodesic in the compact,
thick part $X \setminus \TT_\epsilon(X)$. See Figure
\ref{fig:unfilled-notation} for two lifts of this geodesic to $\HH^3$.

We may use the geodesic $\sigma$ to carefully construct an arc $\tau$
that will become an associated tunnel in a Dehn filling of $X$.  The
point of the following construction is to make Figure
\ref{fig:filled-tunnel} precise.  In the introduction, we stated that
an associated tunnel in the filled manifold runs from the cusp, along
$\sigma$, then once around a longitude, then back along $\sigma$ to
the cusp.  However, this arc as described is not embedded.  In the
following definition, we push the new tunnel off $\sigma$ carefully to
ensure the result is embedded.  We then prove the claim from the
introduction that this arc becomes an unknotting tunnel under Dehn
filling.

\begin{define}\label{def:associated}
Let $X$ be a be an orientable hyperbolic $3$--manifold that has two
cusps (denoted $T$ and $K$), and tunnel number one. Let $\sigma$ be an
unknotting tunnel of $X$, isotoped to be a geodesic. Let $\mu$ be a
Dehn filling slope on $T$, and $\lambda$ be a longitude for
$\mu$. Choose any $\epsilon > 0 $ that is a Margulis number for $X$
(for example, $\epsilon = 0.29$). Let $\lambda_\epsilon$ be a closed
curve representing $\lambda$ on the horospherical torus $\bdy
\TT_\epsilon(T)$, which passes through the endpoint of $\sigma$ on
$\bdy \TT_\epsilon(T)$.

Let $Q$ be an embedded quadrilateral contained in a tubular
neighborhood of $\sigma$, whose top side is on $\lambda_\epsilon$ and
whose bottom side is on $\bdy \TT_\epsilon(K)$, and whose remaining
sides, call them $s_1$ and $s_2$, run parallel to $\sigma$ on the
boundary of the tubular neighborhood.

We define the \emph{tunnel arc associated to $\sigma$ and $\mu$},
denoted $\tau(\sigma, \mu)$, to be the embedded arc $s_1 \cup (
\lambda_\epsilon \setminus Q) \cup s_2$. This three--part arc is
sketched in the right panel of Figure \ref{fig:filled-tunnel}. If
$\sigma$ is oriented from $K$ to $T$, then $\tau(\sigma, \mu)$ is
homotopic to $\sigma \cdot \lambda_\epsilon \cdot \sigma^{-1}$.
\end{define}

The hyperbolic geodesic $\sigma \subset X$, the Euclidean geodesic
$\lambda_\epsilon \subset X$, and the corresponding geodesics
$\overline{\sigma}, \overline{\lambda}_\epsilon \subset X(\mu)$ are
depicted in Figures \ref{fig:unfilled-notation} and
\ref{fig:filled-schematic}.

As the name suggests, the tunnel arc $\tau(\sigma, \mu)$ will become
an unknotting tunnel in $X(\mu)$.

\begin{theorem}\label{thm:tunnel-correspondence}
Let $X$ be an orientable hyperbolic $3$--manifold that has two cusps
(denoted $T$ and $K$), and tunnel number one. Let $\sigma$ be an
unknotting tunnel for $X$. Choose $\epsilon = 0.29$ and $J > 1$, and
let $\mu$ be any Dehn filling slope on $T$ that is sufficiently long
for Theorem \ref{drilling-thm} to ensure a $J$--bilipschitz
diffeomorphism $\phi: X \setminus \TT_\epsilon(X) \to X(\mu) \setminus
\TT_\epsilon(X(\mu)).$ Then
\begin{enumerate}[$(a)$]

\item\label{item:new-tunnel} If $\tau(\sigma, \mu)$ is the tunnel arc
  associated to $\sigma$ and $\mu$, as in Definition
  \ref{def:associated}, then $\bar{\tau}(\sigma, \mu) = \phi(\tau(\sigma, \mu))$ is an
  unknotting tunnel of $X(\mu)$.

\item\label{item:new-gives-old} Suppose, in addition,  that $\ell(\mu) > 6 \pi$ and
  $\ell(\lambda) > 6$ on a maximal cusp about $T$.
Then every  unknotting tunnel of $X(\mu)$ corresponds to a genus--$2$ Heegaard surface $\Sigma \subset X$.

\item\label{item:tunnel-correspondence} Suppose that $\ell(\mu) > 6 \pi$, that 
  $\ell(\lambda) > 6$, and that both cusps of $X$ lie on the same side of every genus--$2$ Heegaard surface $\Sigma \subset X$.
Then every
  unknotting tunnel of $X(\mu)$ is isotopic to $\bar{\tau}(\sigma, \mu) = \phi(\tau(\sigma,
  \mu))$ for some unknotting tunnel $\sigma$ of $X$.
\end{enumerate}
\end{theorem}

\begin{figure}
	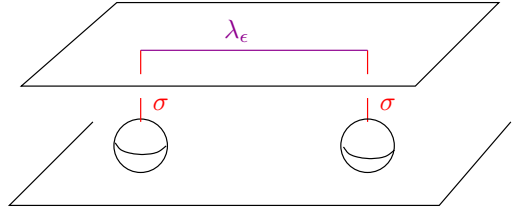
  \caption{$\sigma$ is the geodesic unknotting tunnel in the unfilled manifold, $X$. $\lambda_\epsilon$ is a geodesic representative of the longitude $\lambda$ in the boundary of the $\epsilon$--thin cusp neighborhood. 
  Picture in the universal cover.
  }
  \label{fig:unfilled-notation}
\end{figure}

\begin{figure}
	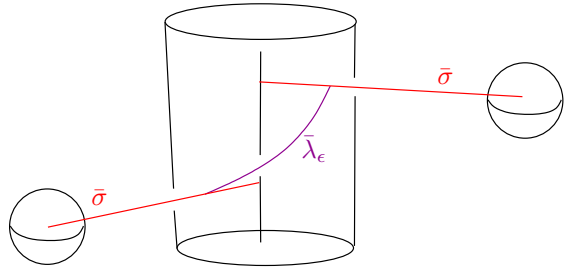
  \caption{
  $\bar{\sigma}$ is the geodesic in the homotopy class of the
    image of $\sigma$ in the filled manifold $X(\mu)$.  The curve
    $\bar{\lambda}_\epsilon$ is the shortest curve along the
    $\epsilon$--Margulis tube between points where $\bar{\sigma}$
    meets the tube on its boundary. Picture in the universal cover.}
  \label{fig:filled-schematic}
\end{figure}

\begin{proof}
Let $\tau = \phi(\tau(\sigma, \mu)) \subset X(\mu)$. We will prove
that $\tau$ is an unknotting tunnel for $X(\mu)$ by showing that
$X(\mu) \setminus \tau$ is homeomorphic to $X \setminus \sigma$, which
is a handlebody by hypothesis.

Let $\lambda_\epsilon$, $Q$, $s_1$, and $s_2$ be as in Definition
\ref{def:associated}. Then observe that $\phi(\lambda_\epsilon)$ is a
longitude for the solid torus $V$ added during Dehn filling, hence is
isotopic to the core curve $\gamma$ of $V$. As a result, $X(\mu)
\setminus \phi(\lambda_\epsilon) \cong X(\mu) \setminus \gamma \cong
X.$ Similarly, $s_1$ is isotopic to $\sigma$. Thus
$$X(\mu) \setminus \phi(s_1 \cup \lambda_\epsilon ) \: \cong \: \phi(
X \setminus s_1) \: \cong \: X \setminus s_1 \: \cong \: X \setminus
\sigma,$$
and $X \setminus \sigma$ is a genus--$2$ handlebody. Finally, note that the arc on top
of $Q$, namely $(\lambda_\epsilon \cap Q)$, can be replaced with $s_2$
without altering the complement. This replacement can be accomplished
by continuously sliding one endpoint of $(\lambda_\epsilon \cap Q)$
along $s_1$, turning the ``eyeglass'' $\lambda_\epsilon \cup \sigma$
into the embedded arc $\tau(\sigma, \mu)$. Thus
$$X(\mu) \setminus \phi \left(s_1 \cup ( \lambda_\epsilon \setminus Q)
\cup s_2 \right) \: \cong \: X(\mu) \setminus \phi(s_1 \cup
\lambda_\epsilon) \: \cong \: X \setminus \sigma,$$
proving \eqref{item:new-tunnel}.

Statement \eqref{item:new-gives-old} follows from Theorem
\ref{thm:heegaard-correspondence}. Let $\tau$ be an unknotting tunnel
of $X(\mu)$, and let $\Sigma \subset X(\mu)$ be the Heegaard surface
associated to $\tau$. By Theorem \ref{thm:heegaard-correspondence},
the core curve $\gamma$ is isotopic into $\Sigma$. Furthermore, isotoping $\gamma$ off $\Sigma$, into one of the pieces separated by $\Sigma$, turns $\Sigma$ into a Heegaard surface for $X
= X(\mu) \setminus \gamma$. 

To prove \eqref{item:tunnel-correspondence}, let $\Sigma \subset X$ be the Heegaard surface guaranteed by \eqref{item:new-gives-old}. By hypothesis, both cusps of $X$ must lie on the same side of $\Sigma$. Thus $\Sigma \subset X$ has a handlebody on one side and a
compression body on the other side. Hence, the core tunnel $\sigma$ of
the compression body in $X \setminus \Sigma$ is an unknotting tunnel
for $X$.

It remains to check that $\tau$ is isotopic to $\tau(\sigma, \mu)$ as
in Definition \ref{def:associated}. This is true because the Heegaard
surface defined by $\tau(\sigma, \mu)$ is the boundary of a regular
neighborhood of $\lambda_\epsilon \cup Q \cup \bdy \TT_\epsilon(K)$,
which is the same Heegaard surface $\Sigma$ defined by $\tau$. Thus,
since $\tau$ and $\tau(\sigma, \mu)$ are core tunnels for the same
compression body in $X(\mu) \setminus \Sigma$, they must be isotopic.
\end{proof}

\begin{remark}
The construction in Definition \ref{def:associated} involved numerous
choices. There are many longitudes for
$\mu$, many representatives of $\lambda$ on $\bdy
\TT_\epsilon(T)$, and many choices for the quadrilateral $Q$ (some of
which are twisted). The argument above implies that all of these
choices are immaterial: up to isotopy in $X(\mu)$, they all produce
the same unknotting tunnel.
\end{remark}


\begin{remark}\label{rem:generic-unique}
As we mentioned in Section \ref{sec:generic}, the work of Lustig--Moriah \cite{lustig-moriah}, Maher
\cite{maher:random-splitting}, and Scharlemann--Tomova
\cite{scharlemann-tomova} severely restricts the ``generic''
possibilities for $\Sigma$.  More precisely, suppose that the
two--cusped manifold $X$ is constructed by gluing a genus--$2$
handlebody $C$ to a compression body $C'$ via some mapping class $\varphi \in \textrm{Mod}(\Sigma)$. If $\varphi$ is chosen by a random walk in the generators of $\textrm{Mod}(\Sigma)$, Maher showed that with probability approaching $1$, the Heegaard splitting has curve
complex distance $d(\Sigma) \geq 5$: see \cite[Theorem
  1.1]{maher:random-splitting}. Similarly, Lustig and Moriah showed that Heegaard splittings satisfying $d(\Sigma) \geq 5$ are generic in the sense of Lebesgue measure on the projective measured lamination space $\mathcal{PML}(\Sigma)$: see  \cite{lustig-moriah}. In either case, once we know that $d(\Sigma) \geq 5$, a result of 
Scharlemann and Tomova implies that $\Sigma$
is the unique minimal--genus Heegaard surface of $X$ \cite[Corollary
  on p. 594]{scharlemann-tomova}.  
 
   Thus, for a generic Dehn filling,
Theorem \ref{thm:tunnel-correspondence} implies the filled manifold
$X(\mu)$ has a unique unknotting tunnel $\tau$, associated to the
tunnel $\sigma$ of $X$.
\end{remark}

\section{The length of unknotting tunnels}\label{sec:length}

The main goal of this section is to write down a proof of Theorem \ref{thm:length-estimate}, which estimates the length of an unknotting tunnel $\tau \subset X(\mu)$ up to additive error. Unfortunately, the clean statement of Theorem \ref{thm:length-estimate} relies on a number of technical estimates about various related lengths in $X$ and $X(\mu)$. We collect these technical estimates in Sections \ref{sec:waist} and \ref{sec:related-lengths}. Then, in Section \ref{sec:length-completion}, we complete the proof of Theorem \ref{thm:length-estimate}.

\subsection{Length and waist size}\label{sec:waist}

As above, let $\sigma$ be an unknotting tunnel of a two-cusped manifold $X$. We know that $\sigma$ is a geodesic arc that runs between the cusps about $K$ and $T$. The first step toward estimating the length of an unknotting tunnel $\tau(\mu, \sigma)$ of $X(\mu)$ is estimating the length of $\sigma$ itself.
The length of $\sigma$ turns out to be closely related to the notion of  waist size, defined and explored by Adams \cite{adams:waist2, adams:waist}. 

\begin{define}\label{def:waist}
Let $H$ be a horocusp in a hyperbolic $3$--manifold $M$ (see Definition \ref{def:mer-long}). Then the \emph{waist size} $w(H)$ is defined to be the length of the shortest non-trivial curve on $\bdy H$. This shortest curve is necessarily a Euclidean geodesic on $\bdy H$.
\end{define}

Adams proved the following statements about the waist size of a two--cusped manifold $X$:

\begin{enumerate}[(A)]

\item\label{fact:adams-length} Given any choice of disjointly embedded horocusps $H_K$ and $H_T$, such that the smaller of the two waist sizes is $w$, the length of an unknotting tunnel $\sigma$ relative to $H_K$ and $H_T$ is $\ell(\sigma) < \ln(4) - 2 \ln(w)$. This is \cite[Theorem 4.4]{adams:tunnels}.

\smallskip

\item\label{fact:adams-waist} If this choice of cusp neighborhoods is \emph{maximal}, in the sense that neither of $H_K$ or $H_T$ can be expanded while keeping them disjointly embedded, then each of $H_K$ and $H_T$, has waist size at least $1$. This universal estimate is  \cite[Lemma 2.4]{adams:waist}.

\end{enumerate}

Facts \eqref{fact:adams-length} and \eqref{fact:adams-waist} have the following consequence.

\begin{lemma}  \label{lemma:waist}
In the two-cusped, tunnel number one manifold $X$, let $N_K$ be a maximal neighborhood about cusp $K$, expanded until it bumps into itself.  Then the waist size of $N_K$ is $1 \leq w(N_K) < 4$. The same estimate holds for the other cusp of $X$.
\end{lemma}

\begin{proof}
Let $H_K^{1}$ be a cusp neighborhood about $K$ whose waist size is exactly $1$. By fact \eqref{fact:adams-waist}, this neighborhood is contained in $N_K$, therefore embedded. Similarly, let $H_T^{1}$ be a horocusp about $T$ whose waist size is exactly $1$. 

We claim that $N_K$ is disjoint from $H_T^{1}$. This is because a maximal choice of neighborhoods can be obtained as follows: expand $K$ until it bumps into itself, obtaining $N_K$. Then, expand $T$ until it bumps into either itself or $K$; in either case, the resulting horocusp about $T$ will have waist size at least $1$, hence contains $H_T^{1}$. Therefore, $H_T^{1}$ is disjoint from $N_K$.

Next, we claim that the horospherical tori $\bdy   N_K$ and $\bdy H_K^1$ are at hyperbolic distance
\begin{equation}\label{eq:shrinkage-dist}
d( \bdy   N_K, \bdy H_K^1) \: < \: \ln 4.
\end{equation}
Here is why. On the one hand, the length of $\sigma$ relative to $N_K$
and $H_T^1$ is at least $0$, since these cusp neighborhoods are
disjoint. On the other hand, the length of $\sigma$ relative to
$H_K^1$ and $H_T^1$ is less than $\ln 4$, by fact
\eqref{fact:adams-length}. The difference between these lengths is
exactly the hyperbolic distance $d( \bdy N_K, \bdy H_K^1)$, which must
be less than $\ln 4$. Similarly, $d( \bdy N_T, \bdy H_T^1) < \ln 4.$

Consider the waist size of $N_K$. This is at least $1$ by fact \eqref{fact:adams-waist}. Also, since $w(H_K^1) = 1$ and waist size grows exponentially with hyperbolic distance, \eqref{eq:shrinkage-dist} implies that  $w(N_K) < 4$.
\end{proof}

Facts \eqref{fact:adams-length} and \eqref{fact:adams-waist}  also allow us to estimate the length of $\sigma$ in the thick part of $X$.

\begin{lemma}\label{lemma:sigma-length} 
In the two-cusped manifold $X$, let $N_K$ be a maximal horocusp about $K$, expanded until it bumps into itself. Let $N_T$ be a maximal horocusp about $T$, expanded until it bumps into itself. For $\epsilon = 0.29$, let $\TT_\epsilon(K)$ and $\TT_\epsilon(T)$, respectively, be the $\epsilon$--thin neighborhoods of those cusps. (See Definition \ref{def:thick-thin}.)

Then, in the thick portion of $X$, the length of $\sigma$ relative to $\TT_\epsilon(K)$ and $\TT_\epsilon(T)$ satisfies
\begin{equation}\label{eq:sigma-thick}
2.46 \: < \: \ell(\sigma_\epsilon)  \: < \:  3.86 .
\end{equation}
Relative to the (possibly overlapping) maximal cusps $N_K$ and $N_T$, the length of $\sigma$ is
\begin{equation}\label{eq:sigma-length}
  - \,  \ln 4 \: < \: \ell(\sigma_{\max}) \: < \:   \ln 4.
  \end{equation}
Here, we are using the convention that the length of $\sigma$ in $X \setminus (N_K \cup N_T)$ counts positively, and the length of $\sigma$ in $N_K \cap N_T$ counts negatively.
\end{lemma}

The length convention in \eqref{eq:sigma-length} is natural, in the
following sense. If a horoball is expanded by distance $d$, the length
of a geodesic running perpendicularly into that horoball decreases by
distance $d$. This natural convention requires negative lengths for
overlapping horoballs.

\begin{proof}
As in Lemma \ref{lemma:waist}, let $H_K^1$ and $H_T^1$ be cusp neighborhoods about $K$ and $T$, respectively, whose waist sizes are exactly $1$. By Adams' fact \eqref{fact:adams-waist}, these horocusps are disjointly embedded in $X$. Thus the length of an unknotting tunnel $\sigma$ relative to these horocusps is at least $0$. On the other hand, by \eqref{fact:adams-length}, the length $\sigma$ relative to these horocusps is less than $\ln 4$.

Now, consider what happens when we replace $H_K^1$ by $\TT_\epsilon(K)$ and $H_T^1$ by $\TT_\epsilon(T)$. By 
Lemma \ref{lemma:13-ideal} in the Appendix, the waist size of $\bdy \TT_\epsilon$ is 
\begin{equation}\label{eq:epsilon-waist}
w(\TT_\epsilon(K)) =  w(\TT_\epsilon(T)) = 2 \sinh(0.145) = 0.29101...
\end{equation}
Because waist size grows exponentially with length, the length of $\sigma$ will increase by a distance of $-\ln(2 \sinh 0.145)$ as $H_K^1$ is replaced by $\TT_\epsilon(K)$. Replacing $H_T^1$ by $\TT_\epsilon(T)$ has the same effect. Thus the length of $\sigma$ relative to $\TT_\epsilon(K)$ and $\TT_\epsilon(T)$ satisfies
$$
2.468... = - 2 \ln (2 \sinh 0.145) \: < \: \ell(\sigma_\epsilon)  \: < \: \ln(4) - 2 \ln (2 \sinh 0.145) = 3.855... .
$$

To prove \eqref{eq:sigma-length}, we begin with disjoint cusp neighborhoods $N_T$ and $H_K^{1}$. By facts \eqref{fact:adams-length} and \eqref{fact:adams-waist}, the length of $\sigma$ relative to these disjoint horocusps is at least $0$ and less than $\ln 4$. As we replace $H_K^{1}$ by the larger cusp neighborhood $N_K$, the length of $\sigma$ can only become smaller, hence is still bounded above by $\ln 4$. In fact, as we replace $H_K^{1}$ by $N_K$, the length of $\sigma$ will decrease by  precisely $d( \bdy   N_K, \bdy H_K^1)$, which is less than $\ln 4$ by equation \eqref{eq:shrinkage-dist}.  Thus $\ell(\sigma)$ is bounded below by  $-\ln 4$.
\end{proof}

\subsection{Estimating a few related quantities}\label{sec:related-lengths}
The next several lemmas involve comparisons between certain geometric measurements in $X$ and those of $X(\mu)$. 

\begin{condition}\label{condition:running}
For the remainder of this section, we set $\epsilon = 0.29$ and $J = 1.1$. We also assume throughout that the Dehn filling slope $\mu$ on $T$ is long enough for Theorem \ref{drilling-thm} to guarantee a $J$--bilipschitz diffeomorphism $\phi : X \setminus \TT_\epsilon \to X(\mu)\setminus  \TT_\epsilon $.
\end{condition}

\begin{lemma}\label{lemma:s-x}
Assume that $\mu$ satisfies Condition \ref{condition:running}. In the two-cusped manifold $X$, let 
$$x := d(\bdy N(T), \bdy \TT_\epsilon(T)),$$
where $N(T)$ is the maximal horocusp about $T$ and $\TT_\epsilon(T)$ is the $\epsilon$--thin horocusp about $T$. In the filled manifold $X(\mu)$, let
$$s := d(\bdy N_\mu(K), \bdy \TT_\epsilon(\gamma)),$$
where $N_\mu(K) \subset X(\mu)$ is the maximal horocusp about the remaining cusp $K$, and $\TT_\epsilon(\gamma)$ is the $\epsilon$--thin Margulis tube. Then 
\begin{equation}\label{eq:x-bound}
x < 2.621.
\end{equation}
Furthermore, $s$ and $x$ are equal up to additive error:
\begin{equation}\label{eq:s-x}
-2 \: < \: s-x \: < \:  2.02.
\end{equation}
\end{lemma}

\begin{proof}
Let us label several more lengths in $X$ and $X(\mu)$. In the unfilled manifold $X$, define 
$$y \: := \:  d(\bdy N(K), \bdy \TT_\epsilon(K)).$$
In the filled manifold $X(\mu)$, define
$$t := d(\bdy N_\mu(K), \bdy \TT_\epsilon(K)),$$
the distance between an $\epsilon$--sized cusp and a maximal cusp about $K$. These definitions are depicted in Figure \ref{fig:setup}.

\begin{figure}
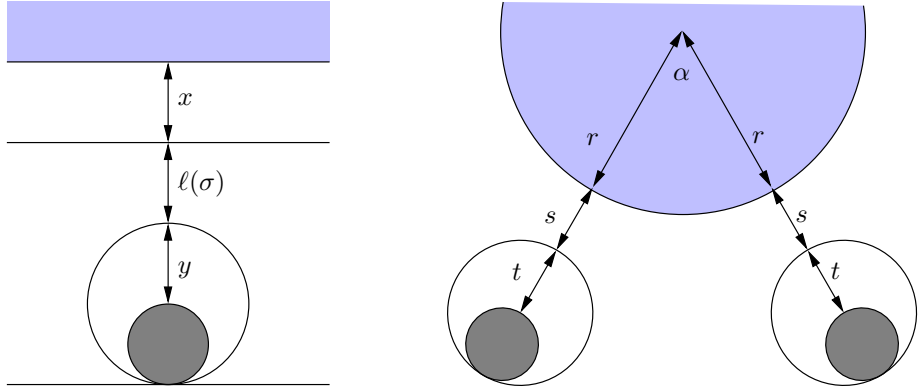

  \input{figures/len-notation-X.pstex_t} \hspace{.5in}
  \input{figures/len-notation-Xu.pstex_t}
  \caption{Notation for Section \ref{sec:length}. The $\epsilon$--thin parts of the two manifolds are shaded. The $J$--bilipschitz diffeomorphism of Theorem \ref{drilling-thm} maps the unshaded area on the left to the unshaded area on the right.}
  \label{fig:setup}
\end{figure}

Note that the waist sizes of $N(K)$ and  $ \TT_\epsilon(K))$, as well as of  $N(T)$ and  $ \TT_\epsilon(T))$, are bounded by Lemma \ref{lemma:waist} and equation \eqref{eq:epsilon-waist}. Thus 
\begin{equation}\label{eq:y-estimate}
y \: = \:  d(\bdy N(K), \bdy \TT_\epsilon(K)) \: < \: \ln \left( \tfrac{4}{2 \sinh 0.145} \right) \: = \: 2.6206 ... \, ,
\end{equation}
and similarly for $x$. This proves \eqref{eq:x-bound}.

In the terminology of Lemma \ref{lemma:sigma-length}, we now have
$$\ell(\sigma_\epsilon) = x + y + \ell(\sigma_{\max}).$$ 
Similarly, the geodesic $\bar{\sigma}_\epsilon$ in the homotopy class of $\phi(\sigma_\epsilon)$ has length
$$\ell(\bar{\sigma}_\epsilon) = s+t.$$ 
By Lemma \ref{lemma:bilip-lengths}, the lengths of $\sigma_\epsilon$ and $\bar{\sigma}_\epsilon$ are $J$--related, for $J=1.1$:
\begin{equation}\label{eq:xyd}
\tfrac{10}{11} \, (\ell(\sigma_{\max}) +x+y)  \:\leq \: s + t \:\leq \: \tfrac{11}{10} \, (\ell(\sigma_{\max}) +x+ y) \\
\end{equation}

Observe that the shortest geodesic $h \subset X$  from $\TT_\epsilon(K)$ back to $\TT_\epsilon(K)$ has length exactly $2y$. This is because an expanding horocusp about $K$ will become maximal, and bump into itself, precisely at the midpoint of a shortest geodesic. Similarly, the shortest geodesic $\bar{h} \subset X(\mu)$ from $\TT_\epsilon(K)$ back to $\TT_\epsilon(K)$ has length exactly $2t$. Thus the lengths of $h$ and $\bar{h}$ are also $J$--related:
\begin{equation}\label{eq:yt}
\tfrac{10}{11} \,  y \: \leq\:  t \: \leq \: \tfrac{11}{10} \, y.
\end{equation}
(Note estimate \eqref{eq:yt} will be true even if $\bar{h}$ is in a different homotopy class from $\phi(h)$, by applying the assumptions that $h$ and $\bar{h}$ are both \emph{shortest}, as in the proof of Lemma \ref{lemma:bilip-lengths}.)

We are now ready to prove the upper and lower bounds of \eqref{eq:s-x}. By equation \eqref{eq:xyd},
$$  \renewcommand{\arraystretch}{1.8}
\begin{array}{r c c c l}
(\ell(\sigma_{\max}) +x+ y) -  \tfrac{1}{11} \, \ell(\sigma_{\epsilon})
  & \leq &  s + t  & \leq & 
 (\ell(\sigma_{\max}) +x+ y) +  \tfrac{1}{10} \, \ell(\sigma_{\epsilon}) \\
\underbrace { (y-t) }_{ \mathrm{use \: \eqref{eq:yt}} } + \ell(\sigma_{\max}) -  \tfrac{1}{11} \, \ell(\sigma_{\epsilon})
& \leq & s-x & \leq &
\underbrace { (y-t) }_{ \mathrm{use \: \eqref{eq:yt}} } + \ell(\sigma_{\max}) + \tfrac{1}{10} \, \ell(\sigma_{\epsilon}) \\
\underbrace{ \left( y -\tfrac{11}{10} \, y \right)  }_{ \mathrm{use \: \eqref{eq:y-estimate}} } + \underbrace{ \ell(\sigma_{\max}) }_{ \mathrm{use \: \eqref{eq:sigma-length}} } - \underbrace{  \tfrac{1}{11} \, \ell(\sigma_{\epsilon}) }_{ \mathrm{use \: \eqref{eq:sigma-thick}} }
& \leq & s-x & \leq &
\underbrace{ \left( y - \tfrac{10}{11} \,  y \right) }_{ \mathrm{use \: \eqref{eq:y-estimate}} } + \underbrace{ \ell(\sigma_{\max}) }_{ \mathrm{use \: \eqref{eq:sigma-length}} } + \underbrace{ \tfrac{1}{10} \, \ell(\sigma_{\epsilon}) }_{ \mathrm{use \: \eqref{eq:sigma-thick}} } \\
\underbrace{ -\tfrac{1}{10} \, (2.621) - \ln 4 - \tfrac{1}{11} \, (3.86) }_{ = \, -1.9993 ...}
& < & s-x & < &
\underbrace{ \tfrac{1}{11} \, (2.621) + \ln 4 + \tfrac{1}{10} \, (3.86)  }_{ = \, 2.0105 ...}
\end{array}
$$
Therefore, $ -2 < s-x < 2.02$.
\end{proof}

Let $\gamma \subset X(\mu)$ be the geodesic core of the solid torus $V$ added during Dehn filling. Let $\bar{\sigma}$ be the geodesic from $\TT_\epsilon(K)$ to $\gamma$ that contains $\bar{\sigma}_\epsilon$ and extends into the thin part of $X(\mu)$ all the way to $\gamma$. There is an arc $\rho_0$ that follows $\bar{\sigma}$ to the core $\gamma$ and runs along $\gamma$ for half the length of $\gamma$, and a similar arc $\rho'_0$ that follows $\bar{\sigma}$ and runs halfway along $\gamma$ in the other direction. Let $\rho$ and $\rho'$ be the geodesics in the homotopy classes of $\rho_0$ and $\rho'_0$, respectively.

Recall that if $\bar{\sigma}$ is oriented toward $\gamma$, then $\tau$ is homotopic to $\bar{\sigma} \cdot \gamma \cdot \bar{\sigma}^{-1}$. Equivalently, if $\rho$ and $\rho'$ are oriented toward $\gamma$, then $\tau$ is homotopic to $\rho' \cdot \rho^{-1}$. The geodesic in this homotopy class is denoted $g_\tau$. 
Figure \ref{fig:filled-geodesic} depicts lifts of $\bar{\sigma}$, $\gamma$, $\rho$, $\rho'$, and $g_\tau$ to the universal cover $\HH^3$.

\begin{figure}
	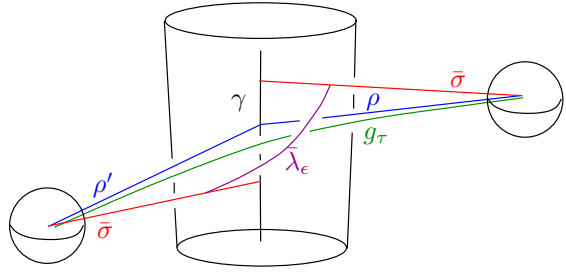
  \caption{A schematic picture of the lifts of $\gamma$,
    $\bar{\sigma}$, $\rho$, and $\rho'$ to the universal cover.  The
    arc $g_\tau$ is the geodesic in the homotopy class of the
    tunnel $\tau$.}
  \label{fig:filled-geodesic}
\end{figure}

The next lemma estimates the radius $r$ of the Margulis tube $\TT_\epsilon(\gamma)$, as well as the distances between $\bar{\sigma}$ and and $\rho$ (similarly, $\bar{\sigma}$ and and $\rho'$) along $\bdy \TT_\epsilon(\gamma)$.

\begin{lemma}\label{lemma:rho-dist}
Assume that the Dehn filling slope $\mu$ is long enough that its representative $\mu_\epsilon$ satisfies $\ell(\mu_\epsilon) \geq 10$, and also that $\mu$ is long enough to satisfy Theorem \ref{drilling-thm}; in particular, the normalized length of $\mu$ is $L(\mu) \geq 4 \sqrt{2} \pi$. Then
the radius of the Margulis tube $\TT_\epsilon(\gamma)$ is 
\begin{equation}\label{eq:r-bound}
r \: \geq \:  \sinh^{-1} \left( \frac{\ell(\mu_\epsilon) }{ 2.2 \, \pi } \right) \: > \: 1.16.
\end{equation}
The distance along $\rho$ between the vertex $v= \rho \cap \rho'$ and $\rho \cap \bdy \TT_\epsilon(\gamma)$ satisfies
\begin{equation}\label{eq:gamma-dist}
r  \: \leq \: d \left( \rho \cap \gamma, \, \rho \cap \bdy \TT_\epsilon(\gamma) \right) \: \leq \: r+h,
\end{equation}
where 
$$h < 1.4 \times 10^{-6} \quad \mbox{and} \quad  h \to 0 \quad \mbox{as} \quad L(\mu) \to \infty.$$
Finally, the distance on $\bdy \TT_\epsilon(\gamma)$ from $\bar{\sigma} \cap \bdy  \TT_\epsilon(\gamma)$ to $\rho \cap \bdy \TT_\epsilon(\gamma)$  satisfies 
\begin{equation}\label{eq:rho-dist}
d \left( \bar{\sigma} \cap \bdy  \TT_\epsilon(\gamma), \, \rho \cap \bdy \TT_\epsilon(\gamma) \right) \:< \:  0.02 \, e^{-r} +  h \:< \: 0.0063,
\end{equation}
and similarly for the distance from $\bar{\sigma} \cap \bdy  \TT_\epsilon(\gamma)$ to $\rho' \cap \bdy \TT_\epsilon(\gamma)$.
\end{lemma}

\begin{proof}
For \eqref{eq:r-bound}, choose $\mu$ so that its length on $\bdy \TT_\epsilon(T)$ is $\ell(\mu_\epsilon) \geq 10$. Then the corresponding curve $\bar{\mu}_\epsilon \subset X(\mu)$  is the circumference of a meridian disk of the Margulis tube $\TT_\epsilon(\gamma)$, and has length $\ell(\bar{\mu}_\epsilon) = 2\pi \sinh(r)$. By Lemma \ref{lemma:bilip-lengths}, the lengths of $\mu_\epsilon$ and $\bar{\mu}_\epsilon$ are $J$--related. Hence,
$$
 2\pi \sinh(r) \: = \: \ell(\bar{\mu}_\epsilon) \: \geq \: \ell(\mu_\epsilon) / 1.1 .
$$
Now, inequality \eqref{eq:r-bound} follows by solving for $r$:
$$
r \: \geq \: \sinh^{-1} \left( \frac{\ell(\mu_\epsilon) }{ 2.2 \, \pi } \right) \: \geq \:  \sinh^{-1} \left( \frac{ 10 }{ 2.2 \, \pi } \right) \: = \: 1.1649...
$$

For \eqref{eq:gamma-dist}, observe that the geodesics $\gamma$, $\rho$, and $\rho'$ form an isosceles $1/3$ ideal triangle $\Delta$, whose axis of symmetry is $\bar{\sigma}$. Lift this triangle to $\HH^3$, and label the two material vertices $v$ and $v'$ (these vertices project to the same point in $X(\mu)$, but are distinct in $\HH^3$). There is a single horocycle $C$ about the ideal vertex of $\Delta$ that passes through $v$ and $v'$. See Figure \ref{fig:rho-triangle}.

Note that by Theorem \ref{drilling-thm}\eqref{item:core-length}, the distance from $v$ to $v'$ along the lift $\tilde{\gamma}$ of $\gamma$ is
$$\ell(\gamma) \: \leq \: \tfrac{1}{8\pi}.$$
By Lemma \ref{lemma:13-ideal}, the distance from $v$ to $v'$ along the horocycle $C$ is
\begin{equation}\label{eq:p-est}
p \: = \: 2 \sinh  \tfrac{\ell(\gamma)}{2} \: \leq \: 2\sinh \tfrac{1}{16\pi} \: = \: 0.03978... 
\end{equation}
Thus, by the triangle inequality, the maximum distance by which $\tilde{\gamma}$ deviates from $C$ is
\begin{equation}\label{eq:h-est}
0 <  h \: \leq \: \left( \sinh \tfrac{1}{16\pi} \right) - \tfrac{1}{16\pi} \: = \: 1.312... \times 10^{-6},
 \end{equation}
where this tiny deviation approaches $0$ as $L(\mu) \to \infty$ and $\ell(\gamma) \to 0$.

To derive the first inequality of \eqref{eq:gamma-dist}, observe that the point $\rho \cap \bdy \TT_\epsilon(\gamma)$ is by definition at distance $r$ from the geodesic $\gamma$. Thus the closest point of $\gamma$ is at distance $r$, and the vertex $\rho \cap \gamma$ is at distance at $\ell > r$. For the second inequality of  \eqref{eq:gamma-dist}, observe that $\rho \cap \bdy \TT_\epsilon(\gamma)$ is closer to $C$ than $\bar{\sigma} \cap \bdy \TT_\epsilon(\gamma)$, which is at distance  $r+h$ from $C$.

\begin{figure}
  \input{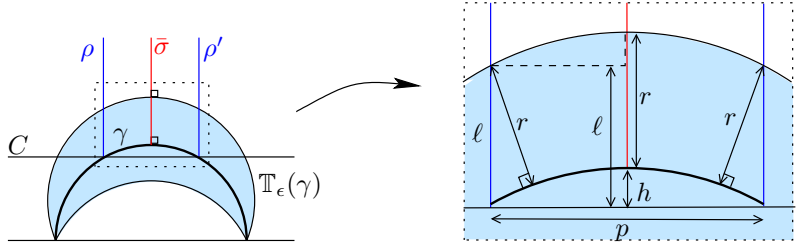}
\caption{(In universal cover) The triangle $\Delta(\gamma \rho \rho')$ intersects the
  Margulis tube $\TT_\epsilon(\gamma)$ as shown. Right:  zoomed in to
  show lengths.}
\label{fig:rho-triangle}
\end{figure}

Finally, to prove \eqref{eq:rho-dist}, note that the path $\beta$ from
$\bar{\sigma} \cap \bdy \TT_\epsilon$ to $\rho \cap \bdy \TT_\epsilon$
along $\bdy \TT_\epsilon(\gamma)$ has a length that can be computed as
an arclength integral in the upper half-space model:
$$\ell(\beta) \: = \: \int_\beta ds \: = \: \int_\beta \frac{\sqrt{dx^2 + dy^2}}{y} 
\: < \: \int_\beta \frac{\abs{dy}}{y} +  \int_\beta \frac{\abs{dx}}{y} \, .$$
In words, $\beta$ is shorter than the union of a
geodesic segment along $\bar{\sigma}$ (vertical in Figure
\ref{fig:rho-triangle}, dashed) followed by a horocyclic segment
(horizontal in Figure \ref{fig:rho-triangle}, dashed). The vertical
segment has length $(r+h)-\ell$, where $\ell$ is the length of the arc
of $\rho$ from $\bdy \TT_\epsilon(\gamma)$ to $\gamma$, hence
$\ell>r$.  Thus the vertical segment has length at most $h$.  The
horizontal segment has length $e^{-\ell} p/2$, which is at most
$e^{-r} p/2$. Thus
$$
\begin{array}{r c l l}
\ell(\beta) = d_{\bdy \TT_\epsilon} \left( \bar{\sigma} \cap \bdy  \TT_\epsilon, \, \rho \cap \bdy \TT_\epsilon \right) 
& \leq & p/2 \: e^{-r}  + h & \\
& < &  0.02 \, e^{-r} + h \, , & \mbox{by \eqref{eq:p-est} }\\
& < & 0.02 \, e^{-1.16} + 2 \times 10^{-6} \, ,  & \mbox{by \eqref{eq:r-bound} and \eqref{eq:h-est}  } \\
& = & 0.00627... \: , &
\end{array}
$$
and similarly for the distance from $\bar{\sigma} \cap \bdy  \TT_\epsilon$ to $\rho' \cap \bdy \TT_\epsilon$.
\end{proof}

\subsection{The triangle of $g_\tau$}\label{sec:length-completion}

To complete the proof of Theorem \ref{thm:length-estimate}, we need to carefully study the triangle  $\Delta \subset X(\mu)$ whose sides are $\rho$, $\rho'$, and $g_\tau$. (See Figures \ref{fig:filled-geodesic} and \ref{fig:def-C1-C2}.) The quantity that we seek is the length of $g_\tau$ relative to the maximal horocusp $N(K) \subset X(\mu)$.

\begin{lemma}\label{lemma:gt-length}
Assume that the Dehn filling slope $\mu$ is long enough to satisfy Condition \ref{condition:running} and Lemma \ref{lemma:rho-dist}.
 As in Figure \ref{fig:setup}, let $r+s$ be the length of $\bar{\sigma}$ from $\bdy \TT_\epsilon(\gamma)$ to the maximal cusp $\bdy N_\mu(K)$.  Then,
\begin{equation}\label{eq:23-length}
  r+s \: = \:
\frac{\ell(g_\tau)}{2} - \frac{1}{2} \ln \left(  \frac{1 - \cos \alpha}{2} \right) - h.
\end{equation}
where  $\alpha$ is the angle at the material vertex $v = \rho \cap \rho'$ of triangle $\Delta(\rho \rho' g_\tau)$, where the length $\ell(g_\tau)$ is measured relative to the maximal cusp $N_\mu(K)$, and where $0 < h < 2 \times 10^{-6}$ is the error term of Lemma \ref{lemma:rho-dist}.
\end{lemma}

\begin{proof}
The triangle $\Delta(\rho \rho' g_\tau)$ is an isosceles $2/3$ ideal triangle. Thus, by Lemma \ref{lemma:23-ideal}, 
\begin{equation}\label{eq:rho-gtau}
 \ell(\rho) + \ell(\rho') \: = \: \ell(g_\tau) - \ln \left(  \frac{1 - \cos \alpha}{2} \right),
 \end{equation}
where the lengths $\ell(\rho) = \ell(\rho')$ are measured from $v$ to the torus $\bdy N_\mu(K)$.

Next, observe from Figure \ref{fig:rho-triangle} that the geodesics $\rho$ and $\bar{\sigma}$ fellow--travel from $\gamma$ to the torus $\bdy N_\mu(K)$, and that their lengths differ by exactly $h$:
\begin{equation}\label{eq:rho-rs}
\ell(\rho) \: = \: \ell(\rho') \: = \: \ell(\bar{\sigma}) +h \: = \: (r+s) + h.
\end{equation}
Plugging \eqref{eq:rho-rs} into \eqref{eq:rho-gtau} completes the proof.
\end{proof}

 We will be able to estimate the length of $g_\tau$ once we obtain bounds on the length of a circle arc opposite $g_\tau$.

\begin{lemma}\label{lemma:circle-arc}
Assume that the Dehn filling slope $\mu$ is long enough to satisfy Condition \ref{condition:running} and Lemma 
\ref{lemma:rho-dist}. 
Let $\Delta \subset X(\mu)$ be the triangle whose sides are $\rho$, $\rho'$, and $g_\tau$. Let $v = \rho \cap \rho'$ be the material vertex of $\Delta$. Then the circle arc of radius $r$ and angle $\alpha$ about $v$ has length
\begin{equation}\label{eq:lambda-sinh-r}
e^{-x - 0.31} \,  \ell(\lambda) \: < \: \alpha \sinh r \: < \: e^{-x + 0.1} \, \ell(\lambda),
\end{equation}
where $x = d(\bdy N(T), \bdy \TT_\epsilon(T))$ in $X$, as in Lemma \ref{lemma:s-x}, and $\lambda$ is the shortest longitude for $\mu$.
\end{lemma}

\begin{proof}
The length of a circle arc of angle $\alpha$ and radius $r$ is always equal to $\alpha \sinh r$. The main goal of the lemma is to estimate this quantity in terms of  $x$ and $\ell(\lambda)$.

By Lemma \ref{lemma:s-x}, the distance between the maximal cusp and the $\epsilon$--sized cusp in $X$ is
$$x = d(\bdy N(T), \bdy \TT_\epsilon(T)) < 2.621.$$
Since the shortest longitude $\lambda$ for $\mu$ has length $\ell(\lambda) \geq 1$ on the maximal cusp $N(T)$,
the corresponding geodesic $\lambda_\epsilon \subset \bdy \TT_\epsilon(T)$ has length
$$
\ell(\lambda_\epsilon) \: = \:  e^{-x} \ell(\lambda) \: > \: e^{-2.621}  \: = \: 0.07273... \, .
$$
Since the lengths of $\lambda_\epsilon$ and the corresponding curve $\bar{\lambda}_\epsilon \subset \TT_\epsilon(\gamma)$ are $J$--related by Lemma \ref{lemma:bilip-lengths}, it follows that 
\begin{equation}\label{eq:lambda-bar}
0.0661... \: < \: \tfrac{10}{11}\, \ell(\lambda_\epsilon) \: \leq \ell(\bar{\lambda}_\epsilon) \: \leq \: \tfrac{11}{10}\, \ell(\lambda_\epsilon) .
\end{equation}

Let $C \subset \TT_\epsilon(\gamma)$ be a curve in the homotopy class
of $\bar{\lambda}_\epsilon$, constructed as follows.  We take $C=C_1
\cup C_2$, where $C_1$ is the shortest segment on $\bdy
\TT_\epsilon(\gamma)$ connecting points $\rho \cap
\TT_\epsilon(\gamma)$ and $\rho' \cap \TT_\epsilon(\gamma)$, hence
lying in the plane of the triangle $\Delta(\gamma\rho\rho')$.  We take
$C_2$ to be the shortest arc of intersection of $\bdy
\TT_\epsilon(\gamma)$ and the plane containing the triangle $\Delta(\rho\rho'g_\tau)$.
These arcs are shown schematically in Figure \ref{fig:def-C1-C2}.

\begin{figure}
  \input{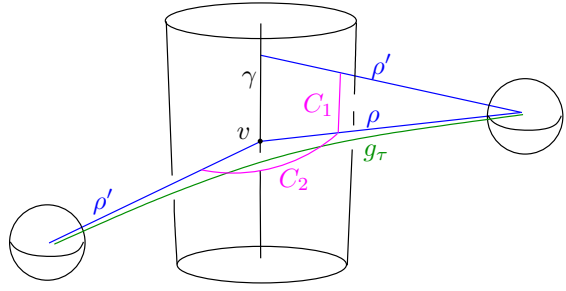}
  \caption{Schematic picture of the arcs $C_1$ and $C_2$ in the
    universal cover of $X(\mu)$.}
  \label{fig:def-C1-C2}
\end{figure}

Note that by equation \eqref{eq:rho-dist}, 
$$\ell(C_1) = d \left( \rho \cap \bdy  \TT_\epsilon(\gamma), \, \rho' \cap \bdy \TT_\epsilon(\gamma) \right)  \:< \: 2 \times 0.0063 \: = \: 0.0126.$$
Thus, by the triangle inequality and \eqref{eq:lambda-bar}, the longer segment $C_2 = C \setminus C_1$ has length
\begin{equation}\label{eq:c2-lower}
 \ell(C_2) \: > \:  \ell(\bar{\lambda}_\epsilon) - 0.0126 \: >  \: 0.8094 \, \ell(\bar{\lambda}_\epsilon)  \: > \:  0.7358 \,\ell(\lambda_\epsilon) \: = \: 0.7358 \,e^{-x} \ell(\lambda),
 \end{equation}
 where the last inequality used the fact that $\ell(\bar{\lambda}_\epsilon)$ and $  \ell(\lambda_\epsilon)$ are $J$--related.

Meanwhile, the points $\rho \cap \bdy \TT_\epsilon$ and $\rho' \cap
\bdy \TT_\epsilon$ are closer to each other along $C_2$ than the full
length of the longitude $\bar{\lambda}_\epsilon$.  This is because a
full longitude runs from one endpoint of $C_2$ to an endpoint of $C_1$
in Figure \ref{fig:def-C1-C2}, and the arc $C_1$ is vertical in the
cylindrical coordinates on the Euclidean torus $\bdy
\TT_\epsilon(\gamma)$. 
Thus
\begin{equation}\label{eq:c2-upper}
 \ell(C_2) \: < \:  \ell(\bar{\lambda}_\epsilon) \:< \: 1.1 \, \ell(\lambda_\epsilon) \: = \: 1.1 \,e^{-x} \ell(\lambda).
\end{equation}

It remains to relate the length of $C_2$ to the circle arc of length $\alpha \sinh r$. By equation \eqref{eq:gamma-dist}, the distance between the vertex $v = \rho \cap \rho'$ and the endpoints of $C_2$ is somewhere between $r$ and $r+h$, where $h < 1.4 \times 10^{-6}$. Meanwhile, the midpoint of $C_2$ is at distance exactly $r$ from $v$. Thus all of $C_2$ lies between circles of radius $r$ and $r+h$. Hence, up to a multiplicative error of less than
$$e^h < 1.000002,$$
the length of $C_2$ is the same as the length  $\alpha \sinh r$ of the circle arc of radius $r$. Combining this with \eqref{eq:c2-lower} and \eqref{eq:c2-upper}, we conclude that
$$ e^{-x-0.31} \ell(\lambda) \: < \: e^{-h} \cdot 0.7358 \,e^{-x} \ell(\lambda) \: < \: \alpha \sinh r \: < \:  e^{h} \cdot 1.1 \,e^{-x}  \ell(\lambda) \: < \: e^{-x+0.1} \ell(\lambda),$$
as desired.
\end{proof}

At this point, we are ready to complete the proof of Theorem \ref{thm:length-estimate}. In fact, we have the following version of the Theorem, which (mostly) quantifies how long $\mu$ needs to be in order to ensure that the estimates hold.

\begin{theorem}\label{thm:length-estimate-quant}
Let $X$ be an orientable hyperbolic $3$--manifold that has two cusps and an unknotting tunnel $\sigma$. Let $T$ be one boundary torus of $X$. Let $\mu$ be a Dehn filling slope on $T$ such that $\ell(\mu) > 138$, and such that $\mu$ is also long enough to satisfy Condition \ref{condition:running}. Then the unknotting tunnel $\tau$ of $M = X(\mu)$ associated to $\sigma$ satisfies
$$2 \ln \ell(\lambda) - 5.6 \: < \: \ell (g_\tau) \: < \: 2 \ln \ell(\lambda) + 4.5,$$
where $\lambda$ is the shortest longitude for $\mu$. 
Here $ \ell(\mu)$ and $\ell(\lambda)$ are lengths on a maximal cusp about $T$ in $X$, and $\ell (g_\tau)$ is the length of the geodesic in the homotopy class of $\tau$, relative to a maximal cusp in $M= X(\mu)$.
\end{theorem}

\begin{proof}
Let $\tau$ be an unknotting tunnel of $M = X(\mu)$ associated to an
unknotting tunnel $\sigma$ for $X$. Our assumptions about the length
of $\mu$ ensure that the estimates of Lemma \ref{lemma:s-x} apply. In
particular, by Lemma \ref{lemma:s-x}, $x < 2.621$. Thus the length of
$\mu$ on $\bdy \TT_\epsilon(T)$ is
$$\ell(\mu_\epsilon) \: = \: e^{-x} \, \ell(\mu) \: > \: e^{-2.621}
\cdot 138 \: > \: 10,$$
ensuring that Lemmas \ref{lemma:rho-dist},
\ref{lemma:gt-length}, and \ref{lemma:circle-arc} apply as well. The
proof of the theorem will involve combining the results of Lemmas
\ref{lemma:s-x} through \ref{lemma:circle-arc}.

The quantity we seek is the length $\ell(g_\tau)$. Recall that equation \eqref{eq:23-length} expresses $\ell(g_\tau)$ in terms of $r+s$, the angle $\alpha$, and the (tiny) error term $h$.
We can now use the previous lemmas to express $r+s$ in terms of $\ell(\lambda)$. One immediate consequence of equation \eqref{eq:r-bound} is that
$$e^{-r} \: < \:  e^{- 2.32} e^{r},$$
which implies that
\begin{equation}\label{eq:exp-sinh}
e^{-0.11} e^{r} \: < \:   \left( 1 - e^{- 2.32} \right) e^{r} \: < \:  e^{r} - e^{-r} \: = \: 2 \sinh r < e^{r} .
\end{equation}
Thus, by combining equation \eqref{eq:exp-sinh} with \eqref{eq:lambda-sinh-r}, we obtain
\begin{equation}\label{eq:expr-estimate}
 \tfrac{2}{\alpha} \, e^{-x-0.31} \, \ell( \lambda) \: < \: 2 \sinh r \: < \:  e^{r} \: < \: e^{0.11} \cdot 2 \sinh r \: < \:    \tfrac{2}{\alpha} \, e^{-x+0.21} \,  \ell(\lambda  ).
\end{equation}

Now, equation \eqref{eq:expr-estimate} bounds $e^r$, hence $r$, in terms of $\ell(\lambda)$, $x$, and the angle $\alpha$. Inserting the estimate of \eqref{eq:expr-estimate} into \eqref{eq:23-length}, as we will do in the  calculation below, will produce a term of the form $f(\alpha) = 2 (1- \cos \alpha)/\alpha^2$. By a derivative calculation,  one observes that $f(\alpha) $ is strictly decreasing on the interval $(0, \pi]$. In our case, the angle $\alpha$ must indeed be positive, and is at most $\pi$. Thus
\begin{equation}\label{eq:cos-taylor}
\frac{4}{\pi^2} = f(\pi) \leq \frac{2(1 - \cos \alpha)}{\alpha^2} < \lim_{\alpha \to 0} f(\alpha) = 1.
\end{equation}

With these preliminaries out of the way, we can perform the final calculation. Taking the log of the first, middle, and last terms of \eqref{eq:expr-estimate} yields

\begin{equation*}
\renewcommand{\arraystretch}{1.8}
\begin{array}{r c c c l}

\ln \ell(\lambda) - x -0.31 + \ln \left(  \tfrac{2}{\alpha} \right) & < &  r & < & \ln \ell (\lambda) - x + 0.21 + \ln \left(  \tfrac{2}{\alpha} \right) \\

\ln \ell(\lambda) +  \underbrace{ s - x }_{ \mathrm{use \: \eqref{eq:s-x}} }  -\, 0.31 + \ln \left(  \tfrac{2}{\alpha} \right) & < &  r+s  & < & \ln \ell (\lambda) +  \underbrace{ s - x }_{ \mathrm{use \: \eqref{eq:s-x}} }  + \, 0.21  + \ln \left(  \tfrac{2}{\alpha} \right) \\

\ln \ell(\lambda)  -  2   - 0.31 + \ln \left(  \tfrac{2}{\alpha} \right) & < & \underbrace{\: r+s \: }_{\mathrm{use \: \eqref{eq:23-length} }} & < & \ln \ell (\lambda)  +  2.02 + 0.21  + \ln \left(  \tfrac{2}{\alpha} \right) \\

\ln \ell(\lambda)   \underbrace {- 2.31+ h}_{  \mathrm{use \: \eqref{eq:h-est}} }  + \underbrace{ \tfrac{1}{2} \ln \left(  \tfrac{2(1 - \cos \alpha)}{\alpha^2}  \right) \!}_{\mathrm{use \: \eqref{eq:cos-taylor} }}   & < &   \ell(g_\tau) / 2   & < & \ln \ell (\lambda)  +  \underbrace {2.23+ h}_{  \mathrm{use \: \eqref{eq:h-est}} }  +  \underbrace{ \tfrac{1}{2} \ln  \left(  \tfrac{2(1 - \cos \alpha)}{\alpha^2}  \right) \! }_{\mathrm{use \: \eqref{eq:cos-taylor} }}  \\

 \ln \ell (\lambda )  + \underbrace{  (- 2.31) + \ln \left( \tfrac{2}{\pi} \right)  }_{= \: - 2.7615...} & < & \ell(g_\tau) /2 & < & 
 \ln \ell (\lambda ) +   2.231    \\
 
2  \ln \ell (\lambda ) - 5.6 & < & \ell(g_\tau) & < & 2  \ln \ell (\lambda ) + 4.5. \hfill \qedhere

\end{array} 
\end{equation*}
\end{proof}

\begin{remark}
By assuming that $\mu$ is extremely long, one can force most of error
in the preceding sequence of inequalities to become arbitrarily small.
The following exceptions are the only sources of error in the argument
which do not disappear as $\ell(\mu) \to \infty$.

First, the length of the unknotting tunnel $\sigma \subset X$, bounded
in equation \eqref{eq:sigma-length}, inevitably contributes to
additive error in estimating the length of $g_\tau$. This can be seen
most clearly near the end of the proof of Lemma \ref{lemma:s-x}, where
the factors of $1/10$ and $1/11$ depend on our choice of $J$, and the
additive error of $\pm \ln 4$, which came from the bound on
$\ell(\sigma_{\max})$, is all that remains if $J \to 1$.

The only other source of error that will not vanish as $\mu\to\infty$
is the term $\ln(2/\pi)$, which comes from equation
\eqref{eq:cos-taylor}.  This error term has a natural geometric
meaning, which can be seen by comparing Figures
\ref{fig:unfilled-notation} and \ref{fig:filled-schematic}.  The
horocycle (and Euclidean geodesic) $\lambda_\epsilon$ in Figure
\ref{fig:unfilled-notation} becomes $\bar{\lambda}_\epsilon$ in Figure
\ref{fig:filled-schematic}, which is essentially a circle arc when
$\mu$ is extremely long.  This arc can cover anywhere up to half of
the circle, corresponding to the angle $\alpha \in (0, \pi]$.  In the
Euclidean setting, the maximum ratio between an arc of a circle and
a chord through the middle is exactly $\pi/2$. In our hyperbolic
setting, the logarithmic savings achieved by a geodesic through the
thin part of a horoball (or the thin part of a Margulis tube) turns
this multiplicative error of $\pi/2$ into an additive error of $\ln
(\pi/2)$.

All in all, the sharpest possible version of the above argument, in which $\ell(\mu) \to \infty$, would give the asymptotic estimate
\begin{equation}\label{eq:length-asymptotic}
 \ln \ell (\lambda )  - \ln 4 - \ln \left( \tfrac{\pi}{2} \right)  \: \leq \: \ell(g_\tau) /2 \:  \leq \: 
 \ln \ell (\lambda ) + \ln 4 .
 \end{equation}
\end{remark}

\section{Geodesic unknotting tunnels}\label{sec:geodesic}

In this section, we prove Theorem \ref{thm:geodesic}, showing that unknotting tunnels created by generic Dehn filling are isotopic to geodesics. Here are the ingredients of the proof.

First, we will make extensive use of Theorem \ref{drilling-thm}. As in Sections \ref{sec:setup} and \ref{sec:length}, this will allow us to compare the geometry of the unfilled manifold $X$ to that of its Dehn filling $X(\mu)$. In particular, we will need the geometric control of Theorem \ref{drilling-thm} to construct an embedded collar about a geodesic $\bar{\sigma} \subset X(\mu)$. This is done in Section \ref{sec:collars}.

Second, we will build on several results from Section \ref{sec:length} --- specifically, Lemmas \ref{lemma:s-x} and \ref{lemma:rho-dist}, and Theorem \ref{thm:length-estimate-quant} --- to show that  the triangle $\Delta(\rho \rho' g_\tau)$ visible in Figure \ref{fig:filled-geodesic} is embedded in $X(\mu)$. The length estimates from Section \ref{sec:length} are only needed in a soft way. Mainly, we need to know that for generic Dehn fillings, the geodesic $g_\tau$ is arbitrarily long, which will imply that it stays within the thin part of $X(\mu)$, or else within an embedded tube about $\bar{\sigma}$. This argument is carried out in Section \ref{sec:embedded-triangles}.

Finally, in Section \ref{sec:geodesic-proof}, we combine these ingredients to prove Theorem \ref{thm:geodesic}. Using Theorem \ref{thm:tunnel-correspondence}, we can locate an arc  $\bar{\tau}(\sigma, \mu) \subset X(\mu)$ that is in the isotopy class of the unknotting tunnel $\tau$. By carefully sliding this arc through  the embedded triangle $\Delta(\rho \rho' g_\tau)$, we perform an isotopy of the tunnel to the geodesic $g_\tau$.  

\subsection{Embedded collars about geodesics} \label{sec:collars}
Following the notation of Sections \ref{sec:setup} and
\ref{sec:length}, $X$ will denote a two--cusped, hyperbolic manifold
that has tunnel number one.  For $\epsilon = 0.29$, and for any $J
>1$, the Drilling Theorem, Theorem \ref{drilling-thm}, implies that if
we exclude finitely many slopes $\mu$, there is a $J$--bilipschitz
diffeomorphism $\phi$ between thick parts of $X$ and
$X(\mu)$. We will assume throughout that $1 < J \leq 1.1$; this means
that all the estimates of Section \ref{sec:length} that hold true for
a $1.1$--bilipschitz diffeomorphism will also apply here.

Let $\sigma_{\epsilon} := \sigma \setminus \TT_\epsilon(X)$ denote the portion of the unknotting tunnel  $\sigma \subset X$ that lies in the thick part of $X$. (Recall Remark \ref{rem:thick}.) We begin the argument by studying the diffeomorphic image $\phi(\sigma_{\epsilon}) \subset X(\mu)$.

\begin{lemma}
  The image of the unknotting tunnel $\sigma \subset X$ under the
  embedding of $X$ into $X(\mu)$ is homotopic to a geodesic arc
  $\bar{\sigma}$.  Let $\bar{\sigma}_\epsilon \subset X(\mu)$ be the
  portion of this geodesic which lies in the thick part of
  $X(\mu)$.  Then there exists $J \in (1, \, 1.1]$,
  depending only on $X$ and $\sigma$, and $\delta>0$, depending on $(X, \sigma, J)$, such that the following hold:
  \begin{enumerate}[$(a)$]
  \item $N_\delta(\bar{\sigma}_\epsilon)$ is an embedded solid cylinder $D^2 \times
    I$ in $X(\mu)$.
  \item If $\phi\co X \setminus \TT_\epsilon(X) \to X(\mu)\setminus
    \TT_\epsilon (X(\mu))$ is a $J$--bilipschitz diffeomorphism, then
    $\phi(\sigma_{\epsilon})$ is contained in $N_\delta(\bar{\sigma}_\epsilon)$.
  \end{enumerate}
  In particular, $J$ and $\delta$ do not depend on the slope $\mu$.
  \label{lemma:embedded-Nsigma} 
\end{lemma}

\begin{proof}
Let $\sigma_{\epsilon} = \sigma \setminus \TT_\epsilon(X)$ denote the portion of the tunnel $\sigma $ that lies in the thick part of $X$.  Choose $r>0$ so that
$N_r(\sigma_{\epsilon})$ is an embedded tube in $X$.  Choose $J \in (1, \, 1.1]$ so that every arc homotopic to $\sigma_{\epsilon}$
of length less than $J^3\ell(\sigma_{\epsilon})$ lies in $N_r(\sigma_{\epsilon})$.  
For  this $J$, and
for all but finitely many slopes $\mu$, Theorem \ref{drilling-thm} gives
 a $J$--bilipschitz diffeomorphism $\phi$ between the
thick parts of $X$ and $X(\mu)$. (Here and below, we are using the term ``thick'' in the sense of Remark \ref{rem:thick}.)

For any $\mu$ such that $X(\mu)$ is hyperbolic, choose the smallest $\delta_\mu>0$ so that every
arc homotopic to $\bar{\sigma}_\epsilon$ in $X(\mu) \setminus
T(\epsilon)$ of length strictly less than $J^2\ell(\sigma_{\epsilon})$ must be
contained in $N_{\delta_\mu}(\bar{\sigma}_\epsilon)$.  Note that when the $J$--bilipschitz diffeomorphism $\phi$ exists, 
$\phi(\sigma_{\epsilon})$ has length bounded by $J\,\ell(\sigma_{\epsilon})$, hence
$\phi(\sigma_{\epsilon})$ lies in $N_{\delta_\mu}(\bar{\sigma}_\epsilon)$.

Let $\delta = \inf_\mu \delta_\mu$. We claim $\delta$ is strictly
greater than $0$.  For, suppose not. Then we have a sequence of slopes
$\mu_j$ such that $\delta_{\mu_j} =: \delta_j$ approaches $0$.
Note that the normalized lengths of these slopes must necessarily
approach infinity.  Hence by the Drilling theorem, for any sequence
$J_i>1$, with $J_i < J$ and $J_i \to 1$, we may find a subsequence
$\{\mu_i\}$ such that there is a $J_i$--bilipschitz diffeomorphism
$\phi_i$ from the thick parts of $X$  to those of
$X(\mu_i)$.  Let $d>0$ be such that the neighborhood $N_d(\sigma_{\epsilon}) \subset X
\setminus T(\epsilon)$ is foliated by arcs parallel to $\sigma_{\epsilon}$ of length
less than $J\ell(\sigma_{\epsilon})$.  Then $\phi_i(N_d(\sigma_{\epsilon}))$ is contained in
$N_{\delta_i}(\bar{\sigma}_\epsilon)$ for all $i$, and as $i\to
\infty$, the radius of $\phi_i(N_d(\sigma_{\epsilon}))$ approaches $d>0$.  Hence
$\delta_i$ cannot approach $0$, and thus $\delta>0$.

Now, $N_\delta(\bar{\sigma}_\epsilon)$ is foliated by arcs $\alpha$ of
length $\ell(\alpha) \leq J^2 \ell(\sigma_{\epsilon})$.
Then $\phi^{-1}(N_\delta(\bar{\sigma}_\epsilon))$ is foliated by arcs
$\phi^{-1}(\alpha)$, each of length bounded by $J^3 \ell(\sigma_{\epsilon})$.  By
choice of $r$, this implies
$\phi^{-1}(N_\delta(\bar{\sigma}_\epsilon))$ is contained in $N_r(\sigma_{\epsilon})$,
and hence is embedded.  Then $N_\delta(\bar{\sigma}_\epsilon)$ must
also be embedded.
\end{proof}

We may bootstrap from Lemma \ref{lemma:embedded-Nsigma} to the following statement.

\begin{lemma}\label{lemma:sigma-isotopy}
Let $\sigma$, $J$, and $\delta$ be as in Lemma \ref{lemma:embedded-Nsigma}. Then, whenever there is a $J$--bilipschitz diffeomorphism $\phi : X \setminus \TT_\epsilon(X) \to X(\mu)\setminus
    \TT_\epsilon (X(\mu))$, the arc $\phi(\sigma_{\epsilon})$ is isotopic to the geodesic $\bar{\sigma}_\epsilon$, through the embedded tube $N_\delta(\bar{\sigma}_\epsilon)$.
\end{lemma}

The point of the lemma is that the homotopy from $\phi(\sigma_{\epsilon})$ to $\bar{\sigma}_\epsilon$ is actually an isotopy. The proof below is inspired by Scharlemann's proof of Norwood's theorem that tunnel number one knots are prime \cite[Theorem 2.2]{scharlemann:poenaru-conj}.

\begin{proof}[Proof of Lemma \ref{lemma:sigma-isotopy}]
By Lemma \ref{lemma:embedded-Nsigma}, $\phi(\sigma_{\epsilon})$ is contained in $N_\delta(\bar{\sigma}_\epsilon)$. Furthermore, the complement of $\phi(\sigma_{\epsilon})$ in the thick part of $X(\mu)$ is homeomorphic to the complement of $\sigma$ in the thick part of $X$, which is a genus--$2$ handlebody. In particular, $X(\mu) \setminus (\TT_\epsilon X(\mu) \cup \phi(\sigma_{\epsilon}))$ has free fundamental group.

Now, suppose that $\phi(\sigma_{\epsilon})$ is not isotopic to the geodesic $\bar{\sigma}_\epsilon$. Since $\phi(\sigma_{\epsilon})$ is contained in the embedded ball $N = N_\delta(\bar{\sigma}_\epsilon)$, the only way this can happen is if it is knotted inside that ball. Let $A = \bdy N \setminus \bdy \TT_\epsilon X(\mu)$ be the annulus boundary between $N$ and the rest of the thick part of $X(\mu)$. Then, by van Kampen's theorem, 
$$\pi_1(N \setminus \phi(\sigma_{\epsilon})) \ast_{\pi_1(A)} \pi_1 \big(X(\mu) \setminus (\TT_\epsilon X(\mu) \cup N) \big) \: \cong \: \pi_1  \big( X(\mu) \setminus (\TT_\epsilon X(\mu) \cup \phi(\sigma_{\epsilon})) \big),$$
which is a free group. This means that $\pi_1(N \setminus \phi(\sigma_{\epsilon}))$ is a subgroup of a free group, hence itself free. Since $N \setminus \phi(\sigma_{\epsilon})$ is homeomorphic to a knot complement in $S^3$, and the only knot whose complement has free fundamental group is the unknot, the arc $\phi(\sigma_{\epsilon})$ must be unknotted in $N$. But then $\phi(\sigma_{\epsilon})$ must be isotopic to the geodesic at the core of $N$.
\end{proof}

We also note the following immediate consequence of Definition \ref{def:thick-thin}.

\begin{lemma}\label{lemma:embedded-tube}
The Margulis tube $V = \TT_\epsilon(\gamma)$ is an embedded solid torus in $X(\mu)$.
  Therefore, any arcs that are isotopic within $V$
must be isotopic in $X(\mu)$.  \qed
\end{lemma}

\subsection{Embedded triangles in $X(\mu)$} \label{sec:embedded-triangles}

Recall the geodesics $\rho, \rho' \subset X(\mu)$ that are depicted in Figure \ref{fig:filled-geodesic}, and played a role in Section \ref{sec:length}: $\rho$ is the geodesic in the homotopy class that follows $\bar{\sigma}$ and runs halfway along the core $\gamma$, while $\rho'$ is the geodesic in the homotopy class that follows $\bar{\sigma}$ and runs halfway along $\gamma$ in the other direction.

For our purposes here, an isotopy from an unknotting tunnel $\tau$ to the geodesic $g_\tau$ will involve deforming arcs through the triangle $\Delta(\rho \rho' g_\tau)$ with sides  $\rho$, $\rho'$, and $g_\tau$. (See Figure \ref{fig:filled-geodesic} for a lift of $\Delta$ to $\HH^3$.) Thus we need to know more about this triangle.

\begin{lemma}\label{lemma:rho-in-tube}
Let $\delta > 0$ be the tube radius of Lemma \ref{lemma:embedded-Nsigma}. Then, for a sufficiently long Dehn filling slope $\mu$ on $T$, the geodesics $\rho, \rho' \subset X(\mu)$ are contained in either the thin part of $X(\mu)$ or a tube of radius $\delta/2$ about $\bar{\sigma}$:
\begin{equation}\label{eq:rho-in-tube}
(\rho \cup \rho') \: \subset \: \TT_\epsilon X(\mu) \cup N_{\delta / 2}(\bar{\sigma}_\epsilon).
\end{equation}
\end{lemma}

\begin{proof}
First, observe that since $\rho$ and $\bar{\sigma}$ share an endpoint on the sphere at infinity, the distance between $\bar{\sigma}$ and a representative point $x \in \rho$ decreases monotonically as $x$ moves toward the shared ideal vertex. Thus, in the thick part of $X(\mu)$, the distance between $\rho$ and $\bar{\sigma}$ is maximized at $x = \rho \cap \bdy \TT_\epsilon(\gamma)$. Similarly, the distance between $\rho'$ and $\bar{\sigma}$ is maximized at $x' = \rho' \cap \bdy \TT_\epsilon(\gamma)$.

Now, the result follows immediately from Lemma \ref{lemma:rho-dist}. In the statement of that lemma, choosing $\mu$ long enough forces the tube radius $r$ to be as large as we like via equation  \eqref{eq:r-bound}. Choosing $\mu$ long enough also forces the error $h$ to be as small as we like. Thus, in equation \eqref{eq:rho-dist}, choosing a long slope $\mu$ ensures that
\begin{equation}
d \left( \bar{\sigma} \cap \bdy  \TT_\epsilon(\gamma), \, \rho \cap \bdy \TT_\epsilon(\gamma) \right) \:< \:  0.02 \, e^{-r} +  h \:< \: \delta/2,
\end{equation}
and similarly for the distance from $\bar{\sigma} \cap \bdy  \TT_\epsilon$ to $\rho' \cap \bdy \TT_\epsilon$.
\end{proof}

\begin{lemma}\label{lemma:triangle-in-tube}
For a generic Dehn filling slope $\mu$, the triangle $\Delta =
\Delta(\rho\rho'g_\tau)$ is contained in $\TT_\epsilon
X(\mu) \cup N_{\delta}(\bar{\sigma}_\epsilon)$.
Furthermore,  $g_\tau$ intersects the Margulis tube $\TT_\epsilon(\gamma)$. The portion of $g_\tau$ outside $\TT_\epsilon(\gamma)$, and outside the maximal cusp of
$X(\mu)$, has length less than $10$.
\end{lemma}

\begin{proof}
Recall that $\HH^3$ is a Gromov hyperbolic metric space, with a $\delta$--hyperbolicity constant of $\ln( \sqrt{2} +1)$. This means that any side of a triangle in $\HH^3$ is contained within a  $\ln( \sqrt{2} +1)$ neighborhood of the other two sides. Since triangle $\Delta = \Delta(\rho\rho'g_\tau)$ has a reflective symmetry, it follows that the midpoint $m$ of $g_\tau$ (fixed by the reflective symmetry) is within a  $\ln( \sqrt{2} +1)$ neighborhood of both $\rho$ and $\rho'$.

Next, observe that because $\rho$ and $g_\tau$ share an ideal vertex, the distance between side $\rho$ and $x \in g_\tau$ decreases exponentially toward $0$ as $x$ moves with unit speed toward the shared ideal vertex. The same statement is true for $\rho'$ and $g_\tau$. Thus, there is a constant $R$ (depending only on $\delta$) such that if $m$ is the midpoint of $g_\tau$,
\begin{equation}\label{eq:thin-triangle}
x \in g_\tau\: \mbox{ and } \: d(x, m) > R \qquad \Rightarrow \qquad d(x, \rho) < \delta/2 \: \mbox{ or } \: d(x, \rho') < \delta/2.
\end{equation}

Now, we recall several estimates from Section \ref{sec:length}. First, by Lemma \ref{lemma:s-x}, the distance in $X(\mu)$ between the Margulis tube $V = \TT_\epsilon(\gamma)$ and the maximal horocusp $N_\mu(K)$ is
\begin{equation}\label{eq:s-bound}
s \: = \: d(\bdy N_\mu(K), \bdy \TT_\epsilon(\gamma)) \: < \: 2.621 + 2.02 \: < \: 5.
\end{equation}
On the other hand, by Theorem \ref{thm:length-estimate-quant}, choosing $\mu$ and $\lambda$ sufficiently long will ensure that the length of $g_\tau$ relative to the maximal cusp $N_\mu(K)$ is arbitrarily long, in particular
\begin{equation}\label{eq:g-bound}
\ell(g_\tau) \: > \: 2 \ln \ell(\lambda) - 5.2 \: \gg \: 2R + 10.
\end{equation}
Combining \eqref{eq:s-bound} with \eqref{eq:g-bound}, we conclude that the middle $2R$ of the length of $g_\tau$  (i.e.\  the portion of $g_\tau$ that is within $R$ of the midpoint $m$) is contained inside the Margulis tube $\TT_\epsilon(\gamma)$. Thus, by \eqref{eq:thin-triangle}, it follows that when $\mu$ and $\lambda$ are sufficiently long,
\begin{equation}\label{eq:g-collared}
g_\tau \setminus \TT_\epsilon(\gamma) \: \subset \: N_{\delta/2}(\rho) \cup N_{\delta/2}(\rho').
\end{equation}
Combining \eqref{eq:rho-in-tube} and \eqref{eq:g-collared} gives the desired result.
\end{proof}

\begin{lemma}\label{lemma:embedded-triangle-gt}
For a generic slope $\mu$, the triangle  $\Delta = \Delta(\rho\rho'g_\tau)$ is embedded in $X(\mu)$. 
\end{lemma}

\begin{proof}
Consider  the intersections between $\Delta$ and the following three regions of $X(\mu)$: the Margulis tube $\TT_\epsilon(\gamma)$, the cusp neighborhood $\TT_\epsilon(K)$, and the remaining thick part $X(\mu) \setminus \TT_\epsilon$.  See Remark \ref{rem:thick} and Figure \ref{fig:filled-geodesic}.

We may lift $\Delta$ to a triangle $\widetilde{\Delta} \subset \HH^3$. There is a decomposition of $\HH^3$ into tubes covering the Margulis tube $\TT_\epsilon(\gamma)$, horoballs covering $\TT_\epsilon(K)$, and the remaining piece covering the thick part. We will show that $\Delta$ is embedded in $X(\mu)$ by showing that none of these pieces of $\HH^3$ contains a point of intersection between $\widetilde{\Delta}$ and one of its translates by the deck transformation group.

First, consider a solid tube about a geodesic $\widetilde{\gamma}
\subset \HH^3$, which covers the Margulis tube
$\TT_\epsilon(\gamma)$. By Lemma \ref{lemma:embedded-tube}, the only
deck transformations of $\HH^3$ that fail to move this solid tube off
itself belong to a $\ZZ$ subgroup fixing the geodesic axis
$\widetilde{\gamma}$. Now, observe that each lift of $\bar{\sigma}$
intersecting $\widetilde{\gamma}$ is contained in a totally geodesic
plane perpendicular to $\widetilde{\gamma}$, and that the triangle
$\widetilde{\Delta}$ is bounded between two consecutive planes.  See
Figure \ref{fig:filled-geodesic}.  Thus every non-trivial deck
transformation in the $\ZZ$ subgroup fixing $\widetilde{\gamma}$ will
move $\widetilde{\Delta}$ completely off itself.

Next, consider the piece of $\Delta$ in the thick part $X(\mu)
\setminus \TT_\epsilon X(\mu)$. By Lemma \ref{lemma:triangle-in-tube},
this piece of $\Delta$ is entirely contained in the embedded tube
$N_\delta(\bar{\sigma}_\epsilon)$.  Therefore, $\widetilde{\Delta}
\setminus \widetilde{\TT_\epsilon}$ is entirely contained in one of
the disjointly embedded solid cylinders covering
$N_\delta(\bar{\sigma}_\epsilon)$, hence is disjoint from all the
images of $\widetilde{\Delta}$ by the deck transformation group.

Finally, consider the part of $\Delta$ in the horocusp $\TT_\epsilon(K)$. By Lemma \ref{lemma:triangle-in-tube}, we already know that $\Delta \cap \bdy \TT_\epsilon(K)$ is contained in an embedded disk of radius $\delta$ about $\bar{\sigma} \cap \bdy \TT_\epsilon(K)$. Thus $\Delta \cap \TT_\epsilon(K)$ follows the $T^2 \times [0, \infty)$ product structure of the horocusp $\TT_\epsilon(K)$, and the portion of $\Delta$ in this cusp is also embedded.
\end{proof}

\subsection{Completing the proof}\label{sec:geodesic-proof}
We can now complete the proof of Theorem \ref{thm:geodesic}.

\begin{named}{Theorem \ref{thm:geodesic}}
Let $X$ be an orientable hyperbolic $3$--manifold that has two cusps and tunnel number one. Choose a generic filling slope $\mu$ on one cusp of $X$, and let $\tau \subset X(\mu)$ be an unknotting tunnel associated to a tunnel $\sigma \subset X$. Then $\tau$ is isotopic to  the geodesic $g_\tau$ in its homotopy class.
\end{named}

\begin{proof}
The two-cusped manifold $X$ has finitely many Heegaard splittings of genus $2$, hence finitely many unknotting tunnels \cite{hass:genus2, li:waldhausen-conj}. Choose a value $J \in (1, \, 1.1]$ that satisfies Lemma \ref{lemma:embedded-Nsigma} for every unknotting tunnel of $X$.  Next, choose a generic Dehn filling slope $\mu$, such that both $\mu$ and its shortest longitude $\lambda$ are long enough to satisfy Lemmas \ref{lemma:sigma-isotopy}, \ref{lemma:rho-in-tube}, and \ref{lemma:embedded-triangle-gt}.

Choose an unknotting tunnel $\sigma \subset X$, and let $\tau = \overline{\tau}(\sigma, \mu) \subset X(\mu)$ be an unknotting tunnel associated to $\sigma$, as in Definition \ref{def:associated} and Theorem \ref{thm:tunnel-correspondence}\eqref{item:new-tunnel}. Let us unpack what this means.

Using the geodesic $\overline{\sigma} \subset X(\mu)$ and the core curve $\gamma$ of the Dehn filling solid torus, construct geodesics $\rho$ and $\rho'$ as in Figure \ref{fig:filled-geodesic}. 
Let $\bar{Q} = \Delta(\rho \rho' \gamma) \setminus \TT_\epsilon
X(\mu)$ be a totally geodesic quadrilateral obtained by removing the
thin part $\TT_\epsilon X(\mu)$ from the triangle $\Delta(\rho \rho'
\gamma)$.  The quadrilateral $\bar{Q}$ can be seen schematically in
Figure \ref{fig:def-C1-C2}, where a lift of $\bar{Q}$ lies on the
right of the figure, with one side labeled $C_1$, two sides on (lifts
of) $\rho'$ and $\rho$, and the last side on the horosphere
where these lifts of $\rho'$ and $\rho$ meet.

By Lemma \ref{lemma:rho-in-tube}, $\bar{Q}$ is embedded in the tubular
neighborhood $N_{\delta/2}(\bar{\sigma}_\epsilon)$. Let $\bar{C}$ be
an arc in $\bdy \TT_\epsilon(\gamma)$ whose endpoints are $\rho \cap
\bdy \TT_\epsilon(\gamma)$ and $\rho' \cap \bdy \TT_\epsilon(\gamma)$,
and which lies in the same plane as triangle $\Delta(\rho \rho'
g_\tau)$. (This is the same arc that was denoted $C_2$ in the proof of
Lemma \ref{lemma:circle-arc}, shown in Figure \ref{fig:def-C1-C2}.)
Notice that the union of $C_1 = \bar{Q} \cap \bdy
\TT_\epsilon(\gamma)$ and $C_2 = \bar{C}$ is a closed curve isotopic to
$\bar{\lambda}_\epsilon$. See Figure \ref{fig:def-C1-C2}.

Now, $Q = \phi^{-1}(\bar{Q})$ is an embedded quadrilateral in the thick part of $X$. Furthermore, observe that the two sides of $\bar{Q}$ that run along $\rho$ and $\rho'$ are each isotopic to $\bar{\sigma}_\epsilon$, and $\bar{\sigma}_\epsilon$ is isotopic to $\phi(\sigma_\epsilon)$ by Lemma \ref{lemma:sigma-isotopy}. Pulling back these isotopies via $\phi^{-1}$, we conclude that two opposite sides of $Q$, namely $s_1 = \phi^{-1}(\rho)$ and $s_2 = \phi^{-1}(\rho')$, run parallel to $\sigma$ through the thick part of $X$. Furthermore, the union of $\phi^{-1} (\bar{C})$ and $Q \cap \bdy \TT_\epsilon(T)$ is a closed curve isotopic to $\lambda$. In other words, the quadrilateral $Q$ satisfies all the criteria of Definition  \ref{def:associated}.

In the language of Definition \ref{def:associated}, this means that $\tau(\sigma, \mu) = s_1 \cup \phi^{-1} (\bar{C}) \cup s_2$ is a tunnel arc associated to $\sigma$ and $\mu$. Therefore, by Theorem \ref{thm:tunnel-correspondence}, 
$$\phi(\tau(\sigma, \mu)) = (\rho \setminus \TT_\epsilon) \cup \bar{C} \cup (\rho' \setminus \TT_\epsilon)$$ 
is an arc isotopic to our unknotting tunnel $\tau$.

Now, we may construct an isotopy from $\phi(\tau(\sigma, \mu))$ to $g_\tau$ in two stages. First, homotope arc $\bar{C}$ to the piecewise geodesic arc $(\rho \cap \TT_\epsilon(\gamma)) \cup (\rho' \cap \TT_\epsilon(\gamma))$, through the Margulis tube  $\TT_\epsilon(\gamma)$. Because the Margulis tube is embedded by Lemma \ref{lemma:embedded-tube}, and in fact the entire homotopy occurs in the geodesic plane containing triangle $\Delta(\rho \rho' g_\tau)$, this homotopy is an isotopy. At the end of this isotopy, we have shown that tunnel $\tau$ is isotopic to $\rho \cup \rho'$, two of the sides of $\Delta(\rho \rho' g_\tau)$.

Next, homotope $\rho \cup \rho'$ to the third side $g_\tau$ of $\Delta(\rho \rho' g_\tau)$, through the triangle $\Delta(\rho \rho' g_\tau)$. Because $\Delta(\rho \rho' g_\tau)$ is embedded in $X(\mu)$ by Lemma \ref{lemma:embedded-triangle-gt}, the homotopy from two sides of the triangle to the third side is again an isotopy. Therefore, $\tau$ is isotopic to the geodesic $g_\tau$.
\end{proof}

It is worth observing that in the proof of Theorem \ref{thm:geodesic}, we used the hypothesis on the length of $\lambda$ only at the very end, to apply Lemma \ref{lemma:embedded-triangle-gt} and construct an isotopy from $\rho \cup \rho'$ to the geodesic $g_\tau$. When we strip away the last paragraph of the proof, what remains is the following corollary, which does not require any hypotheses on the longitude $\lambda$.

\begin{corollary}\label{cor:piecewise-geodesic}
Let $X$ be an orientable hyperbolic $3$--manifold that has two cusps and an unknotting tunnel $\sigma$. Choose a Dehn filling slope $\mu$ on one boundary torus $T$ of $X$, and let $\tau \subset X(\mu)$ be the unknotting tunnel associated to $\sigma$, as in Figure \ref{fig:filled-tunnel}. Then, for all but finitely many choices of filling slope $\mu$, the tunnel $\tau$ is isotopic to the piecewise geodesic curve $\rho \cup \rho'$ depicted in Figure \ref{fig:filled-geodesic}.
\end{corollary}

\section{Canonical and non-canonical geodesics}\label{sec:canonical}

Heath and Song showed by a single counterexample that unknotting tunnels are not necessarily
isotopic to canonical geodesics \cite{heath-song}.  In this section, we
give conditions that will guarantee that an unknotting tunnel is a
canonical geodesic, or is \emph{not} a canonical geodesic.  This results in the proof of Theorem \ref{thm:canonicity}. In addition, in Theorem \ref{thm:2bridge}, we construct an infinite family of one--cusped manifolds $M_i$, such that each $M_i$ in the family has one canonical tunnel and one non-canonical tunnel.

Recall, from Definition \ref{def:canonical}, that a geodesic in $M = X(\mu)$ is called \emph{canonical} when it is an edge of the canonical polyhedral decomposition $\mathcal{P}$, i.e.,  the geometric
  dual to a face of the Ford domain.  
Note that when there are multiple cusps, the Ford domain depends on choice
of horoball expansion for all cusps.  However, when there is only one
cusp, canonical geodesics are well defined.  

Consider again the manifold $X$, with cusps corresponding to torus
boundary components $K$ and $T$, and unknotting tunnel $\sigma$ running
between them.  Take a horoball expansion as follows.  For the cusp
corresponding to $K$, to be left unfilled, expand a horocusp 
maximally, until it becomes tangent to itself.  Now expand a horocusp 
about $T$ slightly, so that it has very small volume.  This expansion
determines a Ford domain, hence a canonical polyhedral decomposition. We will use this canonical polyhedral decomposition of $X$ in the
results below.   (It is a theorem of Akiyoshi \cite{akiyoshi:finiteness} that the combinatorics of the Ford domain stabilizes as the volume of a horocusp shrinks toward $0$; thus there is a well--defined canonical decomposition corresponding to a ``sufficiently small'' horocusp about $T$.)

The proof of Theorem \ref{thm:canonicity} will use the recent work of
Gu\'eritaud and Schleimer \cite{gueritaud-schleimer}.  In that paper,
the authors show that if the canonical polyhedral decomposition of a
manifold $X$ is a triangulation, and if there is a unique shortest
canonical geodesic meeting a cusp to be filled, then the canonical
decomposition of the filled manifold $X(\mu)$ can be
determined by replacing two tetrahedra of the canonical triangulation 
of $X$ with a collection of tetrahedra forming a
``layered solid torus.''

The results we need are contained in the following lemma.

\begin{lemma}
  \label{lemma:layered-torus}
  Let $X$ be a hyperbolic manifold with two cusps, denoted $K$ and $T$,
such that there is a unique shortest geodesic $\alpha$ running
  from $K$ to $T$. Let $\mathcal{D}$ be the canonical polyhedral decomposition of $X$, relative to a cusp neighborhood where the horocusp about $T$ is sufficiently small.
  Then, for a sufficiently long slope $\mu$ on $T$, every edge $E$ of the canonical decomposition $\mathcal{D}_\mu$ of $X(\mu)$ satisfies one of the following:
 \begin{enumerate}[$(a)$]
 \item\label{item:old-edge} $E$ is isotopic to the image of a
   canonical edge in $\mathcal{D}$ under the (topological) inclusion
   $X \hookrightarrow X(\mu)$, or
  \item\label{item:subdivided-edge} $E$ is isotopic to the image of an
    edge created by subdividing some polyhedron of $\mathcal{D}$ into
    ideal tetrahedra, under the inclusion $X \hookrightarrow X(\mu)$, or
  \item\label{item:layered} $E$ is isotopic to an arc that follows
    $\alpha$ into the filled solid torus, runs $n$ times around the
    core of the solid torus (for some positive integer $n \in \NN$),
    and then follows $\alpha$ back out. There is exactly one such edge
    for $n=1.$
 \end{enumerate}
The edges in \eqref{item:layered} are exactly the edges of the
canonical tetrahedra that make up the layered solid torus
$\mathcal{V}$, which will be described in the proof.

\end{lemma}

\begin{proof}
Suppose, for the moment, that the polyhedral decomposition $\mathcal{D}$ consists entirely of tetrahedra. (This hypothesis is assumed globally in  \cite{gueritaud-schleimer}. See Remark \ref{rem:why-tetrahedra} below.)

If the horocusp about $T$ is sufficiently small, Gu\'eritaud and Schleimer show that the only edge of $\mathcal{D}$ entering cusp $T$ is the unique shortest geodesic $\alpha$. In particular, the cusp cellulation of $T$ will contain exactly one vertex, corresponding to the endpoint of $\alpha$ at $T$.  There are two special tetrahedra, denoted $\Delta$ and $\Delta'$, such that three edges of $\Delta$ (sharing an ideal vertex) and three edges of $\Delta'$ are identified at the shortest geodesic $\alpha$. All other tetrahedra of $\mathcal{D}$ are disjoint from the horocusp about $T$, and have all of their vertices at $K$. See  \cite[Section 4.1]{gueritaud-schleimer}.

For a sufficiently long Dehn filling slope $\mu$, Gu\'eritaud and Schleimer observe that the tetrahedra of  $\mathcal{D} \setminus (\Delta \cup \Delta')$ remain canonical in $X(\mu)$. This is because the canonicity of a tetrahedron can be encoded in finitely many strict inequalities (which express relative distance to various horoballs in $\HH^3$, or equivalently convexity in Minkowski space $\RR^{3+1}$). Hence, the canonicity of a tetrahedron is an open condition, and remains true as the complete hyperbolic metric on $X$ is perturbed slightly to give the hyperbolic structure on $X(\mu)$.

If $\mathcal{D}$ consists entirely of tetrahedra, 
Gu\'eritaud and Schleimer show that
 the canonical decomposition $\mathcal
D_\mu$ will combinatorially be of the form
$$ \mathcal D_\mu = \left( \mathcal{D} \setminus \{\Delta,\Delta'\} \right) \cup \mathcal V,$$
where $\mathcal V = \Delta_1 \cup \dots \cup \Delta_N$ is a solid torus
with one point on the boundary removed. This removed point corresponds to the endpoint of $\alpha$ on $K$.
The solid torus  $\mathcal{V}$ has a layered triangulation by tetrahedra $\Delta_1, \ldots, \Delta_N$, 
as follows.
The tetrahedron $\Delta_1$ is glued along two faces to the punctured torus $\bdy \mathcal{V} = \bdy(\Delta \cup \Delta')$, then $\Delta_2$ is glued along two faces to the punctured torus on the other side of $\Delta_1$, and so on. At the core of the solid torus $\mathcal{V}$, two faces of $\Delta_N$ are glued by folding onto a M\"obius band. 
See  \cite[Section 2]{gueritaud-schleimer} for more details, and in particular \cite[Figure 3]{gueritaud-schleimer}.

\smallskip

For our purposes, the salient points are as follows. First, every edge of $\mathcal{V} = \Delta_1 \cup \dots \cup \Delta_N$ runs some nonzero number of times about the core of $\mathcal{V}$. 
Second, all of these edges share the same ideal vertex, namely the endpoint of $\alpha$ at cusp $K$.
Thus every edge of $\mathcal{V}$ is homotopic to a portion of $\alpha$, followed by $n \in \NN$ trips about the core, followed by returning to the ideal vertex along $\alpha$.
The homotopy class of each (unoriented) edge in $\mathcal{V}$ is completely determined by the positive integer $n$. Thus, since there can only be one hyperbolic geodesic in any homotopy class, there is at most one canonical edge for any value $n \in \NN$. Note that there will actually be an edge in $\mathcal{V}$ for $n=1$: this is the core of the M\"obius band onto which the tetrahedra are layered. This edge is marked with an arrow in  \cite[Figure 3]{gueritaud-schleimer}.

This completes the proof in the case where $\mathcal{D}$ consists entirely of tetrahedra. Next,  consider what can be said without this restrictive hypothesis.

\smallskip

By the same argument as in Section 4.1 of \cite{gueritaud-schleimer}, the cusp cellulation of  a horospherical torus about $T$ will contain exactly one vertex, corresponding to the endpoint of $\alpha$ at $T$. Thus the Delaunay decomposition of the torus $T$ is either two triangles or one rectangle. The corresponding canonical $3$--cells of $\mathcal{D}$ with ideal vertices at $T$ are either two tetrahedra (which may be labeled $\Delta, \Delta'$ as before), or a single ideal rectangular pyramid, which can be subdivided into two tetrahedra $\Delta, \Delta'$ by choosing a diagonal for the rectangle. 

When we perform Dehn filling along a sufficiently long slope $\mu$, the same argument as above implies that every canonical tetrahedron of $\mathcal{D} \setminus (\Delta \cup \Delta')$ will remain canonical in $X(\mu)$. The presence of a larger canonical polyhedron in $\mathcal{D}$ is equivalent (by duality) to a vertex $v$ of the Ford--Voronoi domain that is equidistant to five or more horoballs. As the complete hyperbolic metric on $X$ is perturbed slightly to give the metric on $X(\mu)$, this equality of distances may break into inequalities, and a large polyhedron in $\mathcal{D}$ may become subdivided into tetrahedra or other cells. Nonetheless, we observe that any new edges created in this fashion are interior to polyhedra of $\mathcal{D}$ or to faces of $\mathcal{D}$. Thus all edges in $X(\mu)$ that do not come from $(\Delta \cup \Delta')$ satisfy \eqref{item:old-edge} or \eqref{item:subdivided-edge} in the statement of the lemma.

After all the cells of $\mathcal{D} \setminus (\Delta \cup \Delta')$ are mapped to $X(\mu)$, and subdivided as necessary, what remains is again a solid torus $\mathcal{V}$ with one point on the boundary removed. Then, by the same argument of Gu\'eritaud and Schleimer \cite[Section 4.4]{gueritaud-schleimer}, the canonical subdivision of $\mathcal{V}$ will again be a layered triangulation. As above, every edge of $\mathcal{V}$ satisfies \eqref{item:layered} in the statement of the lemma, completing the proof.
\end{proof}

\begin{remark}\label{rem:why-tetrahedra}
The main reason why Gu\'eritaud and Schleimer assumed that the canonical decomposition $\mathcal{D}$ of $X$ consists entirely of tetrahedra is that their goal was to completely describe the canonical decomposition of $X(\mu)$. Typically, large cells in the canonical decomposition $\mathcal{D}$ of $X$  tend to break up into tetrahedra in $\mathcal{D_\mu}$, in a pattern that seems hard to predict from the combinatorial data alone. For our purposes in this paper, it suffices to know that each edge created in this subdivision is contained in the closure of a cell of $\mathcal{D}$.
\end{remark}

We may now complete the proof of Theorem \ref{thm:canonicity}.

\begin{named}{Theorem \ref{thm:canonicity}}
{
Let $X$ be a two-cusped, orientable hyperbolic $3$--manifold in which there is a unique
shortest geodesic arc between the two cusps. Choose a generic Dehn filling slope $\mu$ on a cusp $T$ of $X$. Then, for each unknotting tunnel $\sigma \subset X$, the tunnel $\tau \subset X(\mu)$ associated to $\sigma$ will be canonical 
 if and only if $\sigma$ is the shortest geodesic between the two cusps of $X$.
}
\end{named}

\begin{proof}
For the ``if'' direction of the theorem, suppose that the unknotting tunnel $\sigma$ of $X$ is the unique
shortest canonical geodesic from $T$ to the other cusp $K$ of $X$.  By Definition \ref{def:associated}, the associated tunnel $\bar{\tau}(\sigma, \mu) \subset X(\mu)$ follows $\bar{\sigma}$ to the added solid torus $V$, runs once around the longitude $\lambda$ (i.e., once around the core of $V$), and returns along $\bar{\sigma}$. By Lemma \ref{lemma:layered-torus}\eqref{item:layered}, there is exactly one edge $E$ in the canonical decomposition of $X(\mu)$ that does exactly that. Thus the tunnel $\bar{\tau} = \tau(\sigma,\mu)$ is homotopic to this edge $E$. In fact, since $E$ and $\bar{\tau}$ are both boundary--parallel arcs in the layered solid torus $\mathcal{V}$, they are isotopic in $\mathcal{V}$, hence in $X(\mu)$.

We remark that the proof of this direction does not need any hypotheses on $\lambda$; all that's needed is that the slope $\mu$ is long enough to apply the work of Gu\'eritaud and Schleimer.

\smallskip

For the ``only if'' direction of the theorem, suppose that  the unknotting tunnel $\sigma$ of $X$ is \emph{not}
the shortest canonical geodesic connecting cusps $T$ and $K$.  Denote this unique shortest geodesic by $\alpha$.

Choose a value of $\epsilon \leq 0.29$ small enough so that the $\epsilon$--thin part $\TT_\epsilon(T) \subset X$ is contained in a horocusp about $T$ that is \emph{sufficiently small} to satisfy Lemma \ref{lemma:layered-torus}. Then Theorem \ref{drilling-thm} implies that for a sufficiently long Dehn filling slope $\mu$, the boundary of the $\epsilon$--thin Margulis tube $\TT_\epsilon(\gamma) \subset X(\mu)$ will be contained in $\mathcal{V}$:
$$\bdy \TT_\epsilon(\gamma) \:  \subset \:  \mathcal{V} \: \subset \: X(\mu),$$
where $\mathcal{V}$ is the layered solid torus of Lemma \ref{lemma:layered-torus}.

Next, let $\tau = \bar{\tau}(\sigma, \mu)$ be the unknotting tunnel of $X(\mu)$ associated to $\sigma$. We claim that when $\mu$ and $\lambda$ are sufficiently long, the geodesic $g_\tau$ in the homotopy class of $\tau$ must intersect the Margulis tube $\TT_\epsilon(\gamma)$. The proof of this claim is essentially identical to the proof of Lemma \ref{lemma:triangle-in-tube}, where it is proved for the value $\epsilon = 0.29$. (For a different value of $\epsilon$, there would be different numerical estimates in \eqref{eq:s-bound}, but the argument is otherwise the same.) Thus we may conclude that for a generic slope $\mu$, the geodesic  $g_\tau$ in the homotopy class of $\tau$ must intersect the layered solid torus $\mathcal{V}$.

We are now ready to complete the proof. By Lemma
\ref{lemma:layered-torus}, if an edge of the canonical decomposition meets $\mathcal{V}$, then it is isotopic to an arc that follows $\alpha$, runs some number of times around the core curve $\gamma$ of $\mathcal{V}$, then follows $\alpha$ back out. But we have assumed that $\alpha \neq \sigma$; in particular, these geodesics enter the horocusp $K$ at different points. Thus the geodesic $g_\tau$, which is homotopic to $\bar{\sigma} \cdot \gamma \cdot \bar{\sigma}^{-1}$, cannot be a canonical edge in $X(\mu)$.
\end{proof}

\begin{theorem}
  There exists an infinite family $\{M_i\}$ of one-cusped hyperbolic tunnel number one manifolds, such that each $M_i$ has 
  two unknotting tunnels, of which one is canonical and the other
  is not.
  \label{thm:2bridge}
\end{theorem}

\begin{proof}
To prove the result, all we need is a two--cusped, one--tunnel manifold $X$
with two unknotting tunnels, one of which is the unique shortest
canonical geodesic, and one of which is not.  We may then Dehn fill one cusp of $X$ and apply
Theorem \ref{thm:canonicity}
to obtain infinitely many tunnel
number one manifolds $M_i = X(\mu_i)$ as in the statement of the theorem.

\begin{figure}
  \input{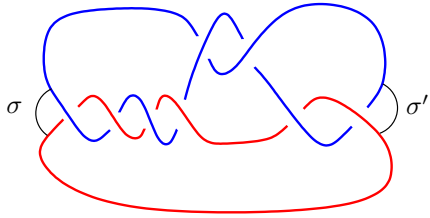}
  \caption{A two--bridge link $L$ of slope $5/22$, with upper tunnel
    $\sigma$ and lower tunnel $\sigma'$.}
  \label{fig:2-bridge-link}
\end{figure}

The example we use is the complement of a 2--bridge link $L$, shown in
Figure \ref{fig:2-bridge-link}.  By a result of Adams and Reid
\cite{adams-reid}, the only unknotting tunnels of a two--bridge link
are the upper and lower tunnels.  These are shown in Figure
\ref{fig:2-bridge-link} as $\sigma$ and $\sigma'$.

Using SnapPy \cite{weeks:snappea}, we compute the hyperbolic
structure on $S^3 \setminus L$, shrink the cusp about the lower (red)
component of Figure \ref{fig:2-bridge-link}, and expand the cusp about the upper (blue) component maximally to determine the appropriate canonical polyhedral decomposition.  Figure \ref{fig:cusp-neighborhood} shows the cusp neighborhood for this link.
Notice that there is a single maximal horoball in a fundamental domain in the figure.  This implies that there is a unique shortest canonical geodesic between the two cusps.  We will see that this corresponds to the lower tunnel $\sigma'$.

\begin{figure}
  \includegraphics[width=5in]{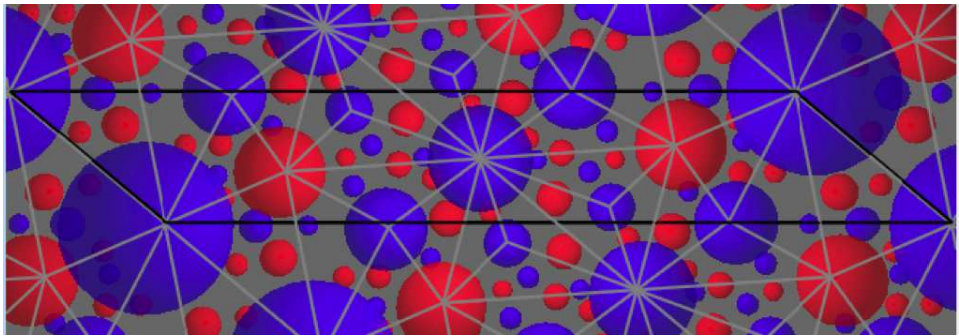}
  \caption{A horoball packing diagram for the two--bridge link of
    slope $5/22$. The red cusp (shaded lighter in grayscale) lifts to the horoball about infinity.}
  \label{fig:cusp-neighborhood}
\end{figure}

To do so, we consider the combinatorics of the canonical decomposition
of the 2--bridge link complement.  By a theorem of Akiyoshi, Sakuma,
Wada, and Yamashita \cite{aswy:book}, and independently 
Gu\'eritaud \cite{gueritaud:thesis}, the canonical polyhedral
decomposition of the link complement follows a combinatorial pattern that can be
read off the diagram of the link. See Sakuma
and Weeks \cite{sakuma-weeks} or  Futer \cite[Appendix]{gf:two-bridge} for detailed descriptions of this triangulation.  

In particular, the canonical
decomposition determines a triangulation of the cusp.  
This cusp triangulation has the feature that exactly two edges run over an entire meridian of the cusp. Each of these two special edges forms one side of an ideal triangle, while the other two sides of the triangle run along the (same) upper or lower tunnel. In Figure \ref{fig:cusp-neighborhood}, we see the upper and lower tunnels ``head on,'' as vertices in the cusp torus. 
One of the tunnels
must correspond to the maximal horoball, as claimed, and the other
corresponds to the smaller horoball shown in the center of 
Figure \ref{fig:cusp-neighborhood}.  By counting valences of the
vertices lying over these horoballs, we see that the lower tunnel
$\sigma'$ corresponds to the maximal horoball.

Now, Theorem \ref{thm:tunnel-correspondence}
 implies that for a generic Dehn
filling, the resulting manifold has two unknotting tunnels
$\bar{\tau}(\sigma, \mu)$ and $\bar{\tau}(\sigma', \mu)$.  By Theorem 
\ref{thm:canonicity}, the second of these is
canonical, while the first is not.  
\end{proof}

\section{Knots with long tunnels in ${\mathbf S^3}$}\label{sec:knots}

In this section, we prove there exist knots in $S^3$ with arbitrarily
long unknotting tunnel.
Our starting point is the alternating chain link in $S^3$ with four
link components.  We will denote the link by $C$, its complement by
$M = S^3\setminus C$.  This link is hyperbolic, for example by work of
Neumann and Reid
\cite{neumann-reid:arithmetic}.  (In fact, the arguments below will
apply to any choice of hyperbolic chain link on four strands. We
choose the alternating one for concreteness.)

Label the four link components $L_1$, $L_2$, $L_3$, and $L_4$, with
$L_1$ and $L_3$ opposite each other.  
Notice that there is a
3--punctured sphere with boundary components a longitude of $L_1$ and
a meridian of $L_2$ and $L_4$, embedded on the plane of projection of
the link diagram.  This will play a part in the arguments below. See
Figure \ref{fig:4-link-chain}.

Our knot complement in $S^3$ will be obtained by Dehn filling three of
the four link components, along carefully chosen slopes.

\begin{lemma}
  \label{lemma:2-splittings}
  The link complement $M = S^3\setminus C$ has exactly two
  genus--2 Heegaard splittings.
\end{lemma}

\begin{proof}   
One genus--2 Heegaard surface separates $L_1$ and $L_2$ from $L_3$
  and $L_4$.  It restricts to a standard genus--2 Heegaard splitting
  of $S^3$, and the links $L_1$ and $L_2$ run through two distinct
  1--handles of a corresponding genus--2 handlebody in $S^3$.  The two
  core tunnels in the compression bodies separated by this surface
   correspond to arcs between
  $L_1$ and $L_2$, and between $L_3$ and $L_4$.  These core tunnels are the
  unique arcs on the intersection of 3--punctured spheres bounded by
  longitudes of these link components. 
  See Figure \ref{fig:4-link-chain}.

  The other genus--2 Heegaard surface of $M$ is obtained by rotating the
  right panel of Figure \ref{fig:4-link-chain} by a quarter-turn.  It separates $L_1$ and $L_4$ from $L_2$ and $L_3$,
  with core tunnels running between the corresponding boundary
  components.

 By Lemma \ref{lemma:hyperelliptic}, every core tunnel of $M$ can be isotoped
  so it is fixed by an isometry that is a  hyper-elliptic involution. Since $M$ has 4 torus
  boundary components, such an involution must map each boundary component to itself.
  It must also reverse the orientation on each longitude and each meridian.
   We claim there is only one such isometry, and that it preserves the $3$--punctured sphere $S$ depicted in Figure \ref{fig:4-link-chain}.

  \begin{figure}
  \begin{center}
    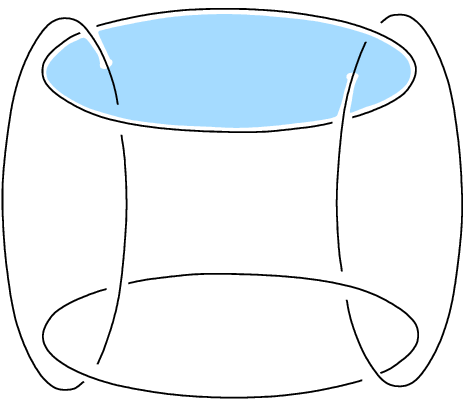
    \hspace{.5in}
    \includegraphics{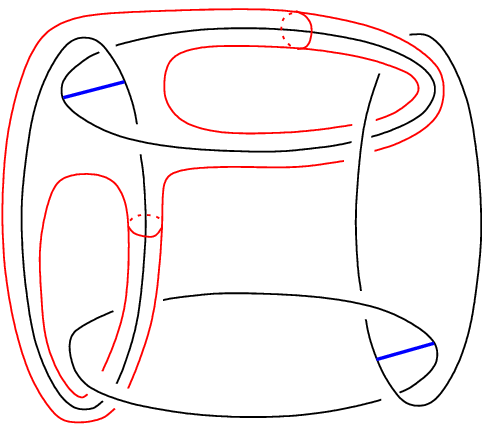}
  \end{center}
  \caption{The alternating chain link with four link components.  On
    the left, an embedded 3--punctured sphere $S$ is shown, shaded.  On
    the right, a Heegaard surface is shown, along with corresponding
    core tunnels.}
  \label{fig:4-link-chain}
\end{figure}

  If there are two such involutions, their composition is an isometry, $\psi,$ which sends
  (a neighborhood of) each component of the link to itself, preserving the orientation
  of each longitude. 
  The orbit of $S$ under  the finite cyclic group $F$ generated by $\psi$  consists of disjoint $3$--punctured spheres,
  one for each element of  $F$. These spheres separate the link complement into $|F|$ isometric pieces.
  However, the cusp of $L_3$ lies in exactly one of these complementary components.
 This contradicts the existence of a second hyper-elliptic isometry.

Now, suppose for a contradiction that the manifold $M = S^3 \setminus C$ contains another genus--$2$ Heegaard surface $\Sigma$, apart from the two surfaces that we have already exhibited. We know that the hyper-elliptic involution corresponding to $\Sigma$ rotates the $3$--punctured sphere $S$ about its central axis, and that two link components of $C$ lie on the same side of $\Sigma$. By the pigeonhole principle, an arc in the axis of $S$ that connects two link components on the same side of $\Sigma$ must be a core tunnel of $\Sigma$.

  There are three potential core tunnels contained in the axis of $S$.  Two of these are known
  core tunnels of existing splittings, described above.  The third is an arc $\alpha$ running
  from $L_4$ to $L_2$, shown in Figure \ref{fig:4-link-chain}.  We now argue that this arc cannot be a core tunnel in $M \setminus \Sigma$.  
  For suppose that $\alpha$ is a core  tunnel.
  Then the genus--2 Heegard surface $\Sigma$ must separate $L_2$ and $L_4$
  from $L_1$ and $L_3$.  This surface must carry the
  $\ZZ_2$--homology of $M = S^3\setminus C$.  Note $H_1(M ;
  \ZZ_2)$ is generated by the four meridians $m_1$, $m_2$, $m_3$, and
  $m_4$, hence has rank 4.

Let $N$ be the compression body that is the neighborhood of the cusp tori of $L_2$ and $L_4$, along with the arc $\alpha$ joining these tori. The subgroup $G \subset H_1(M; \ZZ_2)$ corresponding to the inclusion $N \hookrightarrow M$ is generated by meridians and
  longitudes of the boundary components $L_2$ and $L_4$.  The longitude
  of $L_2$ is homologous to $m_1+m_3 \pmod 2$, as is the longitude of
  $L_4$.  Hence $G \subset H_1(M; \ZZ_2)$ has rank only $3$, and cannot be the entire group. On the other hand, if the positive boundary $\bdy_+ N$ was a Heegaard surface $\Sigma$, $N$ would have to carry all of the homology of $M$. This is a contradiction.
 Hence there
  are exactly two Heegaard splittings, described above.
\end{proof}

\begin{lemma}
  \label{lemma:link-TcrossI}
  Dehn filling $M = S^3 \setminus C$ along an integer slope $p/1$ on cusp $L_3$ yields a manifold  homeomorphic to $(T^2 \times I) \setminus K$, where $K$ is a knot. The homeomorphism maps the boundary torus of $L_2$ to $T^2 \times \{0\}$, the boundary torus of $L_4$ to $T^2 \times \{1\}$, and the link component $L_1$ to the knot $K \subset T^2 \times I$.
\end{lemma}

\begin{proof}
  If we perform meridian Dehn filling on $L_1$, then we obtain an
  unlinked chain of three unknots.  The complement of this chain of unknots is
  homeomorphic to the  product of a 3--punctured sphere $P$ and
  the circle, $P\times S^1$, which is Seifert fibered with base
  orbifold $S^2(\infty, \infty, \infty)$.  Thus $L_1$ is a knot in $P\times S^1$.

  Now consider $p/q$ Dehn filling on the component $L_3$.  The
  longitude of $L_3$ bounds the 3--punctured sphere fiber, so is
  horizontal.  Dehn filling yields a new Seifert fibered space with
  base orbifold $S^2(\infty,\infty, q)$.  Hence if we perform integral
  Dehn filling, $q=1$, then the Seifert fibered space has base
  orbifold $S^2(\infty, \infty)$.  This is an annulus.  The Seifert
  fibered space must be the product of an annulus and a circle,
  which is homeomorphic to $T^2\times I$. Note that the two boundary tori $T^2 \times \{0,1 \}$ come from the boundary tori of $L_2$ and $L_4$. Thus, after $p/1$ integer filling
  on $L_3$, the link component $L_1$ is a knot $K \subset T^2\times I$.  
\end{proof}

\begin{lemma}
  For large $n$, the manifold obtained by Dehn filling $M = S^3 \setminus C$ along the slope
  $2n$ on $L_3$ is hyperbolic, and has exactly two genus--2 Heegaard
  splittings, which come from the Heegaard splittings of $M$.
  \label{lemma:tube-involution}
\end{lemma}

\begin{proof}
Let $X = T^2\setminus K$ denote the manifold obtained by this Dehn
filling.  Again by Lemma \ref{lemma:hyperelliptic}, any genus--2
Heegaard splitting of $T^2\setminus K$ gives rise to an involution of
the manifold $X$ that fixes the boundary components.  For large $n$, the
core of the Dehn filling solid torus will be shorter than any
other closed geodesic.  Thus it must be taken to itself
by the involution.  But then the involution restricts to an involution
of $S^3\setminus C$ that fixes boundary components, preserving slopes
on the boundary components $L_1$, $L_2$, and $L_4$.

As in the proof of Lemma \ref{lemma:2-splittings}, this fixes the
3--punctured sphere $S$ with boundary components isotopic to the longitude
of $L_1$ and meridian of $L_2$ and $L_4$.  So again the core
tunnel must be one of the three arc components fixed by a reflection
of that 3--punctured sphere.  Again two of these arcs are known
core tunnels, and we must again rule out the third arc $\alpha$, shown in Figure \ref{fig:4-link-chain}.  We do so by
a homology argument.

As before, if a core tunnel $\alpha$ connects $L_2$ and $L_4$, then the compression body
$N$ obtained as a neighborhood of these two components and the tunnel must carry the
homology of the manifold $M$.  The homology $H_1(S^3\setminus C; \ZZ_2)$
is generated by meridians $m_1$, $m_2$, $m_3$, and $m_4$.  After Dehn
filling along $(2n)m_3 + \ell_3$, the longitude $\ell_3$ becomes equal
to zero $\mod 2$.  Hence $\ell_3 = 0 = m_2 + m_4$, and the homology now
has rank $3$.

On the other hand, the subgroup $G \subset H_1(M; \ZZ_2)$ induced by the inclusion $N \hookrightarrow M$ is generated by
meridians and longitudes $m_2$, $\ell_2$, $m_4$, and $\ell_4$.  Since
$\ell_2 = m_1+m_3 \pmod 2 = \ell_4,$ and $m_2 = m_4$ after Dehn
filling, $G$ has rank only $2$. Thus the positive boundary $\bdy_+N$ cannot be a Heegaard splitting surface.
\end{proof}

By Lemmas \ref{lemma:link-TcrossI} and \ref{lemma:tube-involution},
there exists a knot  $K \subset (T^2\times I)$ with exactly two genus--2
Heegaard splittings, and with hyperbolic complement.  Denote the hyperbolic manifold
$X = (T^2\times I) \setminus K$.  To obtain a knot in $S^3$, we will be
Dehn filling two more of the boundary components of $X$: those corresponding
to $T^2 \times \{0\}$ and to $T^2 \times \{1\}$.

Let $\iota(\cdot,\cdot)$ denote the geometric intersection number between two slopes on a torus. The following lemma is well-known.

\begin{lemma}
  \label{lemma:s3-filling}
  Consider the manifold $Y = T^2 \times [0,1]$.  Put the same framing on
  $T^2 \times \{0\}$ and $T^2 \times \{1\}$.  Choose slopes $\mu_i$ on   $T^2 \times \{i\}$,
  such that $\iota(\mu_0, \mu_1) = 1$. Then the Dehn filled manifold $Y(\mu_0, \mu_1)$ is 
 homeomorphic to $S^3$. 
\end{lemma}

\begin{proof}
There is a mapping class $\varphi \in SL(2,\ZZ)$, mapping $T^2 \to T^2$, which sends the slope $\mu_0$ to $0/1$ and the slope $\mu_1$ to $\pm 1/0$. Now, if we identify the target $T^2$ with the standard Heegaard torus of $S^3$, the product homeomorphism $(\varphi \times id)$ maps $Y = T^2 \times I$ to the complement of the Hopf link in $S^3$, with $\mu_0$ and $\mu_1$ mapped to the meridians
 of the two link components. Filling these in gives $S^3$.
 \end{proof}

We are now ready to prove the main theorem of this section.

\begin{theorem}
  \label{thm:long-s3}
  For any $L>0$, there exists a tunnel number one knot in $S^3$ with
  exactly two unknotting tunnels, each of which has length at least
  $L$.
\end{theorem}

\begin{proof}
The knot in $S^3$ will be obtained by Dehn filling two cusps of the
manifold $X = (T^2 \times I) \setminus K$ of Lemma \ref{lemma:link-TcrossI}.  In particular, we will
fill cusps corresponding to $T_0 :=T^2 \times \{0\} $ and $T_1 := T^2
\times \{1\}$.  Note one of the two core tunnels of
$(T^2\times I) \setminus K$, call it $\sigma_0$, runs from $K$ to $T_0$,
and the other, $\sigma_1$, from $K$ to $T_1$.  After appropriate Dehn
filling, we will see that Theorem \ref{thm:tunnel-correspondence}
applies to give unknotting tunnels $\tau_i$ corresponding
to $\sigma_i$, and that by Theorem \ref{thm:heegaard-correspondence} these are the only unknotting tunnels of the resulting knot.

We need to take care in our choice of slopes, to ensure that the
tunnels stay long, and also to ensure the resulting filled manifold is
a knot complement in $S^3$ with no new genus--2 Heegaard splittings.

Put the same framing on $T_0$ and on $T_1$, and choose disjoint horocusps about these tori.  Set  $J= 1.1$, and
$\epsilon =0.29$, as Section \ref{sec:length}. Also, choose a length $L >0$. Now, we will choose slopes $\mu_0$ on $T_0$ and $\mu_1$ on $T_1$, such that the following hold:

\begin{enumerate}[(A)]
\item\label{item:drilling} The normalized lengths $L(\mu_i)$ are long enough to satisfy the Drilling Theorem \ref{drilling-thm} for $J=1.1$ and $\epsilon=0.29$. In fact, each $L(\mu_i)$ is at least $J^2$ times longer than necessary to apply the Drilling Theorem. This way, we can apply Theorem \ref{drilling-thm} twice: once to fill $T_0$ and a second time to fill $T_1$ (or in the opposite order).
\item\label{item:apply-length} On the disjoint horocusps about $T_0$ and $T_1$,  the length of each $\mu_i$ (on its respective torus) satisfies $\ell(\mu_i) > 152$.
\item\label{item:longitude} For each slope $\mu_i$, the shortest longitude $\lambda_i$ satisfies $\ell(\lambda_i) > 7$ and $\ell(\lambda_i) > e^{3+L/2}$. Again, these lengths are measured on the chosen horospherical tori about $T_0$ and $T_1$. Note that \emph{every} longitude for $\mu_i$ satisfies the same lower bound on length.
\item\label{item:intersection-numbers} The intersection number is $\iota(\mu_0, \mu_1) = 1$.
\end{enumerate}

We claim that  in the complement of a bounded region in the Farey graph $\mathcal{F}$, conditions \eqref{item:drilling}--\eqref{item:intersection-numbers} hold simultaneously. In the language of the Farey graph, conditions \eqref{item:drilling} and \eqref{item:apply-length} require $\mu_0$ and $\mu_1$ to avoid finitely many vertices of $\mathcal{F}$. Condition \eqref{item:longitude} requires each of $\mu_0$ and $\mu_1$ to avoid finitely many closed balls of radius $1$ in $\mathcal{F}$, where each ball is the set of all Farey neighbors of a longitude too short for \eqref{item:longitude}. Finally, condition \eqref{item:intersection-numbers} requires $\mu_0$ to be a Farey neighbor of $\mu_1$. Thus, by choosing an edge $[\mu_0, \mu_1] \subset \mathcal{F}$  that lies outside the bounded prohibited region, we satisfy all of \eqref{item:drilling}--\eqref{item:intersection-numbers}.

We Dehn fill $T_0$ along $\mu_0$, and $T_1$ along $\mu_1$.  By Lemma
\ref{lemma:s3-filling}, the result is a knot in $S^3$.

Applying Theorem \ref{thm:heegaard-correspondence} twice (once to fill $T_0$, and again to fill $T_1$), we conclude that the
resulting knot complement $X(\mu_0, \mu_1)$ has Heegaard genus $2$, with any genus--$2$ Heegaard surfaces coming from the Heegaard surfaces of the
original manifold $X$.  By Lemma \ref{lemma:tube-involution}, there are
exactly two Heegaard surfaces in $X$ (with
core tunnels $\sigma_0$ running from $K$ to
$T_0$, and $\sigma_1$ running from $K$ to $T_1$).  Thus there are exactly two unknotting tunnels in $X(\mu_0, \mu_1)$. 

We claim that the two unknotting tunnels for the resulting knot are associated to $\sigma_0$ and $\sigma_1$. To see this, first fill one of the cusps, say
$T_0$.  Then $\sigma_1$, running from $K$ to $T_1$, remains a core tunnel in its compression body, and is now an unknotting tunnel for $X(\mu_0)$.  Now, Theorem \ref{thm:tunnel-correspondence} implies that
$\tau_1 = \bar{\tau}(\sigma_1, \mu_1)$,  is an unknotting
tunnel for the knot complement $X(\mu_0, \mu_1)$.  Similarly, if we first fill
$T_1$, then Theorem \ref{thm:tunnel-correspondence} implies that
$\tau_0 = \bar{\tau}(\sigma_0, \mu_0)$,  is an unknotting
tunnel for the knot complement $X(\mu_0, \mu_1)$.  Thus we have found exactly two
unknotting tunnels, and these must be all the tunnels for the
manifold.

It remains to show that the tunnels have length at least $L$.  First, fill $X$ along slope $\mu_0$ on $T_0$.
Then, by Theorem \ref{drilling-thm}\eqref{item:level-preserve}, there is a $J$--bilipschitz diffeomorphism $\phi$ that maps the horocusp about $T_1$ in $X$  to an embedded horocusp about $T_1$ in $X(\mu_0)$. Since $J=1.1$, the meridian slope $\phi(\mu_1)$ has length at least
\begin{equation}\label{eq:filled-mu}
\ell(\mu_1)/1.1 \: > \: 152/1.1 \: > \: 138.
\end{equation}
Similarly, the shortest longitude for $\phi(\mu_1)$  in $X(\mu_0)$ has length at least
\begin{equation}\label{eq:filled-lambda}
   \ell(\lambda_1)/1.1 \:>\: e^{3+L/2}\cdot e^{-0.1} \: = \: e^{(5.8+L)/2}.
\end{equation}
Thus, by equation \eqref{eq:filled-mu} and condition \eqref{item:drilling}, Dehn filling $X(\mu_0)$ along slope $\mu_1$ on $T_1$ will satisfy all the hypotheses of Theorem \ref{thm:length-estimate-quant}. Therefore, by Theorem \ref{thm:length-estimate-quant} and equation \eqref{eq:filled-lambda}, the unknotting  tunnel $\tau_1$ of $X(\mu_0, \mu_1)$ will have length longer than $L$.

By the same argument, reversing the order of the fillings, the other unknotting tunnel $\tau_0$ of $X(\mu_0, \mu_1)$ will also have length longer than $L$.
\end{proof}

\section{A concrete example}\label{sec:example}

In this section, we take another look at the construction of Section \ref{sec:knots}, giving a concrete
example. As a result, in Theorem \ref{thm:explicit-long-tunnel} we obtain a thoroughly effective version of Theorem \ref{thm:long-s3}, with an explicit sequence of knots and an explicit bound on the length of their tunnels. The downside of this concrete construction is that rigorous computer assistance will be required at two places in the argument (Lemmas \ref{lemma:census-knot} and \ref{lemma:exactly2}).

\subsection{The manifold $X$ and its Heegaard splittings}

For the entirety of this section, we will work with the knot $K$ in $T^2\times I$ that is depicted in  Figure \ref{fig:knot-diagram}. It is a pleasant exercise to show that this knot is obtained by $2/1$ Dehn
filling on one component of the alternating chain link $C$ from Section \ref{sec:knots}. Since we will not need this fact, we omit the derivation.

\begin{figure}
	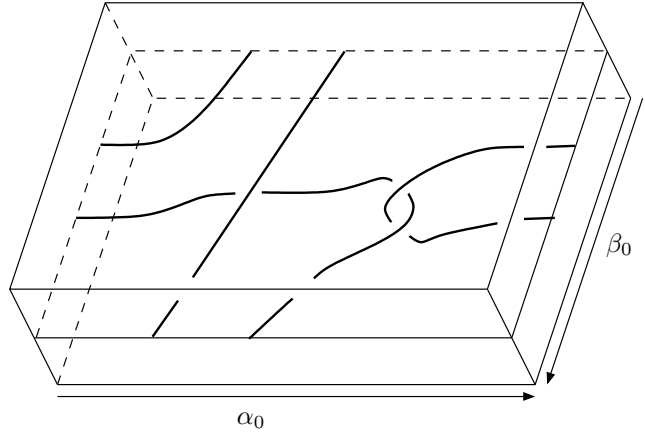
  \caption{Diagram of a knot $K$ in $T^2\times I$. Note the framing:
    $\alpha_0$ runs to the right, $\beta_0$ runs to meet $\alpha_0$ as
    shown. Slopes $\alpha_1$ and $\beta_1$ are on the top torus, parallel to $\alpha_0$ and $\beta_0$, respectively.}
  \label{fig:knot-diagram}
\end{figure}

The following lemma collects a few useful facts about  $X = (T^2 \times I) \setminus K$.

\begin{lemma}\label{lemma:census-knot}
Let $X = (T^2 \times I) \setminus K$ be the complement of the knot in Figure \ref{fig:knot-diagram}. Then
\begin{enumerate}[(a)]
\item\label{item:tetrahedra} $X$ has an ideal triangulation consisting of seven tetrahedra, as depicted in Figure \ref{fig:3227-triangulation}.
\item\label{item:isometric} $X$ is isometric to the SnapPea census manifold  ${\tt v}_{3227}$. 
\item\label{item:basis} The  basis $\langle \alpha_i, \beta_i \rangle$ for $H_1(T^2  \times \{i\})$, shown in Figure \ref{fig:knot-diagram}, places the same framing on  $T_0 := T^2 \times \{0\} $ and $T_1 :=T^2 \times \{1\}$. 
\item\label{item:shortest} In the hyperbolic metric on $X$, there is a unique shortest edge from $K$ to $T_0$ (edge $e_0$, marked $0$ in Figure \ref{fig:3227-triangulation}), and a unique shortest edge from $K$ to $T_1$ (edge $e_1$, marked $1$). 
\item\label{item:parabolic} On the maximal cusp about $T_0$, the slopes $\alpha_i$ and $\beta_i$ are realized by parabolic translations
\begin{equation}\label{eq:parabolic0}
\alpha_0 : z \mapsto z + 2.383, \qquad \beta_0 : z \mapsto z + 4.222 + 2.657 \,  \sqrt{-1}.
\end{equation}
On the maximal cusp about $T_1$, the slopes $\alpha_i$ and $\beta_i$ are realized by parabolic translations
\begin{equation}\label{eq:parabolic1}
\alpha_1 : z \mapsto z + 7.961 + 1.269 \,  \sqrt{-1}, \qquad \beta_1 : z \mapsto z + 4.989 .
\end{equation}
The real and imaginary parts in \eqref{eq:parabolic0} and \eqref{eq:parabolic1} are accurate to within $0.01$.
\end{enumerate}
\end{lemma}

\begin{figure}[b]
\psfrag{0}{$0$}
\psfrag{1}{$1$}
\psfrag{2}{$2$}
\psfrag{3}{$3$}
\psfrag{4}{$4$}
\psfrag{5}{$5$}
\psfrag{6}{$6$}
	\includegraphics{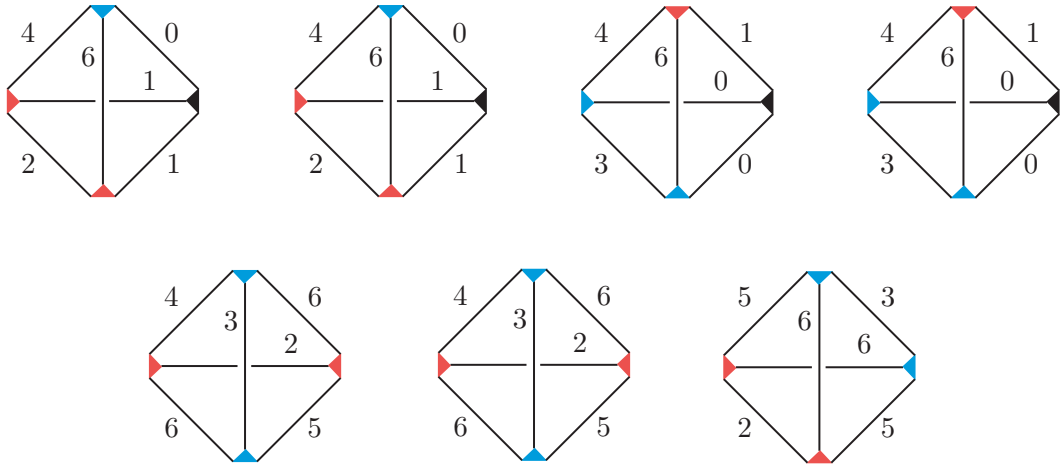}
  \caption{The $7$ tetrahedra in an ideal triangulation of $X = (T^2 \times I) \setminus K$. The cusp of $T_0$ is blue, the cusp of $T_1$ is red, and the cusp of $K$ is black. For convenience in the ensuing argument, edge labels are permuted from the labels in the SnapPea census, via the permutation $e_0 \leftrightarrow e_5$, $e_1 \leftrightarrow e_6$, $e_2 \leftrightarrow e_3$.}
  \label{fig:3227-triangulation}
\end{figure}

\begin{proof}
Conclusions \eqref{item:tetrahedra}, \eqref{item:isometric}, and \eqref{item:basis} are immediate consequences of rigorous routines in SnapPy \cite{weeks:snappea}. In particular, the construction of an ideal triangulation is rigorous. The isometry checker is rigorous, because (in the case of a ``yes'' answer) it exhibits a simplicial isomorphism between $X$ and ${\tt v}_{3227}$. Similarly, the realization of particular peripheral curves as moves in the triangulation is a rigorous combinatorial operation. Note that, once the meridian of $K$ is filled, $\alpha_0$ becomes isotopic to $\alpha_1$, and $\beta_0$ to $\beta_1$.

Next, we turn to the work of Moser \cite{moser}, which can be summarized as follows. Whenever SnapPy produces an approximate solution to the gluing equations for a triangulation, with sufficiently small error, then an exact solution exists nearby (where ``nearby'' is explicitly quantified, with Lipschitz estimates on the distance between approximate tetrahedron shapes and true tetrahedron shapes). For every census manifold, including $X \cong {\tt v}_{3227}$, Moser verifies that the error in SnapPy's approximate solution to the gluing equations is bounded by
\begin{equation}\label{eq:moser}
\abs{b} < 1.8 \cdot 10^{-26},
\end{equation}
easily enough to ensure a true solution nearby. Thus $X$ is hyperbolic.

Moser's estimate on the distance between an approximate solution and a true solution also gives error bounds on every geometric quantity computed by SnapPy. It follows that the edge $e_0$ (marked $0$ in Figure \ref{fig:3227-triangulation}) is the unique shortest geodesic between $K$ and $T_0$, because the next shortest edge is significantly longer. Similarly, the edge $e_1$ (marked $1$ in Figure \ref{fig:3227-triangulation}) is the unique shortest geodesic between $K$ and $T_1$, proving \eqref{item:shortest}.

Finally, SnapPy computes the action of $\alpha_0$ and $\beta_0$ on the complex plane covering the maximal cusp torus about $T_0$.  The computed values are as in \eqref{eq:parabolic0}. Moser's error bound on the gluing equations, expressed in \eqref{eq:moser}, implies that the error in equation  \eqref{eq:parabolic0} is significantly less than $0.01$. Similarly, the lengths of  $\alpha_1$ and $\beta_1$, as computed in  \eqref{eq:parabolic1}, are accurate to an error much less than $0.01$.
\end{proof}

Next, we turn our attention to Heegaard splittings of  $X$.

\begin{lemma}  \label{lemma:example--2}
The arcs $\sigma_0$ and $\sigma_1$, depicted in Figure \ref{fig:arcs}, are core tunnels for genus--$2$ Heegaard splittings of the manifold $X = (T^2 \times I) \setminus K$.
\end{lemma}

\begin{figure}
 	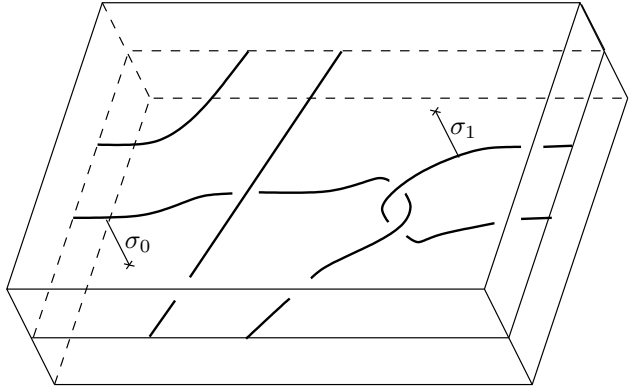
 	\caption{The arcs $\sigma_0$ and $\sigma_1$ are core tunnels for  $X = (T^2 \times I) \setminus K$.}
  \label{fig:arcs}
\end{figure}

  \begin{figure}
	  \includegraphics{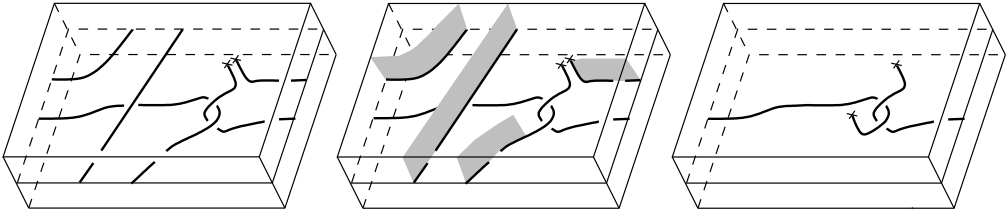}
	  \caption{Manifolds homeomorphic to $X \setminus     \sigma_1$.}
    \label{fig:split-arc}
  \end{figure}

\begin{proof}
  We prove this for the arc $\sigma_1$, as the proof for $\sigma_0$ is similar.
Let $V \subset X$ be a regular neighborhood of $T_1 \cup K \cup \sigma_1$. Then $\overline{V}$ is obtained by attaching a $1$--handle to  the cusp tori of $T_1$ and $K$, along arc $\sigma_1$. Thus $\overline{V}$ is a compression body whose positive boundary is a genus--$2$ surface $S_1 = \bdy_+ V$, and whose core tunnel is $\sigma_1$. 

It remains to show that $W = X \setminus V$ is also a compression body. To see this, note that $X \setminus V$ is homeomorphic to $X \setminus \sigma_1 = (T^2 \times I) \setminus (K \cup \sigma_1)$. Furthermore, the complement of $K \cup \sigma_1$ in $T^2 \times I$ is
  homeomorphic to the manifold shown on the left of Figure
  \ref{fig:split-arc}, where  we now are taking the complement of an arc with
  endpoints on the top boundary of $T^2 \times I$.

  Drag one of these endpoints along the torus $T_1 = T^2 \times \{ 1 \}$, following the
  disk shown in the middle part of Figure
  \ref{fig:split-arc}.  The result is shown
  on the right of that figure.  Now untwist.  The result is a manifold
  homeomorphic to $T^2\times I$, with a simple boundary--parallel arc removed.  Thus $W = X \setminus V$ is a genus--2 compression  body.
\end{proof}

In fact, $\sigma_0$ and $\sigma_1$ lead to the only two genus--$2$ Heegaard surfaces of $X$.

\begin{lemma}  \label{lemma:exactly2}
  The manifold $X$ of Figure \ref{fig:arcs} has exactly
  two genus--2 Heegaard surfaces $S_0$ and $S_1$, associated to tunnels $\sigma_0$ and $\sigma_1$.
\end{lemma}

\begin{proof}Let $S \subset X$ be a genus--$2$ Heegaard surface. Since $X$ is hyperbolic, $S$ must be strongly irreducible (every compression disk on one side of $S$ intersects every compression disk on the other side of $S$). Thus, by a theorem of Rubinstein \cite{rubinstein} and Stocking \cite{stocking}, $S$ must be isotopic to an \emph{almost normal} surface in the triangulation of Figure \ref{fig:3227-triangulation}. Recall that a \emph{normal} surface intersects every tetrahedron in a disjoint union of normal triangles and quadrilaterals, as in the left two panels of Figure \ref{fig:almost-normal}. An \emph{almost normal} surface is composed of triangles and quads, as well as exactly one octagon or tube, as in the right two panels of Figure \ref{fig:almost-normal}.

The program Regina \cite{burton:regina} can perform a rigorous combinatorial analysis of normal and almost normal surfaces in $X$. In particular, Regina verifies that $X$ contains no almost normal genus--$2$ surfaces with an octagon. Thus $S$ must contain a tube.

Because $S$ has genus $2$, compressing $S$ along the tube will produce one or two normal tori. But since the triangulation of Figure \ref{fig:3227-triangulation} supports a positively oriented solution to the gluing equations (by Lemma \ref{lemma:census-knot}), any normal torus must be boundary--parallel, composed entirely of vertex--linking triangles. (See e.g.\ Lackenby \cite[Proposition 4.4]{lackenby:surgery}.) Thus the almost normal tube in $S$ is obtained by tubing together two non-parallel triangles in one tetrahedron.

Any tube between two normal triangles is isotopic to the neighborhood of an edge in the triangulation. Thus, since there are seven edges in Figure \ref{fig:3227-triangulation}, we must consider seven tubed surfaces $S_0, \ldots, S_6$, in one-to-one correspondence with edges $e_0, \ldots, e_6$. For each $S_i$, let $V_i$ be the closure of the component of $X \setminus S_i$ that contains the $1$--handle through the tube. Since $V_i$ is obtained by joining together one or two cusp tori along a $1$--handle, it is a compression body.

For each almost normal tubed surface $S_i$, we begin isotoping $S_i$ toward its associated edge $e_i$, according to the tightening algorithm of Schleimer \cite[Sections 7--9]{schleimer:sphere-np}. In the case at hand, Schleimer's tightening procedure constructs an embedded isotopy between each $S_i$ and a normal surface in the triangulation.

The tubed surfaces $S_0$ and  $S_2$ tighten to the same normal surface, from opposite sides. Thus these surfaces are isotopic, and furthermore $V_0 \cong \overline{X \setminus V_2}$. Thus $S_0$ splits $X$ into compression bodies $V_0$ and $V_2$. Similarly, $S_1$ and $S_3$ tighten to the same normal surface, from opposite sides. Thus $S_1$ splits $X$ into compression bodies $V_1$ and $V_3$.

\begin{figure}
\includegraphics[width=1.2in]{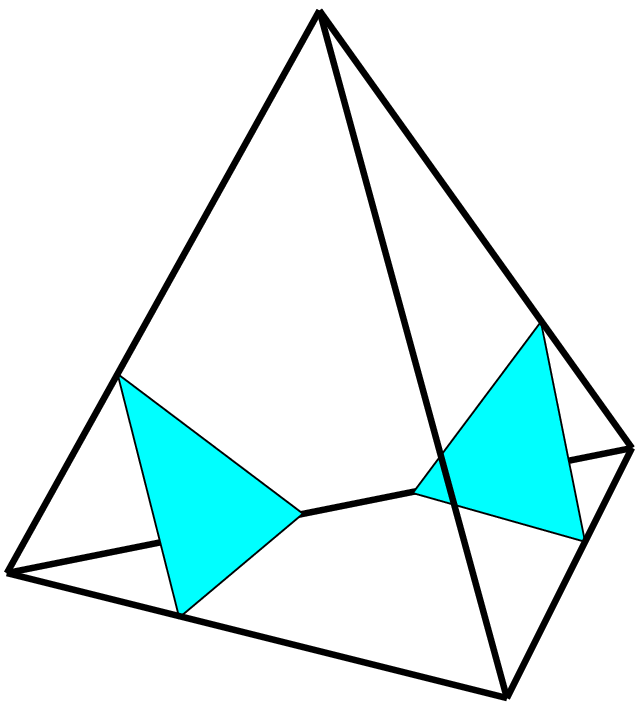} \hspace{0.2in} \includegraphics[width=1.2in]{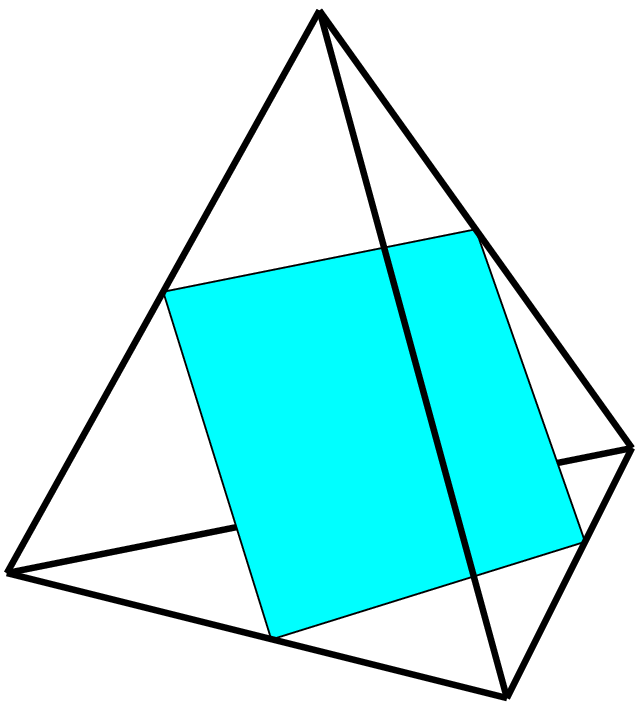} \hspace{0.2in}
\includegraphics[width=1.2in]{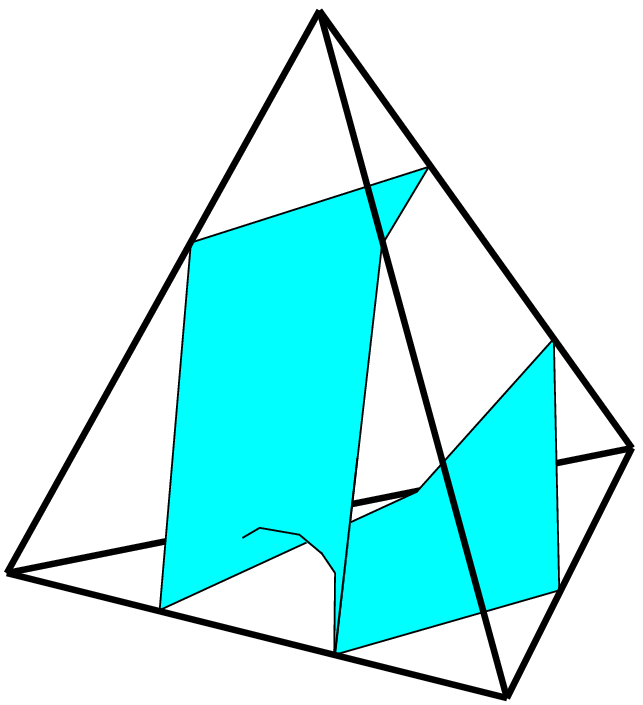} \hspace{0.2in} \includegraphics[width=1.2in]{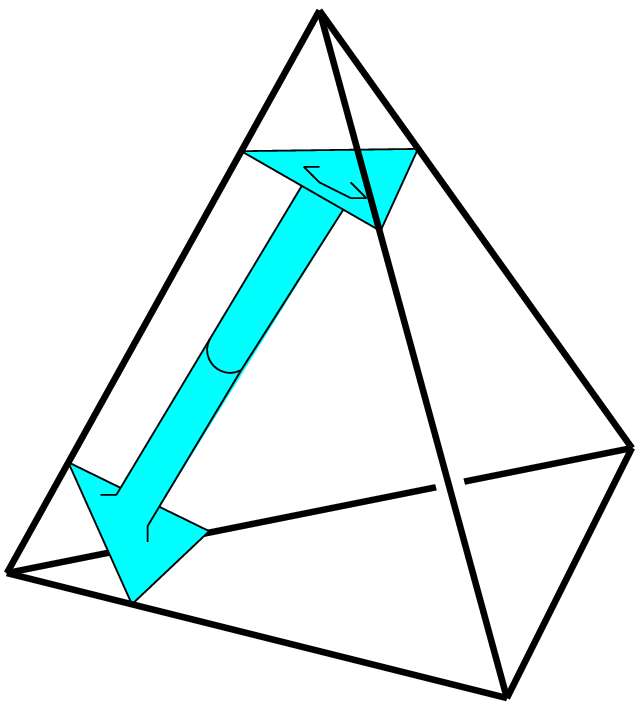}
\caption{Left to right: Normal triangles. A normal quadrilateral. An almost normal octagon. An almost normal tube between two vertex-linking triangles. (Graphics by Saul Schleimer.)}
\label{fig:almost-normal}
\end{figure}

It remains to check that $S_i$ is not a Heegaard surface for $i = 4,5,6$. This follows from a theorem of Bachman   \cite{bachman:normalizing}. He proves that if a normal surface $T \subset X$ is a Heegaard surface, then there must be two almost normal surfaces isotopic to $X$. (For example, this is the situation for $S_0$ and $S_2$.) In our case, each of $S_4$, $S_5$, and $S_6$ separates the cusp of $K$ from the cusps $T_0$ and $T_1$ (see edges $4,5,6$ in Figure \ref{fig:3227-triangulation}). Thus, if two of these surfaces $S_i$ and $S_j$ were isotopic, then the compression bodies $V_i$ and $V_j$ on the side of $T_0$ and $T_1$ would also be isotopic. But the core tunnels of these compression bodies, namely edges $e_i$ and $e_j$, are distinct edges of a geometric triangulation, hence non-isotopic. This is a contradiction.

Alternately, one may show that $S_i$ is not a Heegaard surface for $i = 4,5,6$ by using Regina to  cut $X$ along the normal surface isotopic to $S_i$. Regina can retriangulate the resulting pieces. Then, it verifies that the piece $\overline{X \setminus V_i}$ is not a compression body, by showing that it has incompressible boundary. (If the boundary was compressible, there would have to be a normal compression disk.) Since cutting and retriangulating greatly increases the number of tetrahedra, this verification took six hours of runtime.

We conclude that $S_0$ is the only Heegaard surface of $X$ that separates $T_1$ from $T_0 \cup K$. Since the core tunnel $\sigma_0$ connects $T_0$ to $K$, it must be a core tunnel for $S_0$, isotopic to edge $e_0$ from $T_0 \cup K$.  Similarly, since $\sigma_1$ connects $T_1$ to $K$, it must be a core tunnel for $S_1$, isotopic to edge $e_1$ in the triangulation.
\end{proof}

As a corollary of the above argument, we obtain

\begin{lemma}\label{lemma:sigma-shortest}
For $i=0,1$, the core tunnel $\sigma_i$ of Figure \ref{fig:arcs} is isotopic to the unique shortest geodesic in $X$ between the cusps of $K$ and $T_i$.
\end{lemma}

\begin{proof}
The analysis of almost normal surfaces in Lemma \ref{lemma:exactly2} shows that $\sigma_i$ is isotopic to edge $e_i$, labeled $i$ in the triangulation of Figure \ref{fig:3227-triangulation}. By Lemma \ref{lemma:census-knot}, edge $e_i$ is the unique shortest geodesic between  $K$ and $T_i$.
\end{proof}

\subsection{Fibonacci slopes}
As in Section \ref{sec:knots}, we will construct a sequence of knots in $S^3$ by filling cusp tori  $T_0 = T^2 \times \{0\} $ and $T_1 =T^2 \times \{1\}$. We choose the filling slopes as follows.

\begin{define}\label{def:fibonacci-slopes}
For every integer $n \geq 0$, let $f_n$ be the $n$-th Fibonacci number (with $f_0 = 0$ and $f_1 = 1$). Then, for each $n \geq 1$, define the filling slopes
\begin{equation}\label{eq:explicit-filling}
\mu_0^n = f_n \, \alpha_0 + f_{n+1} \, \beta_0, \qquad \mu_1^n = f_{n-1} \,  \alpha_1 + f_{n} \, \beta_1,
\end{equation}
where $\alpha_i$ and $\beta_i$ are as in Figure \ref{fig:knot-diagram}.
\end{define}

\begin{lemma}\label{lemma:explicit-knots}
The slopes $\mu_0^n$ and $\mu_1^n$ have intersection number $\iota(\mu_0^n, \, \mu_1^n) = 1$. Therefore, by Lemma \ref{lemma:s3-filling}, filling cusp $T_0$ along $\mu_0^n$ and cusp $T_1$ along $\mu_1^n$ produces a knot $K_n \in S^3$. 
\end{lemma}

\begin{proof}
It is well--known that Fibonacci numbers can be produced by the explicit formula
\begin{equation}\label{eq:fibonacci-matrix}
\begin{bmatrix} f_{n-1} & f_n \\ f_n & f_{n+1} \end{bmatrix} = A^n,
\quad \mbox{where} \quad
A = \begin{bmatrix} 0 & 1 \\ 1 & 1 \end{bmatrix}.
\end{equation}
Since the left column of $A^n$ expresses the slope $\mu_1^n$, and the right column expresses the slope $\mu_0^n$, we have 
$$\iota(\mu_0^n, \, \mu_1^n) \: = \: \abs{ \det (A^n) } \: = \:  \abs{ \det A }^n \: = \: 1.$$ 
Thus, by Lemma \ref{lemma:s3-filling}, filling cusp $T_i$ along $\mu_i^n$ produces a knot $K_n \in S^3$. 

In fact, recalling the proof of Lemma  \ref{lemma:s3-filling} allows a concrete way to visualize the knot $K_n$. If we embed the slab $T^2 \times I$ in the complement of the Hopf link in $S^3$ via the mapping class
$$\begin{bmatrix} f_{n-1} &  (-1)^n f_n \\ f_n & (-1)^n f_{n+1} \end{bmatrix} \in SL(2, \ZZ),$$
the curves $\mu_0^n$ and $\mu_1^n$ will be mapped to the meridians of the two Hopf components. Then, filling the two Hopf components along their meridians in $S^3$ (i.e., erasing these two link components from the diagram) will leave the knot $K_n \in S^3$. 
\end{proof}

\begin{lemma}\label{lemma:fibonacci-lengths}
Let $n \geq 5$. Then, on a maximal horocusp about $T_i \subset X$, the slope $\mu_i^n$ of Definition \ref{def:fibonacci-slopes} has length
\begin{equation}\label{eq:fibonacci-lengths}
4.3 \, \varphi^n \: < \:  \ell(\mu_i^n) \: < \:  4.7 \, \varphi^n ,
\end{equation}
where $\varphi = \frac{1 + \sqrt{5}}{2}$ is the golden ratio. Furthermore, the shortest longitude for $\mu_i^n$
is $\lambda_i^n = \mu_i^{n-2}$.
\end{lemma}

\begin{proof}
The Fibonacci number $f_n$ has the closed form expression
\begin{equation}\label{eq:fibonacci-closed}
f_n = \frac{\varphi^n - (-\varphi)^{-n}}{\sqrt{5}} \, ,
\end{equation}
which can be derived by diagonalizing the matrix $A$ of equation \eqref{eq:fibonacci-matrix}. Thus the slope $\mu_0^n$ of Definition \ref{def:fibonacci-slopes} can be written as
\begin{eqnarray}
\mu_0^n 
& = & \left(  \frac{\varphi^n - (-\varphi)^{-n}}{\sqrt{5}} \right) \alpha_0 + \left(  \frac{\varphi^{n+1} - (-\varphi)^{-(n+1)}}{\sqrt{5}} \right) \beta_0
\notag \\
& = & \left(  \frac{\varphi^n }{\sqrt{5}} \right) (\alpha_0 + \varphi \beta_0) + \left(   \frac{(-\varphi)^{-(n+1)}}{\sqrt{5}} \right) (\varphi \alpha_0 - \beta_0).
\label{eq:m0-eigenvalue}
\end{eqnarray}
Note that when $n \geq 3$, we have
\begin{equation}\label{eq:phi-power}
\varphi^{n} \: \geq \: \varphi^7 \cdot \varphi^{-(n+1)} \: > \: 29 \, \varphi^{-(n+1)}.
\end{equation}
Thus, substituting \eqref{eq:phi-power} and the translation lengths of \eqref{eq:parabolic0} into equation \eqref{eq:m0-eigenvalue} gives
$$
\renewcommand{\arraystretch}{2.5}
\begin{array}{r c c c l}
 \dfrac{\varphi^n }{\sqrt{5}} \:  \ell (\alpha_0 + \varphi \beta_0) -   \dfrac{\varphi^n }{\sqrt{5}} \:  \dfrac{\ell (\varphi \alpha_0 - \beta_0)}{29}
& < &  \ell(\mu_0^n) & < & 
 \dfrac{\varphi^n }{\sqrt{5}} \: \ell (\alpha_0 + \varphi \beta_0) + \dfrac{\varphi^n }{\sqrt{5}} \:  \dfrac{\ell (\varphi \alpha_0 - \beta_0)}{29} \\
 \dfrac{\varphi^n }{\sqrt{5}} \: (10.1)  -  \dfrac{\varphi^n }{\sqrt{5}} \: \left( \dfrac{2.7}{29} \right)
& < &  \ell(\mu_0^n) & < & 
 \dfrac{\varphi^n }{\sqrt{5}} \:(10.2) +   \dfrac{\varphi^n }{\sqrt{5}} \: \left( \dfrac{2.7}{29} \right) \\
4.4\, \varphi^n
& < &  \ell(\mu_0^n) & < & 
4.7 \, \varphi^n .
 \end{array}
$$

An identical calculation, using
\begin{equation}
\label{eq:m1-eigenvalue}
\mu_1^n
\: = \:  \left(  \frac{\varphi^{n-1} }{\sqrt{5}} \right) (\alpha_1 + \varphi \beta_1) + \left(   \frac{(-\varphi)^{-n}}{\sqrt{5}} \right) (\varphi \alpha_1 - \beta_1),
\end{equation}
the estimate \eqref{eq:phi-power}, and the translation lengths of \eqref{eq:parabolic1} gives
$$
4.3\, \varphi^n
\: < \:  \ell(\mu_1^n) \: < \: 
4.6 \, \varphi^n .
$$
Thus both $\mu_0^n$ and $\mu_1^n$ satisfy the estimates of \eqref{eq:fibonacci-lengths}.

Now, consider the shortest longitude $\lambda_i^n$ of $\mu_i^n$. Since $\mu_1^n$ is the same slope on $T^2$ as $\mu_0^{n-1}$, Lemma \ref{lemma:explicit-knots} implies that $\iota(\mu_0^{n-1}, \mu_0^n) = 1$. Thus, by definition, $\mu_0^{n-1}$ is a longitude for $\mu_0^n$. Furthermore, every longitude for $\mu_0^n$ must have the form
$$\mu_0^{n-1} + k \mu_0^n, \quad \mbox{where} \quad k \in \ZZ.$$
Now, observe from equations \eqref{eq:m0-eigenvalue} and \eqref{eq:phi-power} that both $\mu_0^{n-1}$ and $\mu_0^n$ are nearly parallel to the vector $(\alpha_0 + \varphi \, \beta_0)$, hence nearly parallel to one another. Furthermore, by \eqref{eq:fibonacci-lengths}, the lengths of $\mu_0^{n-1}$ and $\mu_0^n$ differ by a ratio of at most $\frac{4.7}{4.3}\varphi < 2$. Thus the shortest longitude of the form $\mu_0^{n-1} + k \mu_0^n$ will have $k=-1$, and the shortest longitude is the slope
\begin{eqnarray*}
\lambda_0^n & = & \mu_0^{n-1} - \mu_0^n \\
& = &  (f_{n-1} \,  \alpha_0 + f_{n} \, \beta_0) - (f_n \, \alpha_0 + f_{n+1} \, \beta_0) \\
& = & -(f_{n-2} \,  \alpha_0 + f_{n-1} \, \beta_0) \\
& = & -\mu_0^{n-2}.
\end{eqnarray*}
(The minus sign is immaterial, since we do not need an orientation on $\lambda_0^n$.)
By an identical argument, the shortest longitude for $\mu_1^n$
is $\lambda_1^n = \mu_1^{n-2}$.
\end{proof}

\begin{remark}
Although Fibonacci numbers are convenient for the proof of Lemma \ref{lemma:fibonacci-lengths}, in fact the construction has many generalizations. The key fact that makes the entire argument work is that the sequence of slopes comes from powers of a pseudo--Anosov matrix $A$, as in equation \eqref{eq:fibonacci-matrix}. If we had defined a sequence of slopes $\mu^n$ by iterating some other pseudo--Anosov matrix $B \in SL(2,\ZZ)$, then it would follow, in analogy to equation \eqref{eq:m0-eigenvalue}, that for sufficiently large $n$ the slopes $\mu^n$ are nearly parallel to the stable foliation of $B$. The lengths $\ell(\mu^n)$ of slopes defined in this way would be controlled by the dilatation (or largest eigenvalue) of $B$, and the shortest longitude of $\mu^n$ would be immediately visible. 

In this way, ideas from hyperbolic geometry play a significant role in
a combinatorial construction on the torus.
\end{remark}

\subsection{Completing the argument}

We may now Dehn fill $X$, along slope $\mu_i^n$ on cusp $T_i$, to produce a knot $K_n \subset S^3$. The following result about the unknotting tunnels of $K_n$ immediately implies Theorem \ref{thm:knot-long-tunnel} in the introduction.

\begin{theorem}\label{thm:explicit-long-tunnel}
Let $\mu_0^n$ and $\mu_1^n$ be the Dehn filling slopes of Definition \ref{def:fibonacci-slopes}. Then, for sufficiently large $n$, the knot $K_n \subset S^3$ obtained by filing $T_i \subset X$ along $\mu_i^n$ has exactly two unknotting tunnels $\tau_0^n$ and $\tau_1^n$. Both tunnels are isotopic to canonical geodesics in $S^3 \setminus K_n$, and both tunnels have length satisfying
$$2 n \ln (\varphi ) - 4.8 \: < \: \ell (\tau_n) \: < \: 2 n  \ln (\varphi ) +  5.9.$$
where $ \varphi = \tfrac{1 + \sqrt{5}}{2}$ is the golden ratio.
\end{theorem}

\begin{proof} The proof parallels the proof of Theorem \ref{thm:long-s3}, with more explicit estimates.
As in that proof, let $J = 1.1$ and $\epsilon = 0.29$. Then, by Lemma \ref{lemma:fibonacci-lengths}, choosing a sufficiently large $n$ ensures the following:

\begin{enumerate}[(A)]
\item\label{item:drilling-ex} For $n \gg 0$, the normalized lengths
  $L(\mu_i^n)$ are $J^2$ times longer than necessary to satisfy the
  Drilling Theorem \ref{drilling-thm}, for $J=1.1$ and
  $\epsilon=0.29$.  This way, we can apply Theorem
  \ref{drilling-thm} twice: once to fill $T_0$ and a second time to
  fill $T_1$ (or in the opposite order).
\item\label{item:apply-length-ex} For $n \geq 8$,  the length of each $\mu_i^n$ (on its respective horospherical  torus) satisfies 
$$\ell(\mu_i^n) > 4.3 \, \varphi^8 > 152.$$
\item\label{item:longitude-ex} For $n \geq 5$, the shortest longitude $\lambda_i^n = \mu_i^{n-2}$ satisfies $\ell(\lambda_i) > 4.3 \, \varphi^3 > 7$.
\end{enumerate}

Applying Theorem \ref{thm:heegaard-correspondence} twice (once to fill $T_0$, and again to fill $T_1$), we conclude that the
resulting knot complement $S^3 \setminus K_n$ has Heegaard genus $2$, with any genus--$2$ Heegaard surfaces coming from the Heegaard surfaces of the
original manifold $X$.  Recall that by Lemma \ref{lemma:exactly2}, there are
exactly two Heegaard surfaces in $X$ (with
core tunnels $\sigma_0$ running from $K$ to
$T_0$, and $\sigma_1$ running from $K$ to $T_1$).

When $n$ is sufficiently large and $\mu_0^n$ is sufficiently long, filling torus $T_0$ along $\mu_0^n$ will preserve the property (from Lemma \ref{lemma:sigma-shortest}) that  $\sigma_1$ is the shortest geodesic from $K$ to $T_1$. Note that in the two-cusped manifold $X(\mu_0^n)$, the arc $\sigma_1$ is a core tunnel of a compression body $V_1$, and the complement of $V_1$ is a genus--$2$ handlebody. Thus $\sigma_1$ is an unknotting tunnel of $X(\mu_0^n)$, and is the shortest arc between the two cusps of $X(\mu_0^n)$. Thus, by Theorem \ref{thm:canonicity}, the unknotting tunnel $\tau_1^n$ of $K_n$ that is associated to $\sigma_1$ must be a canonical geodesic in $S^3 \setminus K_n$. (The filling of $T_1$ along $\mu_1^n$ will be generic because  choosing $n$ large ensures that both $\mu_1^n$ and $\lambda_1^n$ are arbitrarily long.)

In a similar way, if choose $n$ sufficiently large and start by filling $T_1$ along $\mu_1^n$, the unknotting tunnel $\sigma_0$ of $X(\mu_1^n)$ will be the shortest arc between the two cusps of $X(\mu_0^n)$. Thus, by Theorem \ref{thm:canonicity}, the unknotting tunnel $\tau_0^n$ of $K_n$ that is associated to $\sigma_0$ must be a canonical geodesic in $S^3 \setminus K_n$. 

Note that by Lemma \ref{lemma:triangle-in-tube}, the large majority of the length of $\tau_0^n$ is contained in the Margulis tube created by filling $T_0$. Similarly, the large majority of the length of $\tau_1^n$ is contained in the Margulis tube created by filling $T_1$. Thus the two unknotting tunnels $\tau_0^n$ and $\tau_1^n$ are distinct, and account for the only two genus--$2$ Heegaard splittings of $S^3 \setminus K_n$.

It remains to estimate the lengths of $\tau_0^n$ and $\tau_1^n$.  First, fill $X$ along slope $\mu_0^n$ on $T_0$.
Then, by Theorem \ref{drilling-thm}\eqref{item:level-preserve}, there is a $J$--bilipschitz diffeomorphism $\phi$ that maps a maximal horocusp about $T_1$ in $X$  to a self-tangent (hence, maximal) horocusp about $T_1$ in $X(\mu_0)$. Applying Lemma \ref{lemma:fibonacci-lengths} with error bounded by $J=1.1$, we conclude that in $X(\mu_0^n)$, the meridian $\phi(\mu_1^n)$ and its shortest longitude $\varphi(\lambda_1^n)$ have lengths satisfying
\begin{equation}\label{eq:filled-once}
\ell(\mu_1^n ) \: > \: 152/ 1.1 \: > \: 138 \qquad \mbox{and} \qquad  \tfrac{4.3}{1.1} \, \varphi^{n-2} \: < \: \ell(\lambda_1^n) \: < \: 4.7 \cdot 1.1 \, \varphi^{n-2}.
\end{equation}

Thus, by equation \eqref{eq:filled-once} and condition \eqref{item:drilling-ex}, Dehn filling $X(\mu_0^n)$ along slope $\mu_1^n$ on $T_1$ will satisfy all the hypotheses of Theorem \ref{thm:length-estimate-quant}. Therefore, by Theorem \ref{thm:length-estimate-quant}, the unknotting  tunnel $\tau_1^n$ associated to $\sigma_1$ will have length satisfying 
$$
\renewcommand{\arraystretch}{1.5}
\begin{array}{r c c c l}
2 \ln \left( \tfrac{4.3}{1.1} \, \varphi^{n-2} \right) - 5.6 
& < & \ell(\tau_1^n) & < & 
2 \ln \left( 4.7 \cdot 1.1 \, \varphi^{n-2} \right) + 4.5 \\
2n \, \ln(\varphi) + 2 \ln \left( \tfrac{4.3}{1.1} \, \varphi^{-2} \right) - 5.6 
& < & \ell(\tau_1^n) & < & 
2n \, \ln(\varphi) + 2 \ln \left( 4.7 \cdot 1.1 \, \varphi^{-2} \right) + 4.5 \\
2n \, \ln(\varphi) -4.8 
& < & \ell(\tau_1^n) & < & 
2n \, \ln(\varphi) + 5.9.
\end{array}
$$
By the same argument, reversing the order of the fillings, the other unknotting tunnel $\tau_0^n$ of $S^3 \setminus K_n$ satisfies the same estimates on length.
\end{proof}

\appendix

\section{Variations on the law of cosines}\label{sec:cosines}

The goal of this appendix is to write down versions of the law of cosines that work for triangles with a mix of material and ideal vertices. The formulae of Lemmas \ref{lemma:13-ideal} and \ref{lemma:23-ideal} can likely be derived from  the extensive tables compiled by Guo and Luo \cite[Appendix]{guo-luo:polyhedral-rigidity2}.
We prefer to begin with the law of cosines for ordinary triangles in $\HH^2$.

\begin{lemma}[Law of cosines]\label{lemma:material}
Let $\Delta$ be a triangle in $\HH^2$, with sidelengths $a,b,c$. Let $\alpha$ be the angle opposite side $a$. Then
\begin{equation}\label{eq:law-of-cosines}
\cosh a = \cosh b \,  \cosh c - \sinh a \, \sinh b \, \cos \alpha.
\end{equation}
\end{lemma}

\begin{proof}
See \cite[Equation 2.4.9]{thurston:book} or \cite[Section VI.3.5]{fenchel:book}.
\end{proof}

When $\Delta$ has an ideal vertex, the natural analogue of an angle $\alpha$ is the length of a horocycle truncating this ideal vertex. As a result, the law of cosines takes the following form. (We prove Lemma \ref{lemma:13-ideal} for isosceles triangles, but the general case is not much harder.)

\begin{figure}[h]
\input{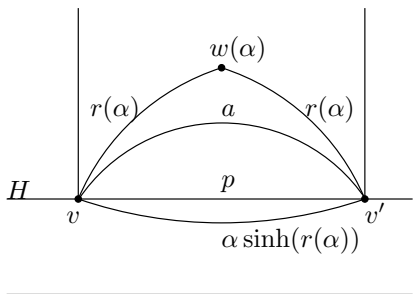}
\caption{Lengths on $1/3$--ideal triangle.}
\label{fig:13-ideal}
\end{figure}

\begin{lemma}\label{lemma:13-ideal}
Let $\Delta$ be a $1/3$ ideal triangle in $\HH^2$. Let $H$ be a horoball truncating the ideal vertex of $\Delta$, and suppose that the two material vertices $v, v'$ are equidistant from $H$. Let $a$ denote the sidelength opposite the ideal vertex, let $b$ denote the distance from $H$ to $v$ or $v'$, and let $p$ denote the length of the horocyclical segment $\bdy H \cap \Delta$. Then
\begin{equation}\label{eq:13-ideal}
2 \sinh (a/2) = p \, e^b .
\end{equation}
\end{lemma}

\begin{proof}
First, observe that the quantity $p \, e^b$ on the right--hand side of equation \eqref{eq:13-ideal} is independent of the choice of horoball $H$. This is because increasing the size of $H$ by hyperbolic distance $d$ will increase the horocycle $p$ by a factor of $e^d$ while subtracting $d$ from the sidelength $b$. Thus the right-hand side of  \eqref{eq:13-ideal} remains unchanged. As a result, no generality is lost in assuming that $b=0$, or equivalently that vertices $v,v' \in \bdy H$.

For every angle $\alpha \in (0, \pi)$, let $\Delta_\alpha$ be an
isosceles triangle in $\HH^2$ that has vertices at $v$ and $v'$, and a
third vertex $w(\alpha)$ with angle $\alpha$. Then the points $v$ and
$v'$ lie on the same circle centered at $w(\alpha)$, whose radius we
denote by $r(\alpha)$. The length of the circle arc from $v$ to $v'$
is $\alpha \sinh r(\alpha)$. See Figure \ref{fig:13-ideal}.

By the law of cosines \eqref{eq:law-of-cosines}, the sidelength $a$ of $\Delta_\alpha$ satisfies
\begin{eqnarray}\label{eq:13-circle}
\cosh a &=& \cosh^2 r(\alpha) - \sinh^2 r(\alpha) \, \cos \alpha  \notag \\
&=& 1 + \sinh^2 r(\alpha) (1- \cos \alpha) \notag \\
&=& 1 + \big( \alpha \sinh r(\alpha) \big)^2  \left( \frac{1- \cos \alpha}{\alpha^2} \right) \notag \\
&=& 1 + \left(    \lim_{\alpha \to 0} \,  \alpha \sinh r(\alpha)    \right)^2   \left( \lim_{\alpha \to 0} \frac{1- \cos \alpha}{\alpha^2} \right) \notag  \\
&=& 1 + p^2 \cdot \tfrac{1}{2} \, .
\end{eqnarray}
The last equality holds because as $\alpha \to 0$,  triangle $\Delta_\alpha$ converges to $\Delta$, and the circle of radius $r(\alpha)$ converges to the horocycle $\bdy H$ through $v$ and $v'$. Thus the circle arc of length $\alpha \sinh r(\alpha)$ converges to the horocyclical segment of length $p$. 
Now, solving equation \eqref{eq:13-circle} for $p$ and recalling that $b=0$ produces \eqref{eq:13-ideal}, as desired.
\end{proof}

For $2/3$ ideal triangles, we have the following version of the law of cosines.

\begin{lemma}\label{lemma:23-ideal}
Let $\Delta$ be a $2/3$ ideal triangle in $\HH^2$. Suppose that the two sides of $\Delta$ meeting at the material vertex $v$ are labeled $\rho$ and $\rho'$, and the angle at $v$ is $\alpha$. Suppose the third side of $\Delta$ is labeled $g$. Choose horoball neighborhoods $H$ and $H'$ about the two ideal vertices. Then, relative to the horoballs $H$ and $H'$,
\begin{equation}\label{eq:23-ideal}
\ell(g) = \ell(\rho) + \ell(\rho') + \ln \left(  \frac{1 - \cos \alpha}{2}    \right).
\end{equation}
\end{lemma}

\begin{proof}
As in the last proof, we begin by observing that the quantity
\begin{equation}\label{eqn:23-ideal}
\ell(g) - (\ell(\rho) + \ell(\rho'))
\end{equation}
is independent of the choice of horoball neighborhoods. This is because expanding horoball $H$ will decrease both $\ell(g)$ and $\ell(\rho)$ by the same amount. Similarly, adjusting the size of $H'$ will affect $\ell(g)$ and $\ell(\rho')$ by the same amount. Thus, no generality is lost in assuming that $\ell(\rho) = \ell(\rho')$. Given this assumption, we will show that the quantity \eqref{eqn:23-ideal} is equal to to $\ln \, (1 - \cos \alpha)/2$.

\begin{figure}
  \input{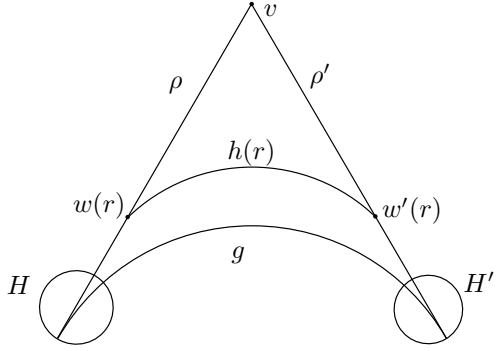}
  \caption{The setup of Lemma \ref{lemma:23-ideal}.}
  \label{fig:cosh-limit}
\end{figure}

For each $r >0$, let $w(r)$ be the point on $\rho$ that is distance $r$ from $v$, and let $w'(r)$ be the corresponding point on $\rho'$. See Figure \ref{fig:cosh-limit}. By Lemma \ref{lemma:material}, the distance between $w(r)$ and $w'(r)$ is
$$h(r) = \cosh^{-1} f(r), \quad \mbox{where} \quad f(r)\:  =  \: \cosh^2 r - \sinh^2 r \cos \alpha.$$
Now, as $r \to \infty$, the geodesic between $w(r)$ and $w'(r)$ approaches $g$. In particular, this geodesic fellow--travels $\rho$ and $\rho'$ for a greater and greater portion of its length. Thus $h(r) - 2r$ becomes a better and better approximation to the quantity $\ell(g) - 2 \ell(\rho)$. Therefore,  

\begin{eqnarray*}
\ell(g) - 2 \ell(\rho)
&=& \lim_{r \to \infty} \, h(r) - 2r\\
&=& \lim_{r \to \infty} \, \cosh^{-1} \left( f(r) \right) - 2r\\
&=& \lim_{r \to \infty} \, \ln \left( f(r) + \sqrt{f(r)^2 - 1} \right) - 2r \\
&=& \lim_{r \to \infty} \,  \ln \left(2 f(r)  \right) - 2r \\
&=& \lim_{r \to \infty} \,  \ln \left( 2 \cosh^2 r - 2 \sinh^2 r \cos \alpha  \right) - 2r \\
&=& \lim_{r \to \infty} \,  \ln \left( \frac{e^{2r}}{2} - \frac{e^{2r}}{2} \cos \alpha  \right) - 2r \\
&=& \lim_{r \to \infty} \, 2r +  \ln \left( \frac{1 - \cos \alpha}{2}   \right) - 2r \\
&=& \ln  \left( \frac{1 - \cos \alpha}{2} \right)   . 
\end{eqnarray*}

\vspace{-4ex}
\end{proof}

%
%
%

\bibliographystyle{hamsplain}
\bibliography{biblio}

\end{document}